\documentclass{article}

\usepackage{amsmath,amsthm,amssymb,mathtools}  
\usepackage{dsfont}			
\usepackage{thmtools}		
\usepackage{enumerate}
\usepackage{graphicx}		
\usepackage{subcaption}		
\usepackage[font=small]{caption}	
\usepackage{cite} 			
\usepackage[latin1]{inputenc} 

\usepackage{latexsym}	

\usepackage{makecell}			
\usepackage{booktabs}			
\renewcommand{\arraystretch}{2}	

\usepackage{stackengine}	
\stackMath					

\usepackage[hyperindex,breaklinks
	]{hyperref}		
\hypersetup{colorlinks=true, linkcolor=blue, citecolor=blue}

	\usepackage{verbatim}		
		
	\usepackage[
	disable,		
	textsize=scriptsize,textwidth=4.2cm]{todonotes}		
	\newcommand{\SELF}[1]{\todo[color=green!40]{#1}} 
	\newcommand{\OMIT}[1]{\todo[color=gray!30]{#1}}  
	\newcommand{\CITE}[1]{\todo[color=cyan!30]{#1}}  
	
	\newcommand{\SELFL}[1]{\reversemarginpar\todo[color=green!40]{#1}} 
	\newcommand{\OMITL}[1]{\reversemarginpar\todo[color=gray!40]{#1}}  
	\newcommand{\CITEL}[1]{\reversemarginpar\todo[color=cyan!30]{#1}} 
	
	\newcommand{\SELFR}[1]{\normalmarginpar\todo[color=green!40]{#1}} 
	\newcommand{\OMITR}[1]{\normalmarginpar\todo[color=gray!40]{#1}}  
	\newcommand{\CITER}[1]{\normalmarginpar\todo[color=cyan!30]{#1}}

\usepackage[nameinlink]{cleveref} 	

\usepackage{bm}					
\usepackage{cmll}				

\usepackage{makecell}			
\usepackage{booktabs}			
\renewcommand{\arraystretch}{2}	

\usepackage[amsstyle]{cases}	

\newtheorem{theorem}{Theorem}[section]
\newtheorem{proposition}[theorem]{Proposition}
\newtheorem{corollary}[theorem]{Corollary}
\newtheorem{lemma}[theorem]{Lemma}

\theoremstyle{definition}
\newtheorem{definition}[theorem]{Definition} 
\newtheorem{example}[theorem]{Example}

\declaretheoremstyle[
spaceabove=6pt, spacebelow=6pt,
headfont=\normalfont\bfseries,
notefont=\normalfont\bfseries, 
notebraces={}{},
bodyfont=\normalfont\itshape
]{Estilo1}

\newcommand\inner[1] 		{\langle #1 \rangle}

\def\Id{\mathds{1}}
\def\N{\mathds{N}}
\def\Z{\mathds{Z}}
\def\R{\mathds{R}}
\def\C{\mathds{C}}
\def\F{\mathds{F}}
\def\PR{\mathds{P}}

\def\D{\mathcal{D}}

\def\II{\mathcal{I}}		
\def\PP{\mathcal{P}}		
\def\AA{\mathbf{A}}
\def\BB{\mathbf{B}}
\def\CC{\mathbf{C}}
\def\ii{\mathbf{i}}				

\def\im	 			{\mathrm{i}}

\def\Plucker		{Pl\"{u}cker}

\def\pperp		{\simperp}

\DeclareMathOperator{\Span}{span}

\DeclareMathOperator{\Ann}{Ann}
\DeclareMathOperator{\vol}{vol}

\DeclareMathOperator{\tr}{tr}
\DeclareMathOperator{\Hom}{Hom}
\DeclareMathOperator{\diam}{diam}
\DeclareMathOperator{\area}{area}
\DeclareMathOperator{\Isom}{Isom}
\DeclareMathOperator{\Transf}{Transf}
\DeclareMathOperator{\Gr}{Gr}

\def\lcontr					{\mathbin{\lrcorner}}		

\def\wrt					{w.r.t.\ }
\def\ie						{i.e.\ }
\def\resp					{resp.\ }
\def\WLOG					{w.l.o.g.\ }

\def\cd						{\bar{d}}
\def\dg						{d_{\mathrm{g}}}
\def\dcF					{d_{\mathrm{cF}}}
\def\dpF					{d_{\mathrm{pF}}}
\def\dFS					{d_{\mathrm{FS}}}
\def\dcw					{d_{\mathrm{c\wedge}}}
\def\dBC					{d_{\mathrm{BC}}}
\def\dA						{d_{\mathrm{A}}}
\def\dcd					{d_{\mathrm{c2}}}
\def\dpd					{d_{\mathrm{p2}}}
\def\MFS					{\hat{d}_{\mathrm{FS}}}
\def\mFS					{\check{d}_{\mathrm{FS}}}
\def\mpF					{\check{d}_{\mathrm{pF}}}

\def\path 					{\leadsto}					
\def\geod					{\hookrightarrow}			
\def\segm					{\rightarrowtail}			
\def\nullpath				{\dashrightarrow}			
\def\expan					{\nearrow}					
\def\contr					{\searrow}					
\newcommand{\pathdim}[1]	{\overset{#1}{\path}}		
\newcommand{\geodim}[1]		{\overset{#1}{\geod}}		
\newcommand{\segdim}[1]		{\mathrel{\stackon[0pt]{\segm}{\scriptstyle#1}}} 
\newcommand\abaixo[2]		{\ \stackunder[.5pt]{#1}{\scriptstyle#2}\ }

\makeatletter
\newcommand{\orthsum}{\DOTSB\mathbin{\mathpalette\boxplus@\relax}}
\newcommand{\boxplus@}[2]{\vcenter{\hbox{$\m@th#1\boxplus$}}}
\makeatother

\begin{document}

\title{Asymmetric Geometry of Total Grassmannians}

\author{Andr\'e L. G. Mandolesi 
               \thanks{Instituto de Matemática e Estatística, Universidade Federal da Bahia, Av. Milton Santos s/n, 40170-110, Salvador - BA, Brazil. E-mail: \texttt{andre.mandolesi@ufba.br}}}
               
\date{\today \SELF{v3.3} }

\maketitle

\abstract{
Metrics in Grassmannians, or distances between subspaces of same dimension, have many uses,
and extending them to the Total Grassmannian of subspaces of different dimensions is an important problem,
as usual extensions lack good properties or give little information.
Dimensional asymmetries call for the use of asymmetric metrics,
and we present a natural method to obtain them, extending all the main Grassmannian metrics (geodesic, projection Frobenius, Fubini-Study, gap, etc.).
Their geometry adequately reflects containment relations of subspaces, continuous paths link subspaces of distinct dimensions,
and we describe minimal geodesics, shortest paths to move a subspace onto another.
In particular, the Fubini-Study metric extends as an asymmetric angle that is easily computed, has many useful properties, and a nice geometric interpretation.

\vspace{.5em}
\noindent
{\bf Keywords:} Grassmannian, Grassmann manifold, asymmetric metric, distance between subspaces, Fubini-Study, asymmetric angle.

\vspace{3pt}

\noindent
{\bf MSC 2020:}	Primary 14M15; 	
				Secondary 15A75, 	
				51K99	
}

\section{Introduction}

Various metrics (geodesic, projection Frobenius, Fubini-Study, gap, etc.) \cite{Edelman1999,Deza2016,Stewart1990,Qiu2005} 
in Grassmannians or Grassmann manifolds $\Gr_p(n)$ \cite{Kobayashi1996,Harris1992,Kozlov2000I,Kozlov2000III,Wong1967,Bendokat2024} 
	\CITE{Kobayashi inclui caso complexo}
have been used to measure the separation of $p$-dimensional subspaces.
They are important in geometry, linear algebra, functional analysis, and applications that use subspaces to represent data: 
machine learning \cite{Hamm2008,Huang2018,Lerman2011}, 
	\CITE{Zhang2018}
computer vision \cite{Lui2012,Turaga2008,Vishwanathan2006},
coding theory \cite{Ashikhmin2010,Conway1996,Barg2002,Dhillon2008},
wireless communication \cite{Dhillon2008,Love2003,Love2005},  
\CITE{Du2018,Pereira2022}
etc.%
	\CITE{language processing Hall2000 \\
	recommender systems Boumal2015}

Many applications \cite{Pereira2021,Basri2011,Draper2014,Kato1995,Wang2006,Sun_2007,Zuccon2009,Gruber2009,Renard2018,Beattie2005,Sorensen2002,Ye2016} 
need the Total Grassmannian $\Gr(n)$ of subspaces of different dimensions,
but distances used in it have drawbacks.
	\CITE{image recognition Basri2011,Draper2014,Sun2007, Wang2006,Wang2008; \\
		numerical linear algebra Beattie2005,Sorensen2002; \\
		information retrieval Zuccon2009; \\
		EEG signal analysis Figueiredo2010; \\
		wireless communication Pereira2022, Pereira2021}
Some fail the triangle inequality \cite{Basri2011,Draper2014,Pereira2021} and do not give a topology.
	\CITE{Pereira2022}
When dimensions differ the gap \cite{Kato1995} is always 1, not giving any information,
and the symmetric distance \cite{Sun_2007,Wang2006,Zuccon2009} 
	\CITE{Bagherinia2011, Figueiredo2010, Sharafuddin2010}
is at least $1$.
It can be best to have $d(V,W)=0 \Leftrightarrow V\subset W$ \cite{Gruber2009,Renard2018}, which usual metrics do not allow.
The containment gap \cite{Beattie2005,Sorensen2002} 
	\CITE{Kato1995 não usa esse nome}
satisfies it, but carries little data (only the largest principal angle).
Other metrics proposed in \cite{Ye2016} have similar issues.

Any usual (symmetric) metric in $\Gr(n)$ is bound to have problems, 
as subspaces of different dimensions have inherently asymmetric relations:
a plane $P$ can contain a line $L$, not vice versa; areas projected on $L$ vanish, lengths projected on $P$ tend not to; etc.
Its usual topology, of a disjoint union of $\Gr_p(n)$'s, is also inadequate:
it isolates subspaces of different dimensions, ignoring containment or proximity relations between them;
it fails to show that a small region around $L$ can intersect but not contain planes, while near $P$ we have lines; 
it restricts continuous paths to a $\Gr_p(n)$;
etc.
In some applications, dimensions can increase, as subspaces store more data, or decrease, if data is lost or compressed,
and in this topology such processes are discontinuous, hence harder to study. 

The symmetry $d(x,y)=d(y,x)$ of metrics has long been recognized as an overly restrictive simplifying assumption, as often the path, time or cost to go from $x$ to $y$ is not the same as from $y$ to $x$: one-way streets, rush-hour traffic, uphill or downhill, toll roads, etc.
The separation condition $d(x,y)=0 \Leftrightarrow x=y$ is also too strong at times (e.g., for subspaces).
The triangle inequality $d(x,z) \leq d(x,y)+d(y,z)$ is what matters most:
if $d(x,z)$ measures an optimal way to go from $x$ to $z$, it can not be worse than going through $y$.
Without it, balls $B_r(x) = \{y:d(x,y)<r\}$ do not form a basis for a topology
(they can be a subbasis, but for $x$ in an open set $S$, we can not say $B_r(x)\subset S$ for some $r>0$;
or, using this condition to define open sets, $B_r(x)$ might not be open).
	\SELF{Balls $B$ generate 2 topologies, must coincide to be basis: \\
		1) open $A$ = finite intersection of arbitrary unions of $B$'s; \\
		2) $A$ open if any $x\in A$ has $x\in B \subset A$. \\
		Problemas se forem $\neq$: \\
		1) $x\in B_1\cap B_2$ can have no $B_3$ with $x\in B_3 \subset B_1 \cap B_2$, so $A$'s are not unions of $B$'s. \\
		2) $B$ might not be open
	}	
Distances violating the inequality can appear in $\Gr(n)$ \cite{Basri2011,Draper2014,Pereira2021,Ye2016}, sometimes inadvertently,
if one keeps the $\Gr_p(n)$ formulas, or 
\SELF{equivalently}
takes the distance from the smaller subspace to the closest one of same dimension in the larger one.
 
Asymmetric metrics or quasi-metrics%
	\SELFL{or quasi-distances or $T_0$-quasi-pseudometrics}
\cite{Wilson1931,Busemann1944,Zaustinsky1959,Albert1941,Mennucci2013}
	\CITE{Kazeem2014}
occur in 
topology \cite{GoubaultLarrecq2013,Kunzi2001,Kunzi2009},
Finsler geometry \cite{Bao2012},
	\CITE{Flores2013} 
computer science \cite{Mayor2010,Romaguera1999,Seda2008},
	\CITE{Wang2022 Appendix A has examples of quasi-metrics in probability and information}
category theory \cite{Gutierres2012,Lawvere1973},
biology \cite{Stojmirovic2004,Stojmirovic2009},
	\CITE{Stojmirovic200}
materials science \cite{Mielke2003a,Rieger2008},
	\CITE{Mainik2005,Mielke2003}
etc.%
	\CITE{AlgomKfir2011,Anguelov2016, Chenchiah2009 \\
	 graph theory Fang2022}
Richer than metrics, they carry data in distances $d(x,\cdot)$ and $d(\cdot,x)$ \emph{from} and \emph{to} $x$,
and usual concepts split into \emph{backward} and \emph{forward} ones,
a common duality in asymmetric structures \cite{Kopperman1995}.  
A weaker separation condition lets them induce and generalize partial orders $\preccurlyeq$, 		
with $d(x,y)$ measuring the failure of $x \preccurlyeq y$.
They give two non-Hausdorff topologies, linked to $\preccurlyeq$ and $\succcurlyeq$, which complement each other in a \emph{bitopological space} \cite{Kopperman1995,Ivanov2000,Kelly1963}, and combine into a symmetrized metric topology.
Many metric results have asymmetric versions, 
while some asymmetric ones have only trivial metric analogues \cite{Chenchiah2009,Collins2007,Cobzas2012,GarciaRaffi2003,Romaguera2015,Kelly1963,Mennucci2014}.

In this article we extend, in a natural way, metrics from $\Gr_p(n)$ to asymmetric metrics $d$ in $\Gr(n)$,
with $d(V,W)$ measuring how far $V$ is from satisfying $V \subset W$.
If $\dim V \leq \dim W$, $d(V,W)$ retains the original formula, so it should be adequate for the same kinds of application as the original metric.
But if $\dim V > \dim W$, $d(V,W)$ stays constant at a maximum, as $V \subset W$ never gets any closer to happening, while $d(W,V)$ still describes their separation.

This seemingly innocuous detail improves the properties and usefulness of $d$,
which incorporates the main condition affecting how subspaces relate (which one is larger).
This leads to dimension-independent results,
that keep us from having to analyze lots of cases when various subspaces are involved.
It is also important as dimensions might change or not be known beforehand in applications (e.g., for subspaces obtained via projections, truncating the spectrum of an operator, etc.).

All the main asymmetric metrics give $\Gr(n)$ the same two natural topologies $\tau^\pm$, linked to the partial orders $\subset$ and $\supset$,
and which combine into the disjoint union topology, linked to the $=$ relation.
While $\tau^\pm$ are asymmetric, the whole bitopological structure is symmetric, reflecting an important symmetry of $\Gr(n)$ given by the orthogonal complement: as $\perp$ flips containment relations ($V \subset W \Leftrightarrow V^\perp \supset W^\perp$),
it also flips $\tau^- \leftrightarrow \tau^+$.

With $\tau^\pm$, $\Gr(n)$ is path connected, so a subspace can move continuously onto another of a different dimension.
In this case, its path can have a discontinuous length change, which can be used as a penalty for data loss, or for data expansion if one wants to compress it.

We describe minimal geodesics (shortest paths from $V$ to $W$),
analyze when they are segments (their length is $d(V,W)$),
and when the triangle inequality attains equality.
This is important for applications as it reveals extremal cases where something relevant can happen.%
\CITE{Qiu2005 p.\,519}

Geodesic and Fubini-Study metrics receive special attention.
The first one gives geodesic lengths for some metrics.
The second one is important for wireless communication \cite{Dhillon2008,Love2003,Love2005,Pereira2021}
	\CITER{Pereira2022}
and quantum theory \cite{Bengtsson2017}, 
	\CITE{Só Fubini-Study em $\PR(H)$. Also	Ortega2002, Brody2001, Yu2019}
with complex spaces,
and extends to $\Gr(n)$ as an asymmetric angle 
\cite{Mandolesi_Grassmann} that measures volume contraction in orthogonal projections.
Links to Grassmann and Clifford geometric algebras \cite{Mandolesi_Products},
powerful but underused formalisms to work with subspaces,
give easy ways to compute the angle and useful properties \cite{Mandolesi_Trigonometry},
like volumetric Py\-thag\-o\-re\-an theorems \cite{Mandolesi_Pythagorean} with interesting implications for quantum theory \cite{Mandolesi_Born}.

\Cref{sc:preliminaries} sets up notation and terminology, and reviews asymmetric metrics, Grassmannians, distances used in $\Gr(n)$ and their problems.  
\Cref{sc:Other asymmetric metrics} obtains asymmetric metrics in $\Gr(n)$, and gives some properties.
\Cref{sc:Asymmetric geometry} describes their topologies, isometries, paths, minimal geodesics, and, in particular, convexity results for the asymmetric geodesic metric.
\Cref{sc:Asymmetric Fubini-Study} links the asymmetric Fubini-Study, chordal $\wedge$ and Binet-Cauchy metrics to asymmetric angles and volume projection factors, 
gives geometric interpretations,
and obtains convexity results for the Fubini-Study one.
\Cref{sc:conclusion} closes with some remarks.

\Cref{sc:Grassmann algebra} reviews Grassmann exterior algebra.
\Cref{sc:Metrics and distances on Grassmannians} reviews and classifies Grassmannian metrics, and proves some inequalities.
\Cref{sc:Fubini-Study metric} discusses different Fubini-Study distances, and extends convexity results to the complex case.
\Cref{sc:Asymmetric angles} reviews asymmetric angles, gives formulas to compute them, and new results for the complex case.

\section{Preliminaries}\label{sc:preliminaries}

In this article, $\F = \R$ or $\C$, and $\F^n$ has its canonical inner product $\inner{\cdot,\cdot}$ (Hermitian, if $\F=\C$, with conjugate-linearity in the first entry).\label{df:F inner}
A $p$-dimensional subspace $V$ is a \emph{$p$-subspace}, or \emph{line} if $p=1$.
We write $V_{(p)}$\label{df:Vp} to show its dimension,
$P_V: \F^n\rightarrow V$ for orthogonal projection,\label{df:orthogonal projections}
$\boxplus$\label{df:orth sum} for orthogonal direct sum,
$0$ for the subspace $\{0\}$, 
$^\dagger$ for conjugate transpose,\label{df:conj transp}
$\|\cdot\|_{\mathrm{F}}$ 
	\SELF{$\|\mathbf{A}\|_{\mathrm{F}}^2 = \sum |a_{ij}|^2 = \tr(A^\dagger A)$}
and $\|\cdot\|_2$ for Frobenius and operator norms,
and $V \subset W$ allows $V=W$, while $\subsetneq$ prevents it.
\Cref{tab:symbols} has other symbols we use.

\begin{table}[]
	\scriptsize
	\centering
	\renewcommand{\arraystretch}{1}
	\begin{tabular}{lll} 
		\toprule
		Symbol & Description & Page
		\\
		\cmidrule(lr){1-1} \cmidrule(lr){2-2} \cmidrule(lr){3-3} 
		$\F$ & $\R$ or $\C$ & \pageref{df:F inner}
		\\
		$\inner{\cdot,\cdot}$, $\|\cdot\|$, $\|\cdot\|_{\mathrm{F}}$, $\|\cdot\|_2$ & Inner/Hermitian product, norms & \pageref{df:F inner}, \pageref{df:inner AB} 
		\\
		$V_\R$, $\inner{\cdot,\cdot}_\R$ & Underlying real space and inner product & \pageref{df:underlying real 2}
		\\
		$V_{(p)}$ & Subspace $V$ of dimension $p$ & \pageref{df:Vp} 
		\\
		$P_V$ & Orthogonal projection on $V$ & \pageref{df:orthogonal projections} 
		\\
		$\PP_W(V)$ & Set of projection subspaces of $W$ \wrt $V$ & \pageref{df:PO}
		\\
		$\boxplus$ & Orthogonal direct sum of subspaces or paths & \pageref{df:orth sum}, \pageref{df:path sum}
		\\
		$^\dagger$ & Conjugate transpose & \pageref{df:conj transp}
		\\
		$\theta_{v,w}$, $\phi_{v,w}$, $\theta_{K,L}, \theta_i$ & Angles between vectors, lines or subspaces & \pageref{eq:angles}
		\\
		$\theta_{K,L}$, $c_{K,L}$, $g_{K,L}$ & Angular, chordal and gap distances of lines & \pageref{df:angle lines}, \pageref{fig:distances_lines}
		\\
		$\Theta_{V,W}$, $\pi_{V,W}$ & Asymmetric angle, volume projection factor & \pageref{df:Theta pi}
		\\
		$\pperp$, $\perp$ & Partially orthogonal, orthog.\ (complement, map) & \pageref{df:pperp}, \pageref{pr:antiisometry} 
		\\
		$\tau^-$, $\tau^+$, $\tau$, $B^-_r(x)$, $B^+_r(x)$, $B_r(x)$ & Topologies and balls & \pageref{df:tau}
		\\
		$\Isom$, $\Transf$ & Isometry and transformation groups & \pageref{df:Isom Transf groups}
		\\
		$L(\gamma)$, $L_1$, $L_2$, $\Delta L(\gamma)\lvert_{t_0}$ & Lengths of curves, length change at $t_0$ & \pageref{df:Delta L}, \pageref{df:L1 L2}
		\\
		$\Delta_p$, $\Delta_{p,d}$ & Diameter of $\Gr_p(\infty)$ for the metric $d$ & \pageref{df:diameter}
		\\
		$[V,W]_d$, $(V,W)_d$ & Sets of between-points from $V$ to $W$ & \pageref{df:between}
		\\
		$\Gr_p(V)$, $\Gr_p(n)$, $\Gr_p(\infty)$ & Grassmannians & \pageref{df:GrV}, \pageref{df:Gr}, \pageref{df:Gr infty}
		\\
		$\Gr(V)$, $\Gr(n)$, $\Gr(\infty)$, $\Gr^{\pm}(n)$ & Total Grassmannians & \pageref{df:Total Gr}, \pageref{df:Gr infty}, \pageref{df:Gr pm}
		\\
		$\bar{d}$, $\hat{d}$, $D$ & Conjugate, max-symm., intrinsic asymm. metrics & \pageref{df:conjugate maxsym}, \pageref{df:D}
		\\
		$\mFS$, $\MFS$, $\tilde{d}_{\mathrm{pF}}$, $\mpF$, $\vec{d}$, $d_{\mathrm{s}}$, $\delta$, $\hat{\delta}$ & Distances in $\Gr(n)$ & \pageref{tab:distances Total}, \pageref{eq:mFS}
		\\
		$\dg$, $\dcF$, $\dpF$, $\dFS$, $\dcw$, $\dBC$, $\dA$, $\dcd$, $\dpd$ & Asymmetric metrics in $\Gr(n)$, metrics in $\Gr_p(n)$ & \pageref{tab:asymmetric metrics}, \pageref{tab:metrics Gpn}
		\\
		$\expan$, $\contr$ & Expansion, contraction & \pageref{df:expansion contraction}
		\\
		$\path$, $\geod$, $\segm$, $\nullpath$ & Path, minimal geodesic, segment, null path & \pageref{df:arrows}
		\\
		$\!\!\abaixo{\path}{\gamma}\!\!$, $\path_I$, $\pathdim{p}$ & Path $\gamma$, path of type I, path in $\Gr_p(n)$  & \pageref{df:arrows}
		\\
		$\bigwedge V$, $\bigwedge^p V$ & Exterior algebra, exterior power & \pageref{df:Grass alg}
		\\
		$\wedge$, $\lcontr$ & Exterior and interior products & \pageref{df:Grass alg}, \pageref{df:contr}
		\\
		$[B]$ & Subpace of a blade $B$ & \pageref{eq:blade space}
		\\
		\bottomrule
	\end{tabular}
	\caption{Some symbols used in this article}
	\label{tab:symbols}
\end{table}

In the complex case, the natural identification of $\C^n$ with $\R^{2n}$ also identifies a $p$-subspace $V$ of $\C^n$ with a $2p$-subspace  of $\R^{2n}$, its \emph{underlying real space} $V_\R$.\label{df:underlying real}
Also, $(\C^n)_\R$ has inner product $\inner{\cdot,\cdot}_\R = \operatorname{Re} \inner{\cdot,\cdot}$, and
\begin{equation}\label{eq:Hermitian}
	\inner{v,w} = \inner{v,w}_\R + \im \inner{\im v,w}_\R.
\end{equation}
In the real case, let $V_\R = V$ and $\inner{\cdot,\cdot}_\R = \inner{\cdot,\cdot}$.\label{df:underlying real 2}

\begin{definition}\label{df:angle lines}
	The \emph{Euclidean} and \emph{Hermitian angles} of $v,w\neq 0$ are 
	\begin{equation}\label{eq:angles}
		\theta_{v,w} = \cos^{-1} \frac{\inner{v,w}_\R}{\|v\||w\|} \in [0,\pi] \quad \text{and} \quad \phi_{v,w} = \cos^{-1} \frac{|\inner{v,w}|}{\|v\||w\|} \in [0,\tfrac\pi2].
	\end{equation}
	The \emph{angle between lines} $K=\Span\{v\}$ and $L=\Span\{w\}$
	\OMIT{$v,w\neq 0$}
	is $\theta_{K,L} = \phi_{v,w}$. 
		\OMIT{$\in [0,\frac\pi2]$}
	Also, $v$ and $w$ are \emph{aligned} if $\inner{v,w} \geq 0$.
\end{definition}

So, $\theta_{v,w}$ is the usual angle, but in the underlying real space if $\F=\C$.
And $\phi_{v,w} = \theta_{v,P_{L} v}$ ($=\frac\pi2$ if $P_L v=0$) \cite{Scharnhorst2001}.
Also, $v$ and $w$ are aligned
$\Leftrightarrow P_L v = \lambda w$ for $\lambda \geq 0$,
in which case $\theta_{K,L} = \phi_{v,w}= \theta_{v,w}$.

\emph{Principal angles} \cite{Bjorck1973,Stewart1990,Qiu2005}
\CITE{Afriat1957,Galantai2006,Golub2013} 
$0\leq \theta_1\leq\cdots\leq\theta_m \leq \frac \pi 2$ of  $V_{(p)}$ and $W_{(q)}$, with $m=\min\{p,q\}\neq 0$, are $\theta_i = \cos^{-1} \sigma_i$, where $\sigma_i$ is the $i\textsuperscript{th}$ singular value of the orthogonal projection $P:V\rightarrow W$.
\emph{Principal bases} $\beta = (e_1,\ldots,e_p)$ of $V$ and $(f_1,\ldots,f_q)$ of $W$ are orthonormal bases with $\inner{e_i,f_j} = \delta_{ij} \cos \theta_i$.
We also say $\beta$ is a principal basis of $V$ \wrt $W$.
The \emph{principal vectors} $e_i$ and $f_j$ are aligned singular vectors of $P$, with $\theta_{e_i,f_i} = \theta_i$.

 \begin{definition}\label{df:pperp}
	$V$ is \emph{partially orthogonal} ($\pperp$) to $W$ if $W^\perp \cap V \neq 0$.
\end{definition}

We have $V_{(p)} \pperp W_{(q)} \Leftrightarrow$ a principal angle is $\frac\pi2$ or $p>q$.
This relation is asymmetric when $p \neq q$.

\subsection{Asymmetric metrics}\label{sc:Asymmetric metrics}

We define asymmetric metrics as follows \cite{Albert1941,GoubaultLarrecq2013,Cobzas2012,Mennucci2013,GarciaRaffi2003,Mennucci2014,Romaguera2015}.
Some authors call them \emph{quasi-metrics}, a term often used for a similar concept \cite{Wilson1931,Busemann1944,Zaustinsky1959,Chenchiah2009,Collins2007,Kelly1963} with a $T_1$ separation condition ($d(x,y)=0 \Leftrightarrow x=y$).

\begin{definition}\label{df:asymmetric metric}
	An \emph{asymmetric metric} on a set $M$ is $d:M\times M\rightarrow [0,\infty]$ 
	\SELF{è bom deixar $\infty$, para a intrinsic metric, se não houver path} satisfying, for $x,y,z\in M$:
	\begin{enumerate}[(i)]
		\item $d(x,y)=d(y,x)=0 \Leftrightarrow x=y$ \ (\emph{$T_0$ separation condition}). \label{it:T0}
		\item $d(x,z) \leq d(x,y)+d(y,z)$ \ (\emph{oriented triangle inequality}). \label{it:oriented triang ineq}
	\end{enumerate}
\end{definition}

Unlike usual metrics, the \emph{distance $d(x,y)$ from $x$ to $y$} can differ from $d(y,x)$.
This greatly increases the applicability of $d$ \cite{GoubaultLarrecq2013,Kunzi2001,Kunzi2009,Bao2012,Mayor2010,Romaguera1999,Seda2008,Gutierres2012,Lawvere1973,Stojmirovic2004,Stojmirovic2009,Mielke2003a,Rieger2008}, 
and lets it carry more data, in $d(x,y)$ and $d(y,x)$.
But it requires some care, and one must pay attention to the order of elements in (\ref{it:oriented triang ineq}).

The advantage of $T_0$ over the $T_1$ condition is that it allows $d$ to induce a nontrivial partial order on $M$ by $x \preccurlyeq y \Leftrightarrow d(x,y)=0$. 
Any partial order can be represented this way for some $d$.
\CITEL{[ex.1.4]Mennucci2013, [p.\,21]{Stojmirovic2005}, Anguelov2016}
When $d$ vanishes, $\preccurlyeq$ can give a further sense of proximity: 
	\CITE{Anguelov2016}
for distinct $x,y,z\in M$, if $d(x,y)=d(y,z)=0$ then 
	\SELF{$d(x,z)=0$, by (\ref{it:oriented triang ineq})}
$x\prec y \prec z$ suggests $y$ and $z$ are, in a sense, closer than $x$ and $z$.
Indeed, (\ref{it:T0}) and (\ref{it:oriented triang ineq}) give
$0 \neq d(z,y) \leq d(z,x) + d(x,y) = d(z,x)$, 
but this $\leq$ is not always strict, as can be seen in \Cref{ex:N} below. 

\begin{definition}\label{df:conjugate maxsym}
	The \emph{conjugate} asymmetric metric of $d$ is $\bar{d}(x,y) = d(y,x)$.
	The \emph{max-sym\-me\-trized} metric is $\hat{d}(x,y)=\max\{d(x,y),d(y,x)\}$.
\end{definition}

While $\bar{d}$ induces $\succcurlyeq$, $\hat{d}$ is related to $=$, in the sense that $\hat{d}(x,y) = 0 \Leftrightarrow x=y$.
Symmetrizing by the $\min$ would violate the triangle inequality,
but any symmetric monotone 
\CITER{Other name for symmetric is permutation-invariant: Matrix Anal., Horn, Johnson p.285; Horn1991 p.169; Qiu2005 p.512}
norm%
\footnote{This means $\|(a,b)\| = \|(b,a)\|$, and $|a|\geq |c|$, $|b| \geq |d| \Rightarrow \|(a,b)\|\geq \|(c,d)\|$ \cite{Bhatia1997}.} 
$\|\cdot\|$ in $\R^2$ symmetrizes $d$ (with some loss of information) into a metric $\rho(x,y) = \|(d(x,y),d(y,x))\|$.


\begin{definition}\label{df:tau}
	\emph{Forward, backward} and \emph{symmetric topologies} 
	\CITER{Flores2013, Mennucci2013, Mennucci2014}
	$\tau^+$, $\tau^-$ and $\tau$ are generated 
		\SELF{in the sense of being the smallest topolgies in which these balls are open}
	by \emph{forward, backward} and \emph{symmetric balls} given by
	\begin{equation}\label{eq:balls}
		\begin{aligned}
			B^+_r(x) &= \{y\in M:d(x,y)<r\}, \\
			B^-_r(x) &= \{y\in M:\bar{d}(x,y)<r\} = \{y\in M:d(y,x)<r\}, \text{ and} \\
			B_r(x) &= \{y\in M:\hat{d}(x,y)<r\} = B^+_r(x) \cap B^-_r(x).
		\end{aligned}
	\end{equation}
\end{definition}

So, $B^+_r(x)$ uses distances $d(x, \cdot)$ from $x$, $B^-_r(x)$ has distances $d(\cdot, x)$ to $x$, and $B_r(x)$ takes the largest one.
Some authors 
	\CITE{Kazeem2014} 
treat $\tau^+$ and $\tau^-$ as the topologies of $d$ and $\bar{d}$, but it is better to view $(M,\tau^\pm)$ as a \emph{bitopological space} \cite{Kopperman1995,Ivanov2000,Kelly1963}.
While $\tau^\pm$ are not Hausdorff,
\OMIT{have $d(x,y)=0$ for $x\neq y$}
they complement each other, and together are as effective (in fact, more) as the metric topology $\tau$ of $\hat{d}$.

\begin{definition}\label{df:Isom Transf groups}
	$f:M \rightarrow M$ is an \emph{isometry} if $d(f(x),f(y)) = d(x,y)$ for all $x,y\in M$, 
	or an \emph{anti-isometry} if $d(f(x),f(y)) = d(y,x)$.
	The \emph{isometry and transformation groups} of $M$ are $\Isom(M) = \{$bijective 
	\SELF{$T_0$ garante injetividade, mas nem $d$ usual dá sobejetividade: $f:\R^+\rightarrow\R^+$, $f(x)=x+1$}
	isometries of $M\}$
	and $\Transf(M) = \{$bijective isometries and anti-isometries of $M\}$.
\end{definition}

A bijective isometry gives self-homeomorphisms of $\tau$, $\tau^+$ and $\tau^-$.
A bijective anti-isometry gives a self-homeomorphism of $\tau$, and a homeomorphism $\tau^+ \cong \tau^-$ switching backward and forward topologies.

Analysis in $M$ requires some care, as limits are not unique in a $T_0$ topology.
Also, 
	\CITE{Chenchiah2009 p.5823, Cobzas2012 p.8.\\
		Stojmirovic2005 p.\,23 errado}
$d(x,y)$ is upper sem\-i\-con\-tin\-u\-ous in $x$ for $\tau^-$ and lower semicontinuous for $\tau^+$ (in $y$ it is the opposite),
as (\ref{it:oriented triang ineq}) gives\OMIT{$d(x_k,y) \leq  d(x_k,x) + d(x,y)$,\,$\tau^-$\\
		$d(x,x_k) + d(x_k,y) \geq d(x,y)$,\,$\tau^+$ \\[3pt]
		$\limsup_{k} d(x,y_k) \leq d(x,y)$ \\
		if $d(y,y_k)\rightarrow 0$ ($\tau^+$), as \\
		$d(x,y_k) \leq d(x,y) + d(y,y_k)$;\\[3pt]
		$\liminf_{k} d(x,y_k) \geq d(x,y)$ \\
		if $d(y_k,y)\rightarrow 0$ ($\tau^-$), as \\
		$d(x,y_k) + d(y_k,y) \geq d(x,y)$}
\begin{equation}\label{eq:semicontinuidade}
	\begin{aligned}
		\limsup_{k\rightarrow\infty} d(x_k,y) &\leq d(x,y) \text{ \ if } d(x_k,x)\rightarrow 0, \text{ and} \\
		\liminf_{k\rightarrow\infty} d(x_k,y) &\geq d(x,y) \text{ \ if } d(x,x_k)\rightarrow 0.
	\end{aligned}
\end{equation}

\begin{definition}
	A sequence $x_k$ is \emph{left (\resp right) K-Cauchy} if given $\epsilon>0$ there is $n \in \N$ such that $d(x_k,x_l)<\epsilon$ (\resp $d(x_l,x_k)<\epsilon$) whenever $n \leq k \leq l$.
	If any such sequence has $\hat{d}(x_k,x) \rightarrow 0$ for some $x\in M$ then $(M,d)$ is \emph{left (\resp right) Smyth complete}.
	\CITER{Cobzas2012}
	It is \emph{Smyth bicomplete} if both.
\end{definition}

See \cite{Cobzas2012} for other completeness concepts.

\begin{definition}
	A \emph{curve} in $M$ (for $\tau$ or $\tau^\pm$) is a continuous $\gamma:I\rightarrow M$, 
	for an interval $I\subset \R$.
	If $I=[a,b]$, $\gamma$ is a \emph{path} from $\gamma(a)$ to $\gamma(b)$,
	and it is \emph{trivial} if its image has no other points besides these.
%
\end{definition}

Being coarser, $\tau^\pm$ have more curves  than $\tau$:
$\gamma$ is continuous at $t$ for $\tau^-$ if
\SELF{and only if} 
$\lim_{s \rightarrow t} d(\gamma(s),\gamma(t)) = 0$;
for $\tau^+$, if $\lim_{s \rightarrow t} d(\gamma(t),\gamma(s)) = 0$;
and $\tau$ needs both.
On the other hand, $\tau^\pm$ have less continuous functions:
if $f:M \rightarrow \R$ is continuous, $x \preccurlyeq y \Rightarrow f(x) = f(y)$.

For $\tau^\pm$, there is a trivial path from $x$ to $y \Leftrightarrow x \preccurlyeq y$ or $y \preccurlyeq x$.%
	\SELF{ in the first case they are null. Null paths can be nontrivial.}

\begin{definition}\label{df:lengths}
		A curve $\gamma:I\rightarrow M$ is \emph{rectifiable} if it has finite \emph{length} 
	\begin{equation}\label{eq:length}
		L(\gamma) = \sup_{\substack{t_1 < \cdots < t_N \\ t_i \in I, N\in\N}} \sum_{i=1}^{N-1} d(\gamma(t_i),\gamma(t_{i+1})).
	\end{equation}
	It is \emph{null} if $L(\gamma) = 0$.%
	\SELF{constant in $\tau$, not $\tau^\pm$}	
	The \emph{length change} of $\gamma$ at $t_0\in I$ is
	\begin{equation}\label{df:Delta L}
		\Delta L(\gamma)\lvert_{t_0} = \lim\limits_{\delta \rightarrow 0} L(\gamma\lvert_{I\cap [t_0-\delta,t_0+\delta]}).
	\end{equation}
\end{definition}

The reversed curve $-\gamma$ (i.e., $-\gamma(t) = \gamma(-t)$) can be non-rectifiable, or have a different length.
If $f$ is an anti-isometry,
\begin{equation}\label{eq:L reversed}
	L(-\gamma) = L(f\circ\gamma).
\end{equation}

In $\tau^\pm$, the length of $\gamma\lvert_{[a,t]}$ can vary discontinuously with $t$, 
so we can have $\Delta L(\gamma)\lvert_{t_0} \neq 0$.
For our purposes, this will be convenient, but some authors avoid it using \emph{run-continuous} paths \cite{Mennucci2014}.
	\CITE{Mennucci2014 uses \emph{run-continuous} paths in $\tau$, so they are continuous. In $\tau$, not $\tau^\pm$, a rectifiable continuous path is run-continuous.}

\begin{lemma}
	For a curve $\gamma:I \rightarrow M$ and $t_0 \in I$,
		\SELF{$L(\gamma\lvert_{I\cap (-\infty,t_0)}) = 0$ if $t_0 = \inf I$, $L(\gamma\lvert_{I\cap (t_0,\infty)}) = 0$ if $t_0=\sup I$.}
	\begin{subequations}\label{eq:L Delta L}
		\begin{align}
			L(\gamma) &= L(\gamma\lvert_{I\cap (-\infty,t_0]}) + L(\gamma\lvert_{I\cap [t_0,\infty)}) \label{eq:L Delta L closed} \\ 
			&= L(\gamma\lvert_{I\cap (-\infty,t_0)}) + \Delta L(\gamma)\lvert_{t_0} + L(\gamma\lvert_{I\cap (t_0,\infty)}). \label{eq:L Delta L open}
		\end{align}
	\end{subequations}
	And, disregarding the first $\lim$ if $t_0 = \inf I$, or the last one if $t_0=\sup I$,
	\begin{equation}\label{eq:Delta L}
		\Delta L(\gamma)\lvert_{t_0} = \lim\limits_{t\rightarrow t_0^-} d(\gamma(t),\gamma(t_0)) + \lim\limits_{t\rightarrow t_0^+} d(\gamma(t_0),\gamma(t)).
	\end{equation}
\end{lemma}
\begin{proof}
	\eqref{eq:L Delta L} is simple.
	We prove \eqref{eq:Delta L} for $t_0 = \inf I$;
	the case $t_0 = \sup I$ is similar, and the general one follows by splitting \eqref{df:Delta L} via \eqref{eq:L Delta L closed}.
	
	Let $\Delta = \Delta L(\gamma)\lvert_{t_0}$.
	Given $\epsilon>0$ there are $\delta>0$ and $t_1 < \cdots < t_N$ in $I$, with $t_1 \in (t_0,t_0+\delta)$,
	such that
	$L(\gamma\lvert_{[t_0,t_1]}) \in \big[\Delta, \Delta + \epsilon \big)$
	and 
	$\epsilon > L(\gamma) - \sum_{i=0}^{N-1} d(\gamma(t_i),\gamma(t_{i+1})) \geq L(\gamma\lvert_{[t_0,t_1]}) - d(\gamma(t_0),\gamma(t_1)) \geq 0$.
	So
	$d(\gamma(t_0),\gamma(t_1)) \in \big(\Delta - \epsilon , \Delta + \epsilon \big)$.
	This gives \eqref{eq:Delta L} without the first $\lim$.
%
\end{proof}

%

\begin{definition}
	The infimum of lengths of paths from $x$ to $y$ gives the \emph{intrinsic asymmetric metric} 
	\label{df:D}
	$D(x,y)$ of $d$,
		\SELFR{$D\geq d$, and $D(x,y) = \infty$ if there is no rect path $x$ to $y$}
	and $d$ is \emph{intrinsic} if $D=d$.
	\SELF{and $M$ is a \emph{length space}}
	A path $\gamma$ from $x$ to $y$ is a \emph{minimal geodesic} if $L(\gamma) = D(x,y)$,
	and a \emph{segment} if $L(\gamma) = d(x,y)$.
	\CITE{Busemann1944}
	$M$ is a \emph{geodesic space} if any two points are linked by a minimal geodesic.
	\SELF{not necessarily unique}
\end{definition}

These concepts depend on the choice of topology.
As $d(x,y) \leq D(x,y)$, any segment is a minimal geodesic, and vice versa if $d$ is intrinsic.

\begin{definition}\label{df:between}
	For $x,y,z\in M$, we say $y$ is a \emph{between-point} from $x$ to $z$ if $d(x,z) = d(x,y)+d(y,z)$.
	The set of such points is $[x,z]_d$, and $(x,z)_d = [x,z]_d\backslash\{x,z\}$.
		\CITE{Some authors require $y \neq x,z$, or all distinct. \\ Jiang1996, Blumenthal1970 p. 33 for metric spaces \\
		Some authors use a stronger concept of Menger convex}
	$M$ is \emph{Menger convex}
	if $(x,z)_d \neq \emptyset$ for all $x \neq z$.
\end{definition}

If $y$ is in a segment from $x$ to $z$ then $y \in [x,z]_d$.
	\OMITR{$d(x,z) = L(\gamma) = L(\gamma\lvert_{x \text{---} y}) + L(\gamma\lvert_{y\text{---}z}) \geq d(x,y) + d(y,z) \geq d(x,z) \Rightarrow$ equalities}
The converse holds if $M$ is geodesic and $d$ is intrinsic.
Note how the $T_0$ separation condition ensures $[x,x]_d = \{x\}$ and $(x,x)_d = \emptyset$. 
	\SELF{so $y \in (x,z)_d$ implies $x,y,z$ are pairwise distinct.}	

The next example shares many features with the Total Grassmannian. \Cref{pr:quotient} will link them, with $\D$ as a set of subspace dimensions.

\begin{example}\label{ex:N}
	An asymmetric metric in $\D=\{0,1,\ldots,n\}$ is
	\begin{equation*}
		d(p,q) = \begin{cases}
			0 &\text{if } p\leq q, \\
			1 &\text{if } p>q,
		\end{cases}
	\end{equation*}
	which induces $\leq$.
	While $(\D,\tau)$ is discrete, $\D^\pm =(\D,\tau^\pm)$ are contractible,
	with open sets of the form $\{0,1,\ldots,p\}$ in $\D^-$, 
	and $\{p,p+1,\ldots,n\}$ in $\D^+$.
	Isometries preserve $\leq$ and anti-isometries reverse it, so $\Isom(\D) = \{\Id\}$ and $\Transf(\D) = \{\Id,f\}$, with $f(p)=n-p$, which flips $\D^- \leftrightarrow \D^+$.
	Left or right K-Cauchy sequences in $\D$ are eventually constant, so $d$ is Smith bicomplete.
	And $\D^\pm$ are geodesic spaces with $d$ intrinsic: a null minimal geodesic from $p$ to $q>p$ is given in $\D^-$ (or $\D^+$, with $[0,1]$ and $(1,2]$) by
	\begin{equation*}
		\gamma(t) = \begin{cases}
			p &\text{if } t \in [0,1), \\
			q &\text{if } t \in [1,2],
		\end{cases}
	\end{equation*} 
	and $-\gamma$ is a minimal geodesic of length $1$ from $q$ to $p$.
	We can also have minimal geodesics from $p$ to $q$ passing through $p<m<q$ (whose reversed paths have at least length 2, and so are not minimal), 
	or others from $q$ to $p$ going from $q$ to $s>q$, decreasing to $r<p$, then increasing to $p$.
\end{example}

\subsection{Grassmannians}\label{sc:Grassmannians}

The \emph{Grassmannian} 
$\Gr_p(V) = \{p$-subspaces of $V\}$ 
	\label{df:GrV}
of a $q$-subspace $V$ has a structure of connected compact manifold of dimension $p(q-p)$ if $p\leq q$ \cite{Kobayashi1996,Harris1992,Kozlov2000I,Kozlov2000III,Wong1967,Bendokat2024},
and $\Gr_p(V) = \emptyset$ if $p>q$.
We also write%
\footnote{Common notations are $\Gr(p,\F^n)$, $\Gr(p,n)$, $\Gr(n,p)$ or even $\Gr(p,n-p)$.}
$\Gr_p(n) = \Gr_p(\F^n)$.\label{df:Gr}
See in \Cref{sc:Metrics and distances on Grassmannians} how we classify ($\ell^2$, $\wedge$, max) the metrics in it.

For $\ell^2$ and $\wedge$ metrics, curve lengths coincide%
	\footnote{As these metrics converge as\-ymp\-tot\-i\-cal\-ly for small $\theta_i$'s (see \cite[p.\,337]{Edelman1999} or \cite[p.\,98]{Stewart1990}).}, 
and $\Gr_p(n)$ is a geodesic space.
Up to reparametrization 
	\SELF{geodésica Riemanniana tem velocidade const, em espaço métrico não}
and choice of principal bases $(e_1,\ldots,e_p)$ of $V$ and $(f_1,\ldots,f_p)$ of $W$,
any minimal geodesic $\mu$ from $V$ to $W$ is given by \emph{direct rotations} \cite{Qiu2005}: with principal angles $\theta_1\leq\cdots\leq\theta_p$ and $t\in[0,1]$, 
	\SELF{the $e_i$'s rotate at constant speeds towards the $f_i$'s}
\begin{subequations}\label{eq:geodesic Grp}
	\begin{align}
			\mu(t) &= \Span\{v_1(t),\ldots,v_p(t)\}, \\
			\shortintertext{where}
			v_i(t) &= 
			\begin{cases}
				\cos(t\theta_i) e_i + \sin (t\theta_i) \frac{f_i - P_{V} f_i}{\|f_i - P_{V} f_i\|} & \text{ if } \theta_i \neq 0, \\
				e_i & \text{ if } \theta_i=0.
			\end{cases}
	\end{align}
\end{subequations}
Its length is the \emph{geodesic metric} $\dg(V,W) = \sqrt{\sum_{i=1}^p \theta_i^2}$, which is the intrinsic metric of all $\ell^2$ and $\wedge$ ones.
For $U,V,W\in\Gr_p(n)$, $U \in [V,W]_{\dg} \Leftrightarrow U = \mu(t_0)$ for some $t_0$ and $\mu$ as above.

The \emph{Total Grassmannian} $\Gr(V) = \{$subspaces of $V\} =\bigcup_p \Gr_p(V)$ 
\label{df:Total Gr}
is usually given the disjoint union topology,
	\SELF{So it is compact}
which is geometrically inadequate
(it separates a plane from its lines, for example).
In \Cref{sc:Asymmetric geometry} we describe more suitable topologies.
Some distances used in $\Gr(n) = \Gr(\F^n)$, shown in \Cref{tab:distances Total}, are extensions of metrics of \Cref{sc:Metrics and distances on Grassmannians}:
\begin{itemize}
	\item $\mFS$ and $\MFS$ extend the Fubini-Study metric $\dFS$ (see \Cref{sc:Fubini-Study metric}).
	
	\item 	
	$\mpF$ \cite{Basri2011,Draper2014,Pereira2021}
		\CITE{Pereira2022}
	is the projection Frobenius distance $\dpF$ from the smaller subspace to the closest one of same dimension in the larger one (in general, its projection).
	Other extensions of $\dpF$ are $\tilde{d}_{\mathrm{pF}}$ \cite{Ashikhmin2010} 
	and the \emph{directional}%
		\footnote{Following \cite{Wang2015}, some authors use $\vec{d}(V,W) = \min_{\mathbf{H} \in \R^{q \times p}} \|\mathbf{X} - \mathbf{Y}\mathbf{H}\|_{\mathrm{F}}$ for any matrices $\mathbf{X}$ and $\mathbf{Y}$ whose columns span $V_{(p)}$ and $W_{(q)}$. But without orthonormality this is a distance between a basis and a subspace, not two subspaces (e.g., it doubles if we multiply $\mathbf{X}$ by 2).}
	and \emph{symmetric distances} $\vec{d}(V,W)$ 
		\SELFL{$= \|P_{W^\perp} P_V\|_{\mathrm{F}}$. \\
		Can use any orthon basis of $V$. \\ $\min\limits_{W'\in\Omega_p^-(W)} \!\!\!\dpF(V,W')^2$ if $p\leq q$, \\	 $\max\limits_{W'\in\Omega_p^+(W)} \!\!\!\dpF(V,W')^2$ if $p>q$. \\  
		Min is $\sqrt{\max\{0,p-q\}}$ (fixed $p$, $q$), if $V\subset W$ or $W \subset V$.}
	and $d_{\mathrm{s}}$%
		\SELF{$= \max\{\vec{d}(V,W), \vec{d}(W,V)\} = \sqrt{\max(p,q)-\sum_{i,j} \inner{e_i,f_j}^2} = \frac{1}{\sqrt{2}} \sqrt{|p-q| + \|P_V-P_W\|_{\mathrm{F}}^2}$ \\
		For fixed $p$, $q$, min is $\sqrt{|p-q|}$, if $V\subset W$ or $W\subset V$. Max is $\sqrt{\max\{p,q\}}$, if $V \perp W$.}%
	\cite{Sun_2007,Wang2006,Zuccon2009}.
		\CITE{Bagherinia2011, Figueiredo2010, Sharafuddin2010.}

	\item The \emph{containment gap} $\delta(V,W)$ \cite{Beattie2005,Sorensen2002}
		\SELFR{$\delta(U,W) = \max\limits_{\|u\|=1} \|u-P_W u\| \leq \max\limits_{\|u\|=1} \|u-P_W P_V u\| \leq \max\limits_{\|u\|=1} (\|u-P_V u\| + \|P_V u-P_W P_V u\|) \leq \delta(U,V) + \max\limits_{\|u\|=1} \|\frac{P_V u}{\|P_V u\|}-P_W \frac{P_V u}{\|P_V u\|}\| \leq \delta(U,V) + \max\limits_{\|v\|=1} \|v-P_W v\| = \delta(U,V) + \delta(V,W)$. \\[6pt]
		$\delta(V,W) =0 \Leftrightarrow V\subset W$. Max occurs when $V\pperp W$.}
	is an asymmetric metric extending $\dpd$,
	and its max-symmetrized metric is the \emph{gap} $\hat{\delta}$ \cite{Kato1995}.%
		\CITE{Stewart1990 não, só equal dim} 
		\SELF{$\|T\|=\sup\{\frac{\|Tx\|}{\|x\|}\}$. To prove let $p=q$ so $\|(P_V - P_W)e_i\| = \|e_i - P_W e_i\|\leq \delta(V,W)$ and $\|(P_V - P_W)e_i^\perp\| = \|P_W e_i^\perp\| = \sin\theta_i \leq \delta(V,W)$}
\end{itemize}

\begin{table}[]
	\centering
		\renewcommand{\arraystretch}{1}
		\begin{tabular}{ll}
			\toprule
			Distance & Formulas
			\\
			\cmidrule(lr){1-1} \cmidrule(lr){2-2}  
			$\mFS$ & $\cos^{-1}(\prod_{i=1}^m \cos\theta_i)$
			\\[3pt]
			$\MFS$ & $\begin{cases}
				\cos^{-1}(\prod_{i=1}^p \cos\theta_i) &\text{ if } p=q, \\
				\frac\pi2 &\text{ if } p\neq q
			\end{cases}$
			\\[3pt]
			$\mpF$ & $\sqrt{\sum_{i=1}^m \sin^2 \theta_i} = \sqrt{m - \|\mathbf{A}^\dagger \mathbf{B}\|_{\mathrm{F}}^2}$ 
			\\[3pt]
			$\tilde{d}_{\mathrm{pF}}$ & $\sqrt{\frac{|p-q|}{2}+\sum_{i=1}^m \sin^2 \theta_i} 
			= \sqrt{\frac{p+q}{2} - \|\mathbf{A}^\dagger \mathbf{B}\|_{\mathrm{F}}^2}
			= \frac{\|P_V-P_W\|_{\mathrm{F}}}{\sqrt{2}}$ 
			\\[3pt]
			$\vec{d}(V,W)$ & $\sqrt{\max\{0,p-q\} + \sum_{i=1}^m \sin^2 \theta_i} 
			= \sqrt{p - \|\mathbf{A}^\dagger \mathbf{B}\|_{\mathrm{F}}^2} 
			= \sqrt{\sum_{i=1}^p \|e_i-P_W e_i\|^2}$
			\\[3pt]
			$d_{\mathrm{s}}$ & $\sqrt{|p-q|+\sum_{i=1}^m \sin^2 \theta_i} 
			= \sqrt{\max\{p,q\} - \|\mathbf{A}^\dagger \mathbf{B}\|_{\mathrm{F}}^2}
			= \max\{\vec{d}(V,W),\vec{d}(W,V)\}$ 
			\\
			$\delta(V,W)$ & $\max\limits_{v\in V, \|v\|=1} \|v-P_W v\| = \begin{cases}
				\sin \theta_p &\text{ if } p \leq q, \\
				1 &\text{ if } p > q
			\end{cases}$
			\\
			$\hat{\delta}$ & $\max\{\delta(V,W),\delta(W,V)\} = \|P_V - P_W\|_2 = 
			\begin{cases}
				\sin \theta_p &\text{ if } p = q, \\
				1 &\text{ if } p \neq q
			\end{cases}$
			\\
			\bottomrule
		\end{tabular}
	\caption{Distances between $V_{(p)}, W_{(q)} \in \Gr(n)$, with $m=\min\{p,q\}\neq 0$ and principal angles $\theta_1\leq\cdots\leq\theta_m$. Also, $(e_1,\ldots,e_p)$ is a principal basis of $V$ \wrt $W$, while $\mathbf{A}$ and $\mathbf{B}$ are matrices with orthonormal bases of $V$ and $W$ as columns.}
	\label{tab:distances Total}
\end{table}

Similar extensions can be defined for other metrics.
Anyway, these distances have some shortcommings:
\begin{itemize}
	\item For $p\neq q$, $\MFS$ and $\hat{\delta}$ are constant, giving no information about the separation of subspaces, while $\tilde{d}_{\mathrm{pF}}$ and $d_{\mathrm{s}}$ have a non-zero minimum, making it harder to detect when one subspace is almost contained in the other \cite{Gruber2009,Renard2018}, specially if dimensions are not known beforehand.
	These metrics give the disjoint union topology.

	\item\label{pg:triangle fail}
	The triangle inequality fails for $\mFS$ and $\mpF$ 
	(e.g., for lines $K \neq L$ in a plane $V$, $\mFS(K,V) + \mFS(V,L) = 0 < \mFS(K,L)$),
	and is unknown%
	\footnote{A counterexample in \cite{Karami2023} is for the distance of \cite{Wang2015}, without orthonormality.}
	for $\vec{d}$.
	This reduces their usefulness, and their balls are not a basis for a topology: e.g.,
	if $V$ has a line $J \notin O = B_r(K) \cap B_r(L)$ then $V\in O$ but $B_\epsilon(V) \not\subset O$ for any $\epsilon >0$, as  $J\in B_\epsilon(V)$.

	\item $\delta$ and $\hat{\delta}$ carry little data (only $\theta_p$), which is not enough for all uses.
\end{itemize}
Also, $\delta$ and $\vec{d}$ are asymmetric, which we do not see as a problem.

The \emph{Infinite Grassmannian} $\Gr_p(\infty)$ of $p$-subspaces in all $\F^n$'s,
and the \emph{Infinite Total Grassmannian}%
\footnote{Called \emph{Doubly Infinite Grassmannian} in \cite{Ye2016}.}
$\Gr(\infty)$ of all subspaces in all $\F^n$'s, are defined \cite{Ye2016}, using the natural inclusion
$\Gr_p(n) \subset \Gr_p(n+1)$, as \label{df:Gr infty}
\begin{equation*}
	\Gr_p(\infty) = \bigcup_{n=0}^\infty \Gr_p(n), \quad \text{ and } \quad \Gr(\infty) = \bigcup_{n=0}^\infty \Gr(n) = \bigcup_{p=0}^\infty \Gr_p(\infty).
\end{equation*}

A metric in  $\Gr_p(\infty)$ is given by a family of metrics $d_p$ in the $\Gr_p(n)$'s that do not depend on $n$.
In \cite{Ye2016}, it is extended to $V_{(p)}, W_{(q)}\in \Gr(\infty)$, with $p<q$,
via an inclusion of $\theta_{p+1} = \cdots = \theta_q = \frac\pi2$,
	\SELF{Can skip $\delta$ and include $\theta$'s in $d_p$, both have same formula}
which turns
$\min\{d_p(V,U):U\in \Gr_p(W)\} = \min\{d_q(Y,W):Y\in \Gr_q(\infty), Y\supset V\}$ 
	\SELF{Minimal distances occur at $W_p$ and $V\oplus W_\perp$, which have with $V$ the same nonzero $\theta_i$'s as $W$}
into a metric $d^*(V,W) = \max\{d_q(Y,W):Y\in \Gr_q(\infty), Y\supset V\}$.
	\SELF{If a $q$-subspace can contain $V$ and $q-p$ extra dimensions in $W^\perp$, what explains use of $\Gr_q(\infty)$}
But $d^*$, with $*$ as in the notation of \cite{Ye2016}, still has issues:
\begin{itemize}
	\item For 
	$*=\alpha,\beta,\phi,\pi,\sigma$ it is constant when $p\neq q$. In fact,
	$d^\pi = \hat{\delta}$ and $d^\phi = \MFS$ (\!\!\cite[p.\,1189]{Ye2016} has a small mistake, as $c_\phi = \frac\pi2$, not $1$).
	
	\item For $* = \kappa,\rho$, or no $*$ (their Grassmann metric), it has a non-zero minimum when $p\neq q$. In fact, $d^\kappa = d_{\mathrm{s}}$.
	
	\item For $*=\mu$, it is not a metric, as it is based on the Martin distance, which does not satisfy a triangle inequality (see \Cref{sc:Metrics and distances on Grassmannians}).
\end{itemize}

Some problems seen above are due to the attempt to force a symmetric metric on a space whose elements have inherently asymmetric relations.

\section{Asymmetric metrics in $\Gr(n)$}\label{sc:Other asymmetric metrics}

Under reasonable conditions, we extend Grassmannian metrics naturally to an asymmetric metric in $\Gr(\infty)$,
which then restricts to $\Gr(n)$.

\begin{definition}\label{df:regular family}
	Let $\{d_p\}_{p\in\N}$ be family of metrics in the $\Gr_p(\infty)$'s which is \emph{regular} in the sense that, for $0 \neq p < q$:
	\begin{enumerate}[(i)]
		\item $d_p(V,W) = \Phi_p(\theta_1,\ldots,\theta_p)$ for a nondecreasing function $\Phi_p$ of the principal angles $\theta_1 \leq\cdots\leq\theta_p$ of $V,W\in \Gr_p(\infty)$; \label{it:fp nondecr}
		\item $\Phi_q(0,\ldots,0,\theta_1,\ldots,\theta_p) = \Phi_p(\theta_1,\ldots,\theta_p)$. \label{it:Phi zeros} 
	\end{enumerate}
\end{definition}

The metrics of \Cref{sc:Metrics and distances on Grassmannians} do not depend on $n$, so they extend to $\Gr_p(\infty)$, and it is easy to check that (\ref{it:fp nondecr}) and (\ref{it:Phi zeros}) hold.


\begin{definition}\label{df:diameter}
	$\Delta_p = \diam \Gr_p(\infty) = \sup\{d_p(V,U):V,U\in \Gr_p(\infty)\}$.
\end{definition}

The last column of \Cref{tab:asymmetric metrics} has $\Delta_p$ for the metrics of \Cref{sc:Metrics and distances on Grassmannians}.
When necessary, we write $\Delta_{p,d}$ to indicate the metric $d$ used.

\begin{lemma}\label{pr:diameter nondecr}
	$\Delta_p = \Phi_p(\frac\pi2,\ldots,\frac\pi2)$ is a non-decreasing function of $p$.%
	\SELF{Pode fazer com $\Gr_p^N$, $N\geq 2p$, mas parece menos natural}
\end{lemma}
\begin{proof}
	Let $0 < p<q$. Since $\Gr_p(\infty)$ has orthogonal subspaces, $\Delta_p = \Phi_p(\frac\pi2,\ldots,\frac\pi2) = \Phi_q(0,\ldots,0,\frac\pi2,\ldots,\frac\pi2) \leq \Phi_q(\frac\pi2,\ldots,\frac\pi2) = \Delta_q$.
\end{proof}

We use $\Gr_p(\infty)$ as $p$-subspaces of $\F^n$ intersect non-trivially for $p > \frac n2$, so $\diam \Gr_p(n) = \Phi_p(0,\ldots,0,\frac\pi2,\ldots,\frac\pi2)$, with $2p-n$ zeros, would decrease for the $\ell^2$ metrics.

\begin{lemma}\label{pr:theta' theta}
	Let $V\in \Gr_p(n)$ have principal angles $\theta_1\leq\cdots\leq\theta_{\min\{p,q\}}$ with $W\in \Gr_q(n)$, 
	$\theta_1'\leq\cdots\leq\theta_{\min\{p,q'\}}'$ with $W' \in \Gr_{q'}(W)$,
	and set $\theta_i = \frac\pi2$ for $i>\min\{p,q\}$.
	Then $\theta_i \leq \theta_i' \leq \theta_{i+q-q'}$ $\forall\, i \leq \min\{p,q'\}$,
	and:
	\begin{enumerate}[(i)]
	\item\label{it:theta V theta W'}%
	For $p\leq q'$,  $P_W(V) \subset W' \,\Leftrightarrow\, \theta_i = \theta_i'$  $\forall\, i \leq p$.
	
	\item\label{it:theta V' theta W}%
		\SELF{$\theta_{j+q-q'} = 
			\begin{cases}
				0 &\text{if } j\leq 0, \\
				\theta'_j &\text{if } 0 < j \leq \min\{p,q'\}.
			\end{cases}$}%
	For $p \geq q-q'$,\\
	$W'^\perp \cap W \subset V \Leftrightarrow \theta_i = 
	\begin{cases}
		0 &\text{if } i\leq q-q', \\
		\theta'_{i-q+q'} &\text{if } q-q' < i \leq \min\{p+q-q',q\}.
	\end{cases}$
	\end{enumerate}
\end{lemma}
\begin{proof}\OMITL{p. \pageref{rm:detailed proof} has detailed proof}
	The inequalities follow from the usual interlacing ones of singular values \cite{Horn1991}.
		\SELF{Cauchy interlacing adapted \\ Horn1991 p. 149 Cor. 3.1.3}
	(\ref{it:theta V theta W'}) Both sides of $\Leftrightarrow$ are equivalent to $W'= \Span\{f_1,\ldots,f_{q'}\}$ for a principal basis $(f_1,\ldots,f_q)$ of $W$ \wrt $V$.
	(\ref{it:theta V' theta W}) Likewise, but with $W'^\perp \cap W = \Span\{f_1,\ldots, f_{q-q'}\} \subset V$ and $W'= \Span\{f_{q-q'+1},\ldots,f_{q}\}$.
\end{proof}

\begin{definition}\label{df:PO}
	A \emph{projection subspace} of $W\in \Gr_q(n)$ \wrt $V\in \Gr_p(n)$ is any $W' \in \Gr_{\min\{p,q\}}(W)$ such that $P_W(V) \subset W'$.
	The set of such subspaces is denoted by $\PP_W(V)$.
		\SELFR{If $V\not\pperp W$ then $W_P=P_W(V)$ and $W_\perp = V^\perp \cap W$, otherwise $W_P \supset P_W(V)$ and $W_\perp \subset V^\perp \cap W$, strict inclusions if $p \leq q$} 
\end{definition}

Principal angles of $V$ with $W' \in \PP_W(V)$ are the same  
	\OMIT{None, if $p$ or $q$ is 0}
as with $W$.%
	\SELF{and the same nonzero principal angles as $V\oplus W_\perp$ and $W$.}
And $P_W(V) \in \PP_W(V) \Leftrightarrow \dim P_W(V) = \min\{p,q\} \Leftrightarrow V \not\pperp W$ or $P_W(V)=W$.

\begin{lemma}\label{pr:d subspace smaller}
	Let $V,W\in \Gr_p(\infty)$, $V' \in \Gr_r(V)$ and $W' \in \PP_W(V')$. Then $d_r(V',W') \leq d_p(V,W)$.
\end{lemma}
\begin{proof}
	If $V$ and $W$ have principal angles $\theta_1 \leq\cdots\leq\theta_p$, 
	and $V'$ and $W'_{(r)}$ have $\theta_1' \leq\cdots\leq\theta_r'$ (same as $V'$ and $W$),
	 \Cref{pr:theta' theta} gives
	$d_r(V',W') = \Phi_r(\theta_1',\ldots,\theta_r') = \Phi_p(0,\ldots,0,\theta_1',\ldots,\theta_r') \leq \Phi_p(\theta_1,\ldots,\theta_p) = d_p(V,W)$.
\end{proof}

We are ready to extend $\{d_p\}$, but first let us explain the idea. 

Let $V_{(p)}, W_{(q)}\in \Gr(\infty)$ and $K=\{d_p(V,U):U\in \Gr_p(W)\}$.
If $p\leq q$ then $K$ is compact and nonempty, and it is usual to set $d(V,W) =  \min K$.
But if $p>q$ then $\Gr_p(W) = \emptyset$, and $\min K = \min \emptyset$ is undefined.
The common solution is to swap $V$ and $W$ and set $d(V,W) = d(W,V)$, but it breaks the triangle inequality (as in $\mpF$ and $\mFS$).

Instead, we take $d(V,W) = \inf_{[0,\Delta_p]} K$, the infimum of $K$ in the ordered set $([0,\Delta_p],\leq)$,
which is natural as it is the smallest interval with all possible values of $d_p(V,U)$.
If $p\leq q$, this infimum gives $\min K$ as usual,
but if $p>q$ then%
	\footnote{Recall \cite[p.\,261]{Bajnok2020} that $\inf_\mathcal{O} \emptyset = M$ in an ordered set $\mathcal{O}$ with greatest element $M$, as the infimum of a subset is its greatest lower bound in $\mathcal{O}$, and $\emptyset$ has no element smaller than $M$.}
$d(V,W) = \inf_{[0,\Delta_p]} \emptyset = \Delta_p$.
This breaks the symmetry, but preserves the (oriented) triangle inequality, as we show.

\begin{theorem}\label{pr:asym metrics}
	Given a regular family of metrics $\{d_p\}_{p\in\N}$ in the $\Gr_p(\infty)$'s, let, for $V_{(p)}, W_{(q)} \in \Gr(\infty)$, 
	\begin{equation}\label{eq:asym metric}
		d(V,W) = \inf\nolimits_{[0,\Delta_p]}\{d_p(V,U):U\in \Gr_p(W)\} \in [0,\Delta_p].
	\end{equation}
	Then $d$ is an asymmetric metric in $\Gr(\infty)$ 
	and, for any $W' \in \PP_W(V)$,%
		\SELF{For $\ell^2$ metrics, $d(V,W) = d(V,W')$ for $W' \in \Gr_p(W) \Leftrightarrow W'$ is projection subspace. For $\wedge$ and  max ones it also happens when $V\pperp W$}%
	\begin{equation}\label{eq:d cases}
		d(V,W) = d(V,W') = \begin{cases}
			d_p(V,W') = \Phi_p(\theta_1,\ldots,\theta_p)  &\text{if } 0<p\leq q,\\
			\Delta_p &\text{otherwise},
		\end{cases}
	\end{equation}
	where $\theta_1 \leq\cdots\leq\theta_p$ are the principal angles of $V$ and $W$.
\end{theorem}
\begin{proof}
	First we prove \eqref{eq:d cases}.
	If $p=0$ then $W'=0$ and $d(V,W) = 0 = \Delta_0$.
	If  $p>q$ then $W' = W$ and $d(V,W) = \Delta_p$, as seen above.
	If $0<p\leq q$ then $V$ has the same $\theta_i$'s with $W'_{(p)}$, so
	$d_p(V,W') = \Phi_p(\theta_1,\ldots,\theta_p) \leq d_p(V,U)$ for all $U\in \Gr_p(W)$, by \Cref{pr:theta' theta}, and thus $d(V,W) = d_p(V,W')$.
	
	\Cref{df:asymmetric metric}(\ref{it:T0}) holds since, by \eqref{eq:d cases},
	$d(V,W)=0 \Rightarrow d_p(V,W')=0 \Rightarrow V=W' \subset W$,
	and likewise $d(W,V)=0 \Rightarrow W \subset V$. 
		\OMIT{$p=0$ trivial, \\ $p\neq 0$ has $\Delta_p \neq 0$}
	
	For \Cref{df:asymmetric metric}(\ref{it:oriented triang ineq}), 
	we must show $d(V,W) \leq d(V,U) + d(U,W)$
	for $U_{(r)}, V_{(p)}, W_{(q)} \in \Gr(\infty)$. 
	If $r<p$ then $d(V,W) \leq \Delta_p = d(V,U)$, by \eqref{eq:asym metric} and \eqref{eq:d cases}.
	If $r\geq p$ and $r>q$ then $d(V,W) \leq \Delta_p \leq \Delta_r = d(U,W)$, by \eqref{eq:asym metric}, \Cref{pr:diameter nondecr} and \eqref{eq:d cases}. 
	If $p=0$ it is trivial.
	If $0<p\leq r \leq q$,
	taking $U'_{(p)} \in \PP_U(V)$, 
	$W'_{(r)} \in \PP_W(U)$
	and $W''_{(p)} \in \PP_{W'}(U')$,
	we obtain
	\begin{align*}
		d(V,W) &\leq d_p(V,W'') \leq d_p(V,U') + d_p(U',W'') \\
		&\leq d_p(V,U') + d_r(U,W') = d(V,U) + d(U,W), 
	\end{align*}
	using \eqref{eq:asym metric}, the triangle inequality of $d_p$, \Cref{pr:d subspace smaller} and \eqref{eq:d cases}.
\end{proof}

\begin{table}[]
	\centering
	\renewcommand{\arraystretch}{1}
	\begin{tabular}{llll}
		\toprule
		Type & Symbol & Formula for $0<p\leq q$ & Value for $p>q$ ($\Delta_p$)
		\\
		\cmidrule(lr){1-1} \cmidrule(lr){2-2}  \cmidrule(lr){3-3} \cmidrule(lr){4-4}
		$\ell^2$ & $\dg$ & $\sqrt{\sum_{i=1}^p \theta_i^2}$  & $\frac\pi2 \sqrt{p}$
		\\[3pt]
		& $\dcF$ & $2\sqrt{\sum_{i=1}^p \sin^2 \frac{\theta_i}{2}}$ & $\sqrt{2p}$
		\\[3pt] 
		& $\dpF$ & $\sqrt{\sum_{i=1}^p \sin^2 \theta_i}$ & $\sqrt{p}$
		\\[6pt]
		$\wedge$ & $\dFS$ & $\cos^{-1}(\prod_{i=1}^p \cos\theta_i)$ & $\frac\pi2$
		\\[3pt]
		& $\dcw$ & $\sqrt{2-2\prod_{i=1}^p \cos\theta_i}$ & $\sqrt{2}$
		\\[3pt] 
		& $\dBC$ & $\sqrt{1-\prod_{i=1}^p \cos^2\theta_i}$ & $1$
		\\[6pt] 
		max & $\dA$ & $\theta_p$ & $\frac\pi2$
		\\[3pt] 
		& $\dcd$ & $2\sin\frac{\theta_p}{2}$ & $\sqrt{2}$
		\\[3pt]
		& $\dpd$ & $\sin\theta_p$ & $1$
		\\
		\bottomrule
	\end{tabular}
	\caption{Asymmetric metrics $d(V,W)$ for $V_{(p)}, W_{(q)} \in \Gr(n)$, with $p \neq 0$, in terms of their principal angles $\theta_1\leq\cdots\leq\theta_{\min\{p,q\}}$. If $p=0$ all distances are $d(V,W) = \Delta_0 = 0$.}
	\label{tab:asymmetric metrics}
\end{table}

Naturally, $d$ restricts to $\Gr(n)$.
\Cref{tab:asymmetric metrics} has the asymmetric metric extending each metric of \Cref{sc:Metrics and distances on Grassmannians}.
We use the same symbol for both.
The asymmetric $\dpd$ is the containment gap $\delta$ of \Cref{tab:distances Total}.
Note that $d(V,W) = d(V,W')$ for $W' \in \PP_W(V)$ does not mean $d(V,W) = d(V,P_W(V))$.

 \begin{example}
 	For the \emph{asymmetric Fubini-Study metric} $\dFS$ in \Cref{tab:asymmetric metrics}, the distance from $V_{(p)}$ to $W_{(q)}$ in $\Gr(n)$ is
 	\begin{equation}\label{eq:adFS}
 		\dFS(V,W) = 
 		\begin{cases}
 			0 &\text{ if } p=0, \\
 			\cos^{-1}(\prod_{i=1}^p \cos\theta_i) &\text{ if } 0 < p \leq q, \\
 			\frac\pi2 &\text{ if } p > q.
 		\end{cases}
 	\end{equation}
 	A comparison shows $\mFS$ \eqref{eq:mFS} and $\MFS$ \eqref{eq:dFS}, in \Cref{sc:Fubini-Study metric}, are its min- and max-\-sym\-met\-riz\-a\-tions, 
 	and if $p=q$ they all give the usual $\dFS$ \eqref{eq:usual FS}.
 	
 	At first, their differences may seem dull: the conditions for each case.
 	But those in \eqref{eq:adFS} reflect what most affects relations between subspaces: which one is larger.
 	This leads to better properties (see \Cref{sc:Asymmetric Fubini-Study}), with less restrictions or cases depending on the dimensions.
 	
 	For example, $\mFS$ does not satisfy a triangle inequality nor gives $\Gr(n)$ a topology;
 	$\MFS$ is a metric with a very restrictive topology;
 	and $\dFS$ is an asymmetric metric with more useful topologies (see \Cref{sc:Asymmetric geometry}).
 	Also, $\dFS$ is more informative when $p \neq q$: the constant $\MFS(V,W) = \frac\pi2$ does not show how close $V$ and $W$ are;
 	a small $\mFS(V,W)$ shows one is almost contained in the other, but not which one;
 	and a small $\dFS(V,W)$ with $\dFS(W,V)=\frac\pi2$ shows $V$ is almost contained in a strictly larger $W$.
 \end{example}

 \subsection{Basic properties}
 
 Here we obtain some properties for any $d$ given by \eqref{eq:asym metric}, and some which are more specific, for later use.

 \begin{proposition}\label{pr:d min max}
 	Let $V_{(p)}, W_{(q)} \in \Gr(n)$.
 	\SELF{In $\Gr(n)$, $V \perp W$ requires $p+q \leq n$}
 	\begin{enumerate}[(i)]
 		\item $d(V,W)=0 \Leftrightarrow V \subset W$. \label{it:d=0}
 		\item $d(V,W) = \Delta_p \Leftrightarrow
 		\begin{cases}
 			V \perp W \text{ or } p>q, \text{ for asymmetric $\ell^2$ metrics,} \\
 			V \pperp W \text{ or } p=0, \text{ for $\wedge$ or max ones.}
 		\end{cases}$ \label{it:d=Delta}
 	\end{enumerate}
 \end{proposition}
 \begin{proof}
 	Follows from \eqref{eq:d cases} 
 	\OMIT{For \ref{it:d=0}}
 	and the formulas in \Cref{tab:asymmetric metrics}.
 \end{proof}
 
By (\ref{it:d=0}), $d$ induces the partial order $\subset$,
so $d(V,W)$ measures how far $V$ is from being contained in $W$. 
If $p>q$ this is never any closer to happening, and $d(V,W)$ remains constant at its maximum $\Delta_p$ (for a given $p$).

\begin{proposition}\label{pr:d subspaces}
	$d(V',W) \leq d(V,W) \leq d(V,W')$, for $V,W,V',W' \in \Gr(n)$ with $V'_{(p')} \subset V_{(p)}$ and $W'_{(q')} \subset W_{(q)}$.
\end{proposition}	
\begin{proof}
	As $\Gr_p(W')\subset \Gr_p(W)$, \eqref{eq:asym metric} gives the second inequality.
		\OMIT{$d(V,W) = \inf \{d_p(V,U):U\in \Gr_p(W)\} \leq \inf \{d_p(V,U):U\in \Gr_p(W')\} = d(V,W')$}		
	If $p>q$, $d(V',W) \leq \Delta_{p'} \leq \Delta_p = d(V,W)$
	by \eqref{eq:asym metric}, \Cref{pr:diameter nondecr} and \eqref{eq:d cases}.
		\OMIT{\ref{pr:asym metrics}}
	If $p\leq q$, taking $U_{(p)} \in \PP_W(V)$ and $U'_{(p')} \in \PP_{U}(V')$
	we find
	$d(V',W) \leq d_{p'}(V',U') \leq d_p(V,U) = d(V,W)$
	by \eqref{eq:asym metric}, \Cref{pr:d subspace smaller} and \eqref{eq:d cases}.
\end{proof}	

So, moving $V$ into $W$ is easier ($d(V,W)$ is smaller) the smaller $V$ is, or the larger $W$ is.
We give equality conditions for $\dFS$ and $\dg$ (see \Cref{tab:asymmetric metrics}):

\begin{proposition}\label{pr:equalities subspaces} 
	For $V,W, V', W'$ as above:%
	\SELF{$\dFS(V,W) = \dFS(V,W_P) =$ \\ $\dFS(V\oplus W_\perp,W) = \min\{\dFS(V,U):U \in \Gr(W)\} =$ \\ $\min\{\dFS(Y,W):Y \in \Gr(n), Y\supset V\}$.}
	\begin{enumerate}[(i)]
		\item $\dFS(V,W) = \dFS(V,W') \Leftrightarrow V\pperp W$ or $P_W(V)\subset W'$.\label{it:W' sub W}
		
		\item $\dFS(V,W) = \dFS(V',W) \Leftrightarrow V'\pperp W$ or $V'^\perp \cap V \subset W$.\label{it:V' sub V}	
		
		\item $\dg(V,W) = \dg(V,W') \Leftrightarrow 
		\begin{cases}
			p\leq q' \text{ and } P_W(V)\subset W'; \text{ or } \\ 
			p>q', \text{ with } V\perp W \text{ if } p\leq q. 
		\end{cases}$ \label{it:dg W' sub W}
		
		\item $\dg(V,W) = \dg(V',W) \Leftrightarrow 
		\begin{cases}
			p\leq q \text{ and } V'^\perp \cap V \subset W; \text{ or } \\ 
			p>q \text{ and } V' = V.
		\end{cases}$\label{it:dg V' sub V}	
	\end{enumerate}
\end{proposition}
\begin{proof}
	(\ref{it:W' sub W}) 
	Follows from \Cref{pr:d min max}, unless $0<p\leq q' \leq q$.
	If $V$ has principal angles $\theta_1\leq\cdots\leq\theta_p$ with $W$,
	and $\theta_1'\leq\cdots\leq\theta_p'$ with $W'$,
	\Cref{pr:theta' theta} gives
	$\theta_i \leq \theta_i' \,\forall i$.
	So
	$\prod_{i=1}^p \cos \theta_i = \prod_{i=1}^p \cos \theta'_i \Leftrightarrow \theta_p=\frac\pi2$ 
		\OMIT{$=\theta'_p$}
	or $\theta_i = \theta_i' \,\forall i$, which corresponds to (\ref{it:W' sub W}), by \Cref{pr:theta' theta}(\ref{it:theta V theta W'}).
	
	(\ref{it:V' sub V})
	Follows from \Cref{pr:d min max}, unless  $0< p' \leq p \leq q$.
	If $W$ has principal angles $\theta_1\leq\cdots\leq\theta_p$ with $V$,
	and $\theta_1'\leq\cdots\leq\theta_{p'}'$ with $V'$,
	\Cref{pr:theta' theta} (swapping $V \leftrightarrow W$, $V' \leftrightarrow W'$) gives
	$\theta'_i \leq \theta_{i+p-p'} \,\forall i \leq p'$.
	So,
	\[ \prod_{i=1}^p \cos \theta_i  = \prod_{i=1}^{p'} \cos \theta'_i \,\Leftrightarrow\, \theta'_{p'} = \frac\pi2 \text{ or } \theta_i = 
	\begin{cases}
		0 &\text{if } i\leq p-p', \\
		\theta'_{i-p+p'} &\text{if } p-p' < i \leq p,
	\end{cases} \]
	which corresponds to (\ref{it:V' sub V}), by \Cref{pr:theta' theta}(\ref{it:theta V' theta W}).

	(\ref{it:dg W' sub W})
	Follows from \eqref{eq:d cases} if $p=0$ or $p>q \geq q'$.
	If $0 < p \leq q' \leq q$, 
	$V$ has principal angles $\theta_1\leq\cdots\leq\theta_p$ with $W$, 
	and $\theta_1'\leq\cdots\leq\theta_p'$ with $W'$,
	so 
	\[\sum_{i=1}^p \theta_i^2 = \dg^2(V,W) = \dg^2(V,W') = \sum_{i=1}^p \theta_i'^2  
	\,\Leftrightarrow\, \theta_i = \theta_i' \ \forall i \,\Leftrightarrow\, P_W(V) \subset W',\]
	by \Cref{pr:theta' theta}.
		\OMIT{\ref{it:theta V theta W'}}
	If $q'<p\leq q$, \eqref{eq:d cases} and $\Delta_{p,\dg} = \frac{\pi\sqrt{p}}{2}$ (see \Cref{tab:asymmetric metrics}) give
	\[ \sum_{i=1}^p \theta_i^2 = \dg^2(V,W)  = \dg^2(V,W') = \frac{\pi^2 p}{4} 
	\,\Leftrightarrow\, \theta_i=\frac\pi2 \,\forall i \,\Leftrightarrow\, V \perp W. \]

	(\ref{it:dg V' sub V})
	Follows from \Cref{pr:d min max} 
		\OMIT{\ref{it:d=0}}
	if $p'=0$.
	If $0 < p' \leq p \leq q$, $W$ has principal angles $\theta_1\leq\cdots\leq\theta_p$ with $V$, and $\theta_1'\leq\cdots\leq\theta_{p'}'$ with $V'$, 
	so \Cref{pr:theta' theta} (swapping $V \leftrightarrow W$, $V' \leftrightarrow W'$)
		\OMIT{\ref{it:theta V' theta W}}
	gives $\theta'_i \leq \theta_{i+p-p'} \,\forall i \leq p'$ and
	\begin{equation*}
		\sum_{i=1}^p \theta_i^2 = \sum_{i=1}^{p'} \theta_i'^2 \Leftrightarrow
		\theta_i = 
		\begin{cases}
			0 &\text{if } i\leq p-p', \\
			\theta'_{i-p+p'} &\text{if } p-p' < i \leq p
		\end{cases}
		\Leftrightarrow V'^\perp \cap V \subset W.
	\end{equation*}
	If $p>q$ we have $\frac{\pi^2 p}{4} = \dg^2(V,W) = \dg^2(V',W) \leq \frac{\pi^2 p'}{4} \Leftrightarrow p=p'$.
\end{proof}

\begin{corollary}\label{pr:proj subsp}
	For $V,W, W'$ as above, with $q' = \min\{p,q\}$:
	\begin{enumerate}[(i)]
		\item $\dFS(V,W) = \dFS(V,W') \Leftrightarrow W' \in \PP_W(V)$ or $V \pperp W$. \label{it:proj subsp dFS}
		\item $\dg(V,W) = \dg(V,W') \Leftrightarrow W' \in \PP_W(V)$.\label{it:proj subsp dg}
	\end{enumerate}
\end{corollary}
\begin{proof}
	(\ref{it:proj subsp dFS}) Follows from \Cref{pr:equalities subspaces}(\ref{it:W' sub W}) and \Cref{df:PO}.
	(\ref{it:proj subsp dg},\,$\Rightarrow$) 
	Follows from \Cref{pr:equalities subspaces}(\ref{it:dg W' sub W}) if $p\leq q$, and otherwise $W'=W$.
	(\ref{it:proj subsp dg},\,$\Leftarrow$) Immediate from \eqref{eq:d cases}.
\end{proof}

\section{Asymmetric geometry of $\Gr(n)$}\label{sc:Asymmetric geometry} 

Throughout this section, $d$ is one of the asymmetric metrics of \Cref{tab:asymmetric metrics}.
We refer to its type as $\ell^2$, $\wedge$ or max cases.

\begin{proposition}\label{pr:extended inequalities}
	 For distances from $V_{(p)}$ to $W_{(q)}$ in $\Gr(n)$, with $V \not\subset W$:%
	 \SELF{otherwise all $d=0$}
	\begin{enumerate}[(i)]
		\item The inequalities in \Cref{pr:ineqs1} hold.
		\item Those in \Cref{pr:ineqs2} hold if $\dim(V\cap W) +2 \leq p \leq q$. 
		All $>$'s become $=$ if $\dim(V\cap W)+1 = p \leq q$ or $(p,q) = (1,0)$.
		If $1\neq p > q$, the middle $>$ in each line holds, and the other $>$ and $\geq$ become $=$.\label{it:new ineqs2}
	\end{enumerate}
\end{proposition}
\begin{proof}
	If $p\leq q$ then $d(V,W) = d(V,W')$ for $W'_{(p)} \in \PP_W(V)$, and we can use the original Propositions, noting that $V \neq W' \Leftrightarrow V \not\subset W$, and  $\dim(V\cap W) = \dim(V\cap W')$. If $p>q$ then $d(V,W) = \Delta_{p,d}$, and the result follows by comparing the values in \Cref{tab:asymmetric metrics}.
\end{proof}

This implies the metrics in \Cref{tab:asymmetric metrics} are topologically equivalent, giving $\Gr(n)$ the same natural  topologies $\tau^-$, $\tau^+$ and $\tau$.
These reflect the relations $\subset$, $\supset$ and $=$, respectively, since the formulas in \Cref{tab:asymmetric metrics} show that, though the precise balls \eqref{eq:balls} depend on $d$, for small $r>0$ (Fig.\,\ref{fig:backward balls}):
\SELF{$r<\Delta_p$, otherwise balls have all subspaces of all dim}
\begin{itemize}
	\item $B^-_r(V)$ has $q$-subspaces with $q\leq p$ and almost contained in $V_{(p)}$ (\ie forming small principal angles with $V$);
	\item $B^+_r(V)$ has those with $q \geq p$ and almost containing $V$;
	\SELF{Unions of Schubert varieties of {Ye2016}.
		Closed balls (in sense of $\leq r$, not topological) are $B^-_0[V] = \{U\in \Gr(n):U\subset V\} = \Gr(V)$,
		$B^+_0[V] = \{U\in \Gr(n):U\supset V\}$ and $B_0[V]=\{V\}$}
	\item $B_r(V)$ has those with $q=p$ and almost equal to $V$.
\end{itemize}
We write $\Gr^{\pm}(n)$ if the topology is $\tau^\pm$, $\Gr(n)$ if it is $\tau$ or does not matter.\label{df:Gr pm}

\begin{figure}[t]
	\centering
	\begin{subfigure}[b]{0.46\textwidth}
		\includegraphics[width=\textwidth]{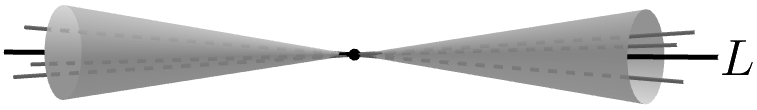}
		\caption{$B^-_r(L)$ has $0$ and lines inside a double cone around $L$. $B^+_r(L)$ would have $\R^3$, lines and planes crossing the cone interior, and $B_r(L)$ only lines.}
		\label{fig:neighborhood line}
	\end{subfigure}
	\qquad
	\begin{subfigure}[b]{0.45\textwidth}
		\includegraphics[width=\textwidth]{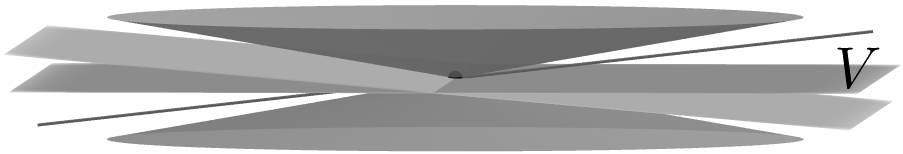}
		\caption{$B^-_r(V)$ has $0$, lines and planes outside a double cone near $V$. $B^+_r(V)$ would have $\R^3$ and planes outside the cone, and $B_r(V)$ only planes.}
		\label{fig:neighborhood plane}
	\end{subfigure}
	\caption{Backward balls of small radius $r$ around a line $L$ and a plane $V$ in $\Gr^-(\R^3)$}
	\label{fig:backward balls}
\end{figure}

The subspace topology 
	\OMIT{for $\tau$ or $\tau^\pm$}
of $\Gr_p(n)$ is its usual one, as in it $d$ is the original metric.
In $\tau$, the $\Gr_p(n)$'s are disconnected from each other (disjoint union topology),
as any small $B_r(V)$ is contained in a $\Gr_p(n)$. 
	\OMIT{since $\hat{d}(V,W) = \Delta_{\max\{p,q\}} \geq 1$ if $p = \dim V \neq \dim W = q$}
But they are deeply interconnected in $\tau^-$ (\resp $\tau^+$),
as $B_r^-(V)$ (\resp $B_r^+(V)$) stretches downwards (\resp upwards) to include smaller (\resp larger) subspaces,
and curves can cross from a $\Gr_p(n)$ to any other (Figs.\,\ref{fig:curve} and \ref{fig:total-grassmannian}).

\begin{figure}[t]
	\centering
	\includegraphics[width=.5\linewidth]{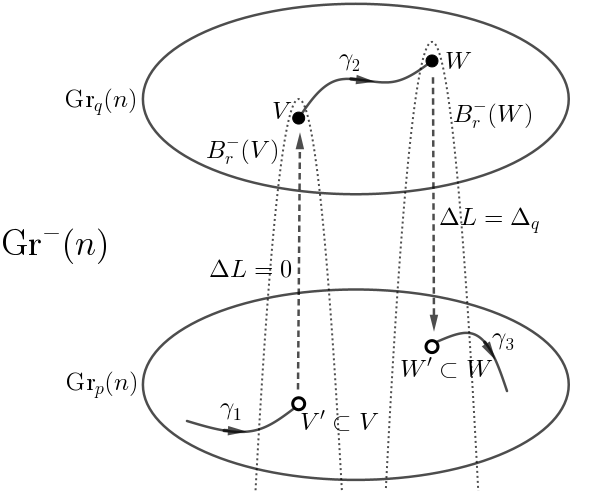}
	\caption{Balls $B_r^-(V)$ stretch downwards to include neighborhoods of $V'\subset V$, and likewise for W, so $\gamma_1$, $\gamma_2$ and $\gamma_3$ form a continuous curve.
	In $\Gr^+(n)$, balls stretch upwards, so closed dots would be below.
	Moving downwards causes a length change $\Delta L = d(W,W') = \Delta_q$, while upwards $\Delta L = d(V',V)=0$ (see \Cref{pr:length partition}).}
	\label{fig:curve}
\end{figure}

If $V\subset W$ then $V \in B_r^-(W)$ and $W \in B_r^+(V)$ for all $r>0$,
and so $\tau^\pm$ are not Hausdorff, but $T_0$, as expected.
These topologies require some attention.
In $\tau^-$, for example, one can easily check that:
\begin{itemize}
	\item any neighborhood of $V$ contains $\Gr(V)$;
	\item the closure of a unit set $\{V\}$ is $\{U \in \Gr(n): V \subset U\}$;
	\SELF{$\{U:U\subset V\}$ in $\Gr^{+}(n)$}
	\item $\Gr_0(n) = \{0\}$ is open, is in all open sets, and its closure is $\Gr(n)$; 
	\SELF{in $\tau^+$, $\Gr_n(n) = \{\F^n\}$ and $\bigcup_{q\leq p} \Gr_q(n)$}
	\item $\Gr_p(n)$ has empty interior if $p\neq 0$, and its closure is $\bigcup_{q\geq p} \Gr_q(n)$; 
	\item the smallest open set containing $\Gr_p(n)$ is $\bigcup_{q\leq p} \Gr_q(n)$.
\end{itemize}
One can obtain similar results for $\tau^+$ directly, or via \Cref{pr:Gr- = Gr+}.

\begin{proposition}
	$\Gr^{\pm}(n)$ are compact, contractible, path connected, simply connected, and Smyth bicomplete.%
	\SELF{$\dim$ is locally const in $\tau$, discontin in $\tau^\pm$.	In $\tau^\pm$, $d(V,W)$ is discontin in $V$ and $W$ (if $V \subsetneq W$, $d(V,V) = d(W,W) = 0$ and $d(W,V)\neq 0$, though $V \in B^-_r(W)$ and $W \in B^+_r(V) \ \forall r$)}
\end{proposition}
\begin{proof}
	Being finite unions of compact $\Gr_p(n)$'s, they are compact.
	Since $\Gr^{-}(n)$ (or $\Gr^{+}(n)$, using $[0,\frac12)$ and $[\frac12,1]$) has a continuous contraction
	\SELF{in $\tau^-/\tau^+$ strict ineq in smaller/larger subspace} 
	\begin{equation*}
		h(V,t) = \begin{cases}
			V &\text{if } t \in [0,\frac12], \\
			0 &\text{if } t \in (\frac12,1],
		\end{cases}
	\end{equation*}
	it is path and simply connected.
	As $d(V_k,V_l)< \Delta_k \Rightarrow \dim V_k \leq \dim V_l$, any left/right K-Cauchy sequence $V_k$ has nondecreasing/nonincreasing dimension, so it is eventually a Cauchy sequence in some $\Gr_p(n)$, which is complete.
	Thus $\Gr^{\pm}(n)$	are Smyth bicomplete.
\end{proof}

The price for connectedness is that any continuous $f:\Gr^\pm(n) \rightarrow \R$ is constant:
$f(V) = f(0)$ for any $V \in \Gr^\pm(n)$, as $0$ is in any neighborhood of $V$ in $\tau^-$, or vice versa in $\tau^+$.
Semicontinuity allows more interesting functions:
with the upper one, $V \subset W$ only implies $f(V) \leq f(W)$ in $\tau^-$, or $f(V) \geq f(W)$ in $\tau^+$,
and with the lower one it is the opposite.

\subsection{Isometries and anti-isometries}

Here we find the isometry and transformation groups of $\Gr(n)$.

\begin{proposition}\label{pr:d perp}
	In the $\wedge$ or max cases, $d(V,W) = d(W^\perp,V^\perp)$ for any $V_{(p)},W_{(q)} \in \Gr(n)$. In the $\ell^2$ case, if, and only if, $p \leq q$ or $p + q = n$.
\end{proposition}
\begin{proof}
	If $p=0$, the result follows from \Cref{pr:d min max}(\ref{it:d=0}).
	If $0 \neq p \leq q$ then $\dim W^\perp \leq \dim V^\perp$, and it follows from the \Cref{tab:asymmetric metrics} formulas,
	as the nonzero principal angles of $V^\perp$ and $W^\perp$ equal those of $V$ and $W$ \cite{Knyazev2007}.
	If $p>q$,  we have in the $\wedge$ and max cases $d(V,W) = \Delta_p = \Delta_{n-q} = d(W^\perp,V^\perp)$, but in the $\ell^2$ case if, and only if, $p = n - q$.
\end{proof}

\begin{corollary}\label{pr:antiisometry}
	In the $\wedge$ or max cases, the map $\perp\,:\Gr(n) \rightarrow \Gr(n)$, $\perp\!(V)=V^\perp$, is a bijective anti-isometry.
	In the $\ell^2$ case, only for $n\leq 1$.%
		\SELF{$n=0$: $\perp$ is also isometry \\
		$n=1$: $\perp$ maps $0 \leftrightarrow \F$}
\end{corollary}

The following holds even in the $\ell^2$ case, as topologies are the same.

\begin{corollary}\label{pr:Gr- = Gr+}
	$\perp$ is a homeomorphism $\Gr^-(n) \cong \Gr^+(n)$ and a self-homeomorphism of $\Gr(n)$, 	with $\Gr_p(n) \cong \Gr_{n-p}(n)$ as usual.
\end{corollary}

So, $\Gr(n)$ has its usual symmetry, and the asymmetric $\Gr^{\pm}(n)$ form together a symmetric bitopological space (Fig.\,\ref{fig:total-grassmannian}).
Note that $\Gr(\infty)$ has no such symmetries, as the map $\perp$ is not defined.

\begin{figure}[t]
	\centering
	\includegraphics[width=\linewidth]{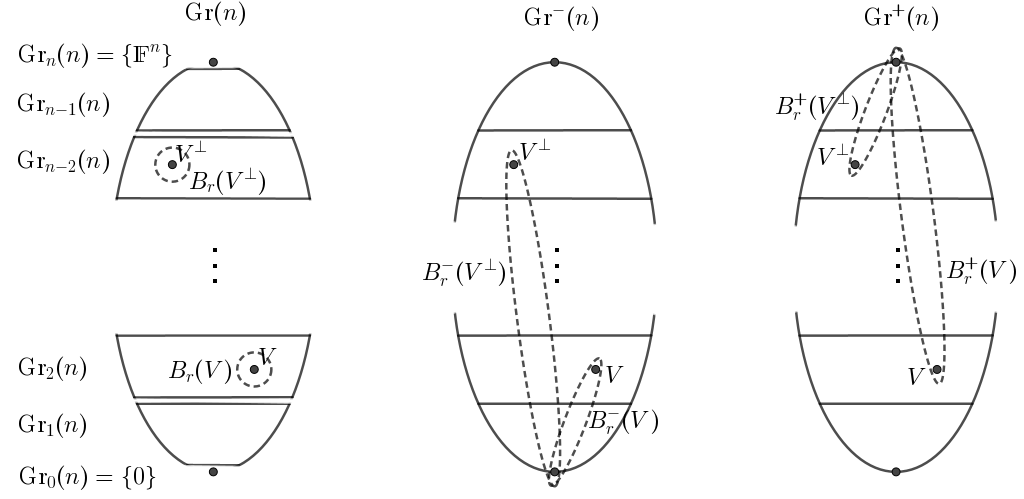}
	\caption{Topologies of the Total Grassmannian. The map $V\mapsto V^\perp$ is a homeomorphism of $\Gr(n)$, and flips the asymmetric $\Gr^{-}(n) \leftrightarrow \Gr^{+}(n)$, so the pair remains symmetric. In $\Gr^\pm(n)$, the $\Gr_p(n)$'s are much more interconnected than shown.}
	\label{fig:total-grassmannian}
\end{figure}

As we show next, the image of a subspace by an isometry of $\Gr(n)$ is formed by the images of its lines.
So, while separately each $\Gr_p(n)$ has its own isometry group,
	\CITE{Wolf2011}
isometries of $\Gr(n)$ are determined by those of the projective space $\PR(\F^n) = \Gr_1(n)$,
and anti-isometries (if any) are compositions of $\perp$ and isometries.
Recall \cite{Wolf2011} that%
\footnote{$\Isom(\PR(\F^n))$ is the same for all metrics of \Cref{tab:metrics Gpn}, as there is a single principal angle.}
$\Isom(\PR(\R^n)) \cong PO(n) = O(n)/\{\pm \mathds{1}\}$ ($\cong SO(n)$ for odd $n$),
\CITE{\url{https://en.wikipedia.org/wiki/Projective_orthogonal_group}}%
and $\Isom(\PR(\C^n))$ is generated by $PU(n) = U(n)/U(1)$
\SELF{$\cong PSU(n) \cong SU(n)/\Z_n$}
and complex conjugation.
\CITE{p.310, (9.3.5) para $q=1$:\\
	$I(M) = \Isom(\PR(\C^n))$,  \\
	$I_0(M) = U(n)/U(1)$ (p.309), \\
	$\alpha = $ complex conj (p.307). \\
	Há erro no caso $2q=n$, pois se $n=2$ temos $\beta \in \alpha I_0$.
	De fato, $\Isom(\C\PR^1) = \Isom(S^2) = O(2)$ e $PU(2) = PSU(2) = SU(2)/\Z_2 = SO(3)$, e nessa identificação $\alpha$ é reflexão de $S^2$ por um plano e $\beta = -\mathds{1}$}

\begin{proposition}
	$\Isom(\Gr(n)) \cong \Isom(\PR(\F^n))$, and
	\[ \Transf(\Gr(n)) \cong 
	\begin{cases}
		\{\mathds{1},\perp\} &\text{if } n=1, \\
		\Isom(\PR(\F^n)) &\text{if } n\geq 2 \ (\ell^2 \text{ case)}, \\
		\Isom(\PR(\F^n)) \rtimes \{\mathds{1},\perp\} &\text{if } n\geq 2 \ (\wedge \text{ or max cases)}.
	\end{cases} \]
\end{proposition}
\begin{proof}
	Any $f\in \Isom(\Gr(n))$ preserves the partial order $\subset$, so it preserves dimensions and
	restricts to $f_1\in \Isom(\Gr_1(n))$.
	And any such $f_1$ extends uniquely to an $f:\Gr(n) \rightarrow \Gr(n)$ preserving $\subset$, and as $f_1$ preserves angles between lines, $f$ preserves principal angles between subspaces, and so $f\in \Isom(\Gr(n))$.
	Clearly, $f \leftrightarrow f_1$ is an isomorphism.

	When $\perp$ is an anti-isometry, 
		\OMIT{\ref{pr:antiisometry}}
	composing it with another one gives an isometry, and vice versa, so 
	$\Transf(\Gr(n)) \cong \Isom(\Gr(n)) \rtimes \{\mathds{1},\perp\}$.
	If $n=1$, $\Isom(\Gr(n)) = \{\mathds{1}\}$.
		\OMIT{Maps in $\Gr(1)=\{0,\F\}$ are $\mathds{1}$ and $\perp$}
	If $n\geq 2$, the $\ell^2$ case has no anti-isometry $f$, as it would reverse $\subset$, so $f(0) = \F^n$, $f(L) \in \Gr_{n-1}(n)$ for $L \in \Gr_1(n)$, and $d(f(0),f(L)) = \Delta_n \neq \Delta_1  = d(L,0)$.
\end{proof}

The spaces $\D^\pm$ of \Cref{ex:N} are related to $\Gr^{\pm}(n)$ by:

\begin{proposition}\label{pr:quotient}
	$\Gr^{\pm}(n)/\Isom(\Gr(n))$ are homeomorphic to $\D^\pm$.
\end{proposition}
\begin{proof}
	As isometries preserve dimensions, any orbit of $\Isom(\Gr(n))$ is in a $\Gr_p(n)$.
	And given $V,W \in \Gr_p(n)$, taking $p$ orthogonal lines in each we can find $f_1\in \Isom(\Gr_1(n))$ mapping them, which extends, as above, to $f\in \Isom(\Gr(n))$ with $f(V)=W$.
	So the orbits are the $\Gr_p(n)$'s.
	
	A quotient map $\pi:\Gr^{-}(n) \rightarrow \D^-$ is given by $\pi(V) = \dim V$,
		\OMIT{with $\pi^{-1}(p) = \Gr_p(n)$}
	since $\pi^{-1}(\{0,1,\ldots,p\}) = \bigcup_{q\leq p} \Gr_q(n)$ 
	and $\pi(B^-_r(V)) = \{0,1,\ldots, \dim V\}$ are open.
	For $\Gr^{+}(n)$ it is similar.
\end{proof}

\subsection{Rectifiable curves in $\Gr^{\pm}(n)$}

Let $\gamma:I\rightarrow \Gr^{\pm}(n)$ be a rectifiable curve.

\begin{definition}
	$t_0\in I$ is a \emph{critical point} of $\gamma$ if there is no $\delta>0$ such that $\gamma \lvert_{I \cap (t_0-\delta,t_0+\delta)}$ is a curve in some $\Gr_p(n)$.
\end{definition}

\begin{proposition}
	$\gamma$ has at most finitely many critical points.
\end{proposition}
\begin{proof}
	Assuming otherwise, we can find $t_1 < t_2 < \cdots$ in $I$ such that $\dim \gamma(t_i) \neq \dim \gamma(t_{i+1})$.
	Passing to a subsequence, we can assume $p_{2i} = \dim \gamma(t_{2i}) > \dim \gamma(t_{2i+1})$.
		\OMIT{as $\dim < \infty$}
	Contradicting rectifiability, \eqref{eq:length} and \eqref{eq:d cases} give
	\[ L(\gamma) \geq \sum_{i=1}^\infty d(\gamma(t_{2i}),\gamma(t_{2i+1})) = \sum_{i=1}^\infty \Delta_{p_{2i}} = \infty. \qedhere \]
\end{proof}

Let $t_0\in I$.
If $t_0\neq\inf I$ there is $\delta>0$ such that $\gamma\lvert_{(t_0-\delta,t_0)}$ has no critical points, so it is a rectifiable curve in some $\Gr_p(n)$.
As $\Gr_p(n)$ is complete and Hausdorff,
	\OMIT{so Hausdorff complete}
$\lim_{t\rightarrow t_0^-} \gamma(t)$ is well defined.
Likewise, $\lim_{t\rightarrow t_0^+} \gamma(t)$ is well defined for $t_0\neq \sup I$.
	\OMIT{in some other $\Gr_q(n)$}
Continuity requires
\begin{equation}\label{eq:continuity}
	\lim_{t\rightarrow t_0^\pm} \gamma(t) \subset \gamma(t_0) \text{ in } \Gr^-(n), \ \text{ and }\  
	\lim_{t\rightarrow t_0^\pm} \gamma(t) \supset \gamma(t_0) \text{ in } \Gr^+(n).
\end{equation}

\begin{definition}\label{df:expansion contraction}	
	Let
	$U  = \lim_{t\rightarrow t_0^-} \gamma(t)$, $V = \gamma(t_0)$ and $W = \lim_{t\rightarrow t_0^+} \gamma(t)$ for%
		\footnote{Disregard $U$ if $t_0 = \inf I$, or $W$ if $t_0 = \sup I$.}
	$t_0 \in I$.
	We say $\gamma$, at $t_0$,
	\begin{itemize}
		\item \emph{expands} ($\expan$) from $U$ to $V$ (\resp $V$ to $W$)  if $U \subsetneq V$ (\resp $V \subsetneq W$);
		\item \emph{contracts} ($\contr$) from $U$ to $V$ (\resp $V$ to $W$) if $U \supsetneq V$ (\resp $V \supsetneq W$).
	\end{itemize}
\end{definition}

Note that $U$ or $W$ might not lie in $\gamma$
(unless $\gamma$ stays at them for a while, for example).
By \eqref{eq:continuity}, we can only have $U \expan \gamma(t_0)$ or $\gamma(t_0) \contr W$ in $\Gr^-(n)$, and $U \contr \gamma(t_0)$ or $\gamma(t_0) \expan W$ in $\Gr^+(n)$.
Also, $t_0$ is critical $\Leftrightarrow \gamma$ has at $t_0$ one expansion, one contraction, or one of each.

\begin{example}
	Let $\gamma:\R \rightarrow \Gr^{-}(n)$ be given, for $c\geq 0$, by
	\vspace{-6pt}
	\[\gamma(t) = 
	\begin{cases}
		\F^n &\text{for } t\in[-c,c], \\
		0 &\text{otherwise.}
	\end{cases} \]
	It is continuous, has $0 \expan \F^n$ at $-c$, and $\F^n \contr 0$ at $c$.
	If $c=0$, a single critical point has both an expansion and a contraction.
	In $\Gr^+(n)$, we must swap $\F^n$ and $0$, or use $(-c,c)$ with $c > 0$.
\end{example}

\begin{lemma}\label{pr:lim Grp}
	If $\lim_{t\rightarrow t_0^\pm} \gamma(t) = U$ then, for any $V \in \Gr(n)$, we have $\lim_{t\rightarrow t_0^\pm} d(\gamma(t),V) = d(U,V)$ and $\lim_{t\rightarrow t_0^\pm} d(V,\gamma(t)) = d(V,U)$.
\end{lemma}
\begin{proof}
	Follows from \eqref{eq:semicontinuidade}, as $\gamma(t)$ converges to $U$ in some $\Gr_p(n)$ and so $d(\gamma(t),U) = d(U,\gamma(t)) \rightarrow 0$.
\end{proof}

The next result shows nonzero length changes \eqref{df:Delta L} occur only at contractions (Fig.\,\ref{fig:curve}).
This can be used in applications as a penalty for data loss,
or, if one wants to compress data, the conjugate asymmetric metric (\Cref{df:conjugate maxsym}) gives a penalty for expansions.

\begin{proposition}\label{pr:length partition}
	$\Delta L(\gamma)\lvert_{t_0} =
	\begin{cases}
		\Delta_{p,d} &\text{if } \gamma \text{ contracts from a $V_{(p)}$ at $t_0$},\\
		0 & \text{otherwise.}
	\end{cases}$
\end{proposition}
\begin{proof}
	By \eqref{eq:Delta L} and \Cref{pr:lim Grp}, $\Delta L(\gamma)\lvert_{t_0} = d(U,V) + d(V,W)$
	for $U,V,W$ as in \Cref{df:expansion contraction}.
	By \eqref{eq:continuity}, $U,W \subset V$ in $\Gr^-(n)$, or $U,W \supset V$ in $\Gr^+(n)$.
	In either case, one of the distances is $0$, and the other is $\Delta_{p,d}$ (if there is a contraction, from $U_{(p)}$ or $V_{(p)}$) or  $0$ (if there is not).
\end{proof}

The length of a curve with contractions is no longer the same for all asymmetric $\ell^2$ or $\wedge$ metrics, as in $\Gr_p(n)$, since $\Delta L$ depends on $d$.

\begin{corollary}\label{pr:null geodesic}
	A curve
		\SELF{continuous so expansions are closed at the correct side}
	$\eta:I\rightarrow \Gr^{\pm}(n)$ is null $\Leftrightarrow$ it is piecewise constant, changing at most by a finite number of expansions.
\end{corollary}
\begin{proof}
	($\Rightarrow$) $\eta$ can have no contraction,
	and in any interval with no expansion it is a null curve in some $\Gr_p(n)$, hence constant.
	($\Leftarrow$) Immediate. 
\end{proof}

So, $\eta$ has a sequence $V = V_1 \expan V_2 \expan \cdots \expan V_k = W$ of expansions (with $V_i \subsetneq V_{i+1)}$), and between them it remains at $V_i$ for a while.

\begin{lemma}\label{pr:d VS WT}
	$d(V\orthsum S, W\orthsum T) = d(V,W)$ for $V_{(p)},W_{(q)},S,T \in \Gr(n)$ with $V,W \subset T^\perp$ and $S \subset T$, and in the $\ell^2$ case $p \leq q$.
		\SELF{If $p\leq q$ holds $\forall \ d$ from \eqref{eq:asym metric}. \\
		$\ell^2$ case also holds if $V \perp W$, or $S=0$ and $p>q+t$}
\end{lemma} 
\begin{proof}
	If $p=0$, it follows from \Cref{pr:d min max}(\ref{it:d=0}).
	If $0 \neq p \leq q$, $V$ and $W$ have principal angles $\theta_1, \ldots,\theta_p$, 
	$V \orthsum S$ and $W \orthsum T$ have the same and $s = \dim S$ zeros,
	and $d(V\orthsum S, W\orthsum T) = \Phi_{s+p}(0,\ldots,0,\theta_1, \ldots,\theta_p) = \Phi_p(\theta_1, \ldots,\theta_p) = d(V,W)$, by \eqref{eq:d cases} and \Cref{df:regular family}. 
	If $p>q$ then $V \orthsum S \pperp W \orthsum T$, and in the $\wedge$ and max cases $\Delta_p$ is constant (\Cref{tab:asymmetric metrics}) and \Cref{pr:d min max}(\ref{it:d=Delta}) gives $d(V\orthsum S, W\orthsum T) = \Delta_{p+s} = \Delta_p = d(V,W)$.
%
\end{proof}

\begin{definition}\label{df:path sum}
	Define $\mu \orthsum \eta : I \rightarrow \Gr(n)$ by
	$(\mu \orthsum \eta)(t) = \mu(t) \orthsum \eta(t)$,
	for $\mu,\eta : I \rightarrow \Gr(n)$ with $\mu(t) \perp \eta(t)$  $\forall\,t$.
	If $\eta(t) = S$ $\forall t$, we write $\mu \orthsum S$.
\end{definition}

\begin{proposition}\label{pr:L gamma eta}
	$L(\mu \orthsum \eta) = L(\mu)$, for a curve $\mu:I \rightarrow \Gr^\pm(V)$ (with no contractions in the $\ell^2$ case) and a null curve $\eta:I \rightarrow \Gr^\pm(V^\perp)$.
\end{proposition}
\begin{proof}
	By \eqref{eq:length}, $L(\mu \orthsum \eta) = \sup \sum d\big(\mu(t_i) \orthsum \eta(t_i)\,,\,\mu(t_{i+1}) \orthsum \eta(t_{i+1}) \big)$.
	By \Cref{pr:null geodesic}, $\eta(t_i) \subset \eta(t_{i+1})$, so \Cref{pr:d VS WT} gives the result.
		\OMIT{$\dim \mu(t_i) \leq \dim \mu(t_{i+1})$ in $\ell^2$ case}
\end{proof}

\begin{definition}\label{df:-gamma perp}
	For a path $\gamma:[a,b] \rightarrow \Gr^\pm(n)$, let $-\gamma^\perp:[-b,-a] \rightarrow \Gr^\mp(n)$ be given by $-\gamma^\perp(t) = \gamma(-t)^\perp$.
\end{definition}

By \Cref{pr:Gr- = Gr+}, $-\gamma^\perp$ is continuous in $\Gr^\mp(n)$.
Minor adjustments at critical points, to satisfy \eqref{eq:continuity}, can make it continuous in $\Gr^\pm(n)$. 

\begin{definition}
	A contraction $V_{(p)} \contr W_{(q)}$ is \emph{balanced}
	if $p+q=n$.
\end{definition}

\begin{proposition}\label{pr:L perp}
	$L(-\gamma^\perp) = L(\gamma)$, for $\gamma:I \rightarrow \Gr^\pm(V)$
		\SELF{nem precisa de topologia}
	(in the $\ell^2$ case, with balanced contractions, if any).
\end{proposition}
\begin{proof}
	Follows from \eqref{eq:length} and \Cref{pr:d perp}.
\end{proof}

\subsection{Geodesics for asymmetric $\ell^2$ or $\wedge$ metrics}\label{sc:Geodesics}

Here, $d$ is an asymmetric $\ell^2$ or $\wedge$ metric.
Max metrics are excluded because, as they use only $\theta_p$, there is little control over the other angles, and their minimal geodesics are less well behaved, even in $\Gr_p(n)$.

\begin{definition}\label{df:arrows}
	$\gamma: V \path W$ or $V \abaixo{\path}{\gamma} W$ mean $\gamma$ is a path from $V$ to $W$ in $\Gr^{\pm}(n)$.
	We also write $\path_I$ to indicate its type (see bellow), and $\pathdim{p}$ for a path in $\Gr_p(n)$.
	Likewise, we use $\geod$ for a minimal geodesic, $\segm$ for a segment, and $\nullpath$ for a null path.
\end{definition}

The notation can be concatenated,
so $\gamma:T \abaixo{\segm}{\mu} U \expan V \geodim{p} W$ means $\gamma$ consists of a segment $\mu$ from $T$ to $U$, an expansion to $V_{(p)}$, then a minimal geodesic of $\Gr_p(n)$ to $W_{(p)}$ (possibly not minimal for $\Gr(n)$).
Some care is needed with expansions or contractions: e.g., in $\Gr^-(n)$ we might not have $\gamma(t_0) = U$ for some $t_0$, just $\lim_{t\rightarrow t_0^-} \gamma(t) = U$ (see \Cref{df:expansion contraction}).

\begin{definition}\label{df:types}
	A path $\gamma: V_{(p)} \path W$ in $\Gr^\pm(n)$ is:
	\begin{enumerate}[(i)]
		\item\label{df:type I}
		\emph{type I} if 
		$\gamma = \mu \orthsum \eta$ with
		$\mu:V \geodim{p} W'$ and $\eta:0 \nullpath T$, for some $W' \in \PP_W(V)$ and $T = W'^\perp \cap W$.	
		
		\item\label{df:type II}
		\emph{type II} if $\gamma:V \nullpath U_{(r)} \contr U' \nullpath W$ with $\Delta_{r,d} = \Delta_{p,d}$.
	\end{enumerate}
\end{definition}

In (\ref{df:type I}), $\mu(t) \subset T^\perp$ and $\eta(t) \subset T$ $\forall\,t$,
by \eqref{eq:geodesic Grp}, as $V,W' \subset T^\perp$,
and \Cref{pr:null geodesic}, which also shows $\gamma$ is a sequence of minimal geodesics in $\Gr_{p_i}(n)$'s and expansions,
$\gamma : V \geodim{p} U_1' \expan U_1 \geodim{p_1} U_2' \expan \cdots \expan U_k \geodim{q} W$.%
	\SELF{Não valeria para max metrics}

In (\ref{df:type II}), $V \subset U$ and $U \neq U' \subset U \cap W$, by \Cref{pr:null geodesic}.
In the $\ell^2$ case, \Cref{tab:asymmetric metrics} shows $\Delta_{r,d} = \Delta_{p,d} \Rightarrow r=p \Rightarrow U=V$, and so $\gamma:V \contr U' \nullpath W$.

Fig.\,\ref{fig:paths I II} has examples of such paths.

\begin{definition}\label{df:L1 L2}
	Let $L_1(V,W) = \dg(V,W)$ and $L_2(V_{(p)},W) = \Delta_{p,d}$.
\end{definition}

When there is no risk of confusion, we write $L_1$ and $L_2$, leaving $V$ and $W$ implicit.
Note that $L_1$ has $\dg$, and $L_2$ has $d$. 
As we will show, they are lengths of type I and II paths from $V$ to $W$, when these exist.

\begin{figure}[t]
	\centering
	\begin{subfigure}[b]{0.46\textwidth}
		\includegraphics[width=\textwidth]{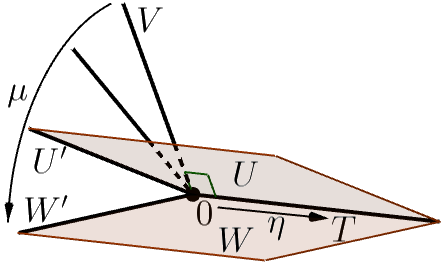}
		\caption{Type I path from $V$ to $W$: as $V$ rotates to $W'$ along $\mu$, at $U'$ it expands to $U = U' \orthsum T$, which rotates to $W$. We can also have a type II path $V \contr 0 \expan W$, or in $\wedge$ case $V \expan \R^3 \contr W$ (in $\ell^2$ case, $\Delta_3 \neq \Delta_1$).}
	\end{subfigure}
	\qquad
	\begin{subfigure}[b]{0.45\textwidth}
		\includegraphics[width=\textwidth]{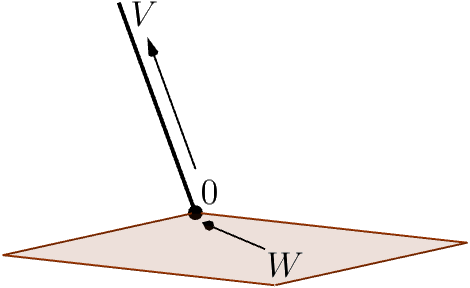}
		\caption{Type II path from $W$ to $V$: it stays at $W$ for a while, then contracts to $0$, and later expands to $V$. In $\wedge$ case we can also have $W \expan \R^3 \contr V$ (in $\ell^2$ case, $\Delta_3 \neq \Delta_2$). There is no type I path from $W$ to $V$.}
		\label{fig:geodesica plano reta}
	\end{subfigure}
	\caption{Type I and II paths between a line $V$ and a plane $W$ in $\R^3$.}
	\label{fig:paths I II}
\end{figure}

\begin{lemma}\label{pr:geodesic phi+X}
	For $S \in \Gr_s(n)$ and $V,W \in \Gr_p(S^\perp)$, 
		\SELF{For max metrics we can have perturbations of $S$ if they are slower than $\tilde{\mu}$ moves}
	$\mu : V\orthsum S \geodim{p+s} W\orthsum S \Leftrightarrow \mu = \tilde{\mu} \orthsum S$ for $\tilde{\mu} : V \geodim{p} W$ in $\Gr_p(S^\perp)$.
	Also, $L(\mu) = L(\tilde{\mu})$.
\end{lemma}
\begin{proof}
	Forming principal bases of $V\orthsum S$ and $W\orthsum S$ from those of $V$ and $W$ and an orthonormal one of $S$,
	follows from \eqref{eq:geodesic Grp} and \Cref{pr:L gamma eta}.
\end{proof}

\begin{proposition}\label{pr:I II exist}
	For paths $\gamma:V_{(p)} \path W_{(q)}$:
	\begin{enumerate}[(i)]
		\item\label{it:I II exist}
		If $p=0$, type I paths exist, type II do not,
			\OMIT{\ref{pr:I II exist}}
		and $L_1 = L_2 = 0$. \\[3pt]
		If $0<p\leq q$, both types exist. \\[3pt]
		If $p > q$, type II paths exist, type I do not, and $L_2 \leq L_1$.
		
%
		
		\item\label{it:L I II}
		If $\gamma$ is type I (\resp II) then $L(\gamma) = L_1$ (\resp $L_2$).
		
		\item If $\gamma$ has no contraction then $L(\gamma) \geq L_1$, with equality $\Leftrightarrow \gamma$ is type I.\label{it:L1}
		
		\item If $\gamma$ has a contraction then $L(\gamma) \geq L_2$, with equality $\Leftrightarrow \gamma$ is type II.\label{it:L2}
	\end{enumerate}
\end{proposition}
\begin{proof}
	(\ref{it:I II exist})	
	As $\Gr_p(n)$ is a geodesic space, there is $\mu$ as in \Cref{df:types}(\ref{df:type I}) $\Leftrightarrow p \leq q$, by \Cref{df:PO}.
	And \Cref{pr:null geodesic} gives $\eta$.
	
	If $p\neq 0$, \Cref{df:types}(\ref{df:type II}) holds for $\gamma:V \contr 0 \expan W$.
		\OMIT{with $U=V$, by \Cref{pr:null geodesic}}
	If $p=0$, $\Delta_{r,d} = \Delta_{0,d} = 0$ implies $r=0$, but $U=0$ can not contract.
	
	If $p>q$, $L_1 = \dg(V,W) = \Delta_{p,\dg} \geq \Delta_{p,d} = L_2$ (see \Cref{tab:asymmetric metrics}).
	\OMIT{\Cref{df:types}}
	
	(\ref{it:L I II}) 
	If $\gamma$ is type I then $L(\gamma) = L(\mu) = d_g(V,W') = d_g(V,W)$, by \Cref{pr:L gamma eta} and \eqref{eq:d cases}. 
	If $\gamma$ is type II, it is null before and after contracting from $U_{(r)}$ at some $t_0$, 
	so \eqref{eq:L Delta L open} and \Cref{pr:length partition} give
	$L(\gamma) = \Delta L(\gamma)\lvert_{t_0} = \Delta_{r,d} = \Delta_{p,d}$.
		\OMIT{\ref{df:types}(\ref{df:type II})}
	
	(\ref{it:L1}) 
	As $\gamma$ does not contract, we have $p\leq q$, so type I paths exist.
	\OMIT{\ref{pr:I II exist}}
	We prove the result via induction on the number $k\geq 0$ of expansions.
	Fig.\,\ref{fig:prova} can help follow the proof.	
	For $k=0$, it is immediate, as $\gamma$ is a path in $\Gr_p(n)$, where minimal geodesics are type I paths (with $\eta=0$).
	
	\begin{figure}[t]
		\centering
		\includegraphics[width=.54\linewidth]{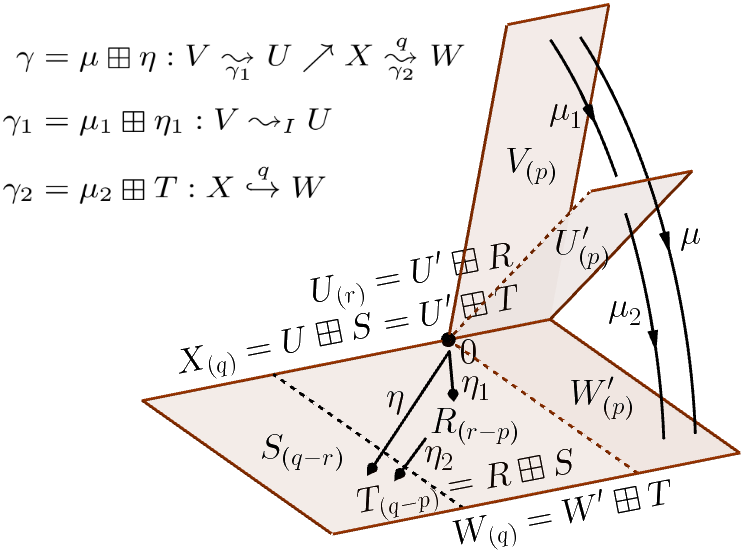}
		\caption{Subspaces and paths used in \Cref{pr:I II exist}(\ref{it:L1}). Dashed lines mean $\orthsum$. 
		The type I path	$\gamma = \mu \orthsum \eta : V_{(p)} \path W$ combines $\mu: V \geodim{p} W'$ formed by $\mu_1 : V \geodim{p} U'$ and $\mu_2:U' \geodim{p} W'$, and $\eta:0 \nullpath T$ formed by $\eta_1: 0 \nullpath R$ and $\eta_2:R \expan  T$.}
		\label{fig:prova}
	\end{figure}
	
	If $k>0$ then $\gamma: V \abaixo{\path}{\gamma_1} U_{(r)} \expan X \abaixo{\pathdim{q}}{\gamma_2} W$, with $k-1$ expansions in $\gamma_1$.
	By \eqref{eq:L Delta L open} and \Cref{pr:length partition}, the induction hypothesis, \Cref{pr:d subspaces}, and the triangle inequality,
	\begin{subequations}
		\begin{align}
			L(\gamma) 
			= L(\gamma_1) + 0 + L(\gamma_2)
			&\geq \dg(V,U) + \dg(X,W) \label{eq:a} \\
			&\geq \dg(V,U) + \dg(U,W) \label{eq:b} \\
			&\geq \dg(V,W) = L_1(V,W). \label{eq:c}
		\end{align}
	\end{subequations}
	
	If $L(\gamma)  = L_1(V,W)$, these are equalities.
	In (\ref{eq:a}), it means $\gamma_1$ is type I, by the induction hypothesis, and $\gamma_2: X \geodim{q} W$.
	So $\gamma_1 = \mu_1 \orthsum \eta_1$ 
	for $\mu_1 : V \geodim{p} U' \in \PP_{U}(V)$
	and $\eta_1 : 0 \nullpath R_{(r-p)} = U'^\perp \cap U$.
		\SELF{$\gamma = \gamma_1 (U' \boxplus  \eta_2) \gamma_2$}
	In (\ref{eq:b}), it means $S_{(q-r)} = U^\perp \cap X \subset W$, by \Cref{pr:equalities subspaces}(\ref{it:dg V' sub V}).
	By \eqref{eq:d cases} and \Cref{pr:d subspaces}, and the triangle inequality,
	\begin{equation}\label{eq:ineqs}
		\dg(V,U) + \dg(U,W)
		\geq \dg(V,U') + \dg(U',W)
		\geq \dg(V,W).
	\end{equation}
	Equality in (\ref{eq:c}) means these $\geq$'s are $=$'s.
	The first one and \Cref{pr:equalities subspaces}(\ref{it:dg V' sub V}) show
	$R \subset W$,
	so $X = U \orthsum S = U' \orthsum T$ for $T_{(q-p)} = R \orthsum S \subset W$.
	Let $W'_{(p)} = T^\perp \cap W \in \PP_W(U')$.
		\OMIT{as $T \perp U'$}	
	By \Cref{pr:geodesic phi+X},
	$\gamma_2 = \mu_2 \orthsum T$ for $\mu_2 : U' \geodim{p} W'$.
		\OMIT{$T = W'^\perp \cap  W$}
	By the last equality in \eqref{eq:ineqs}, \eqref{eq:d cases} and triangle inequality,
	$\dg(V,W) = \dg(V,U') + \dg(U',W) = \dg(V,U') + \dg(U',W') \geq \dg(V,W')$.
	By Propositions \ref{pr:d subspaces} and \ref{pr:proj subsp}, this is an equality and $W' \in \PP_W(V)$.
	Then $\dg(V,W') = L(\mu_1) + L(\mu_2)$,
		\OMIT{$= \dg(V,U') + \dg(U',W')$}
	so $\mu_1$ and $\mu_2$ form $\mu: V \geodim{p} W'$.
	
	Since
	$\gamma: V \abaixo{\path}{\mu_1 \orthsum \eta_1} U' \orthsum R \expan U' \orthsum T \abaixo{\geodim{q}}{\mu_2 \orthsum T} W' \orthsum T$
	can be written as $\gamma = \mu \orthsum \eta$ with $\eta: 0 \abaixo{\nullpath}{\eta_1} R \expan T$, 
	it is type I.	
	
	(\ref{it:L2})
	If $\gamma$ first contracts at $t_0$, after possibly expanding to $U_{(r)}$, 
	\OMIT{in $\wedge$ case}	
	then \eqref{eq:L Delta L} and Propositions \ref{pr:length partition} and \ref{pr:diameter nondecr} give
	$L(\gamma) \geq \Delta L(\gamma)\lvert_{t_0} = \Delta_{r,d} \geq \Delta_{p,d} = L_2$.
	If $L(\gamma) = L_2$ then $\Delta_{r,d} = \Delta_{p,d}$ and $\gamma$ is null outside $t_0$, so type II.
\end{proof}

\begin{theorem}\label{pr:geodesics G^n}
	$\Gr^{\pm}(n)$ is a geodesic space for any asymmetric $\ell^2$ or $\wedge$ metric $d$.
	The  minimal geodesics from $V$ to $W$ are the type I or II paths, whichever exists and has shorter length.
	The intrinsic asymmetric metric of $d$ is $D(V,W) = \min\{\dg(V,W),\Delta_{\dim V,d}\}$.
\end{theorem}
\begin{proof}
	By \Cref{pr:I II exist}, $L(\gamma) \geq \min\{L_1,L_2\}$ for any $\gamma: V \path W$,
	this minimum is always attained, and only at type I or II paths.
%
%
\end{proof}

When $L_1 = L_2$, both types (if they exist) are minimal geodesics.
	\SELF{In the class of run-continuous geodesics, there is none from $V$ to $W$ if $p>q$, and if $p\leq q$ all geodesics are type I.}

\begin{corollary}\label{pr:dg intrinsic}
	The asymmetric $\dg$ is intrinsic in $\Gr^{\pm}(n)$.
\end{corollary}
\begin{proof}\OMIT{\ref{pr:geodesics G^n}}
	By \eqref{eq:asym metric}, $\dg(V,W) \leq \Delta_{\dim V,\dg}$, so $D=\dg$ when $d=\dg$.
\end{proof}

However, unlike in $\Gr_p(n)$, $\dg$ is no longer the intrinsic asymmetric metric of the other $\ell^2$ or $\wedge$ ones.

\begin{example}\label{ex:shortcut}
	If $V, W \in \Gr_2(4)$ have principal angles $\theta_1=\theta_2=\theta$,
	any $\gamma: V \geodim{2} W$ has 
		\OMIT{for an $\ell^2$ or $\wedge$ metric $d$}
	$L(\gamma) = \dg(V,W) = \theta\sqrt{2}$, being a segment only if $d=\dg$.
	If $\theta > \frac{\pi}{2\sqrt{2}}$,  there is a shortcut $\gamma: V \geod_{II} W$ in $(\Gr^\pm(4),\dFS)$ with $L(\gamma) = \Delta_{2,\dFS} = \frac\pi2$.
	\OMIT{$\Delta_{2,\dFS}$}
	It is a segment $\Leftrightarrow \dFS(V,W) = \cos^{-1}(\cos^2 \theta) = \frac\pi2 \Leftrightarrow \theta = \frac\pi2$.
	In $(\Gr^\pm(4),\dg)$ there is no shortcut, as any $V \path_I W$ is a segment of $\Gr_2(4)$, and any $\gamma:V \path_{II} W$ has $L(\gamma) = L_2 = \Delta_{2,\dg} = \frac{\pi \sqrt{2}}{2} \geq \theta \sqrt{2} = L_1$, being a minimal geodesic only if $\theta=\frac\pi2$.
\end{example}

\begin{proposition}\label{pr:trivial geod}
	All minimal geodesics from $V$ to $W$ are trivial $\Leftrightarrow$ $V=W$, or $V$ is a hy\-per\-plane of $W$, or $W=0$ (with $V=\F^n$ in $\wedge$ case).
\end{proposition}
\begin{proof}
	There are trivial paths from $V$ to $W$ $\Leftrightarrow V \subset W$ or $V \supset W$.
	
	If $V \subset W$, minimal geodesics are null, and there is a nontrivial one ($V \expan U \expan W$)
		\OMIT{\ref{pr:null geodesic}}
	if, and only if, there is $U$ with $V \subsetneq U \subsetneq W$.
	
	If $V \supsetneq W$, minimal geodesics are type II.
		\OMIT{\ref{pr:geodesics G^n}, \ref{pr:I II exist}(\ref{it:I II exist})}
	If $W \neq 0$, there is a nontrivial $V \contr 0 \expan W$. 
	If $W=0$, the $\ell^2$ case has only $V \contr 0$,
	but the $\wedge$ case has a nontrivial $V \expan \F^n \contr 0$ if $V \neq \F^n$.
\end{proof}

\begin{proposition}
	 In the $\wedge$ case, 
 		\SELF{Also max case, seja lá como forem as geodésicas}
	 $\gamma:V \geod W$ in $\Gr^\pm(n)$ $\Leftrightarrow -\gamma^\perp: W^\perp \geod V^\perp$ in $\Gr^\mp(n)$, with $-\gamma^\perp$ as in \Cref{df:-gamma perp}.
\end{proposition}
\begin{proof}
	Follows from \Cref{pr:antiisometry} and \eqref{eq:L reversed}.
\end{proof}

Again, minor adjustments turn $-\gamma^\perp$ into a minimal geodesic in $\Gr^\pm(n)$.
In the $\ell^2$ case, $\gamma \leftrightarrow -\gamma^\perp$ preserves path types
and, by \Cref{pr:L perp}, type I lengths,
	\SELF{$L_1$ can change when type I paths do not exist}
but as $L_2(V_{(p)},W_{(q)}) \neq L_2(W^\perp,V^\perp)$ if $p \neq n-q$,
the relation ($\leq$ or $\geq$) between $L_1$ and $L_2$ can change,
so if both types of path exist then $-\gamma^\perp$ might no longer be a minimal geodesic.

\subsection{Convexity in $(\Gr(n),\dg)$}\label{Between-points for dg}

Here we give convexity related results for $(\Gr(n),\dg)$, with topology $\tau^\pm$ when relevant.
\Cref{sc:Asymmetric Fubini-Study metric} does the same for $\dFS$.
In special, we find when there are segments from $V$ to $W$ (so, when a path attains their distance), and when $U \in [V,W]_{\dg}$ (when the triangle inequality attains equality).

\begin{proposition}\label{pr:geodesics dg}
	In $(\Gr^\pm(n),\dg)$:
	\begin{enumerate}[(i)]
		\item Minimal geodesics are segments. More precisely, all type I paths are; type II paths from $V_{(p)}$ to $W_{(q)}$ are segments $\Leftrightarrow$ $V \perp W$ or $p>q$. \label{it:I II segments}
		\item There is a segment linking any $V$ and $W$. \label{it:segments dg}
		\item $U \in [V,W]_{\dg}$ $\Leftrightarrow$ $U$ lies in a segment from $V$ to $W$.\label{it:between segment}
	\end{enumerate} 
\end{proposition}
\begin{proof}
	(\ref{it:I II segments}) Minimal geodesics are segments by \Cref{pr:dg intrinsic}.
	By \Cref{pr:geodesics G^n}, type I paths are minimal geodesics, as $L_1 = \dg(V,W) \leq \Delta_{p,dg} = L_2$.
	Type II ones, when $L_1 = L_2$, so	\Cref{pr:d min max}(\ref{it:d=Delta}) gives the result.
	
	(\ref{it:segments dg}) Follows from (\ref{it:I II segments}), as the space is geodesic.
		\OMIT{\ref{pr:geodesics G^n}}
	
	(\ref{it:between segment}) By (\ref{it:segments dg}), we have segments $V \segm U$ and $U \segm W$, and they form a segment $V \segm U \segm W$ $\Leftrightarrow$ $\dg(V,U) + \dg(U,W) = \dg(V,W)$.
\end{proof}

While (\ref{it:segments dg}) suggests convexity, a few cases have only trivial segments:

\begin{corollary}\label{pr:dg not convex}
	$(\Gr(n),\dg)$ is not Menger convex for $n \neq 0$.
		\SELF{$n=0$ is vacuously convex, as there are no $V \neq W$}
	Still, $(V,W)_{\dg} = \emptyset$  only in the cases of \Cref{pr:trivial geod}.
\end{corollary}

Also, no $U \neq V,W$ satisfies (\ref{it:pqr}) or (\ref{it:pr}) below precisely in those cases.

\begin{proposition}\label{pr:triangle equality dg}
	For distinct $U_{(r)}, V_{(p)}, W_{(q)} \in \Gr(n)$,
		\SELF{Without distinct: \\
			$V=U \leq W$ in \ref{it:pqr}; \\
			$V \neq U = W$ in \ref{it:pr}; \\
			$U$ betw $V=W \Rightarrow V=U=W$; \\
			$V=U > W$ needs new case $r>q$}
	$U \in (V,W)_{\dg}$	if, and only if:
		\SELF{Cases intersect}
	\begin{enumerate}[(i)]
		\item $p\leq r \leq q$, $U=U'\orthsum R$ and $W = W' \orthsum T$ with $V, U', W' \in \Gr_p( T^\perp)$, $R \subset T$ and%
			\footnote{As in \Cref{sc:Grassmannians}, $U' \in [V,W']_{\dg} \Leftrightarrow U' = \mu(t_0)$ for a $t_0$ and $\mu:V \geodim{p} W'$ as in \eqref{eq:geodesic Grp}.}
		$U' \in [V,W']_{\dg}$; or \label{it:pqr}
		
		
		
		\item $U\subset W$, with $V\perp W$ or $p > q$. \label{it:pr} 
	\end{enumerate}
\end{proposition}
\begin{proof}
	($\Rightarrow$) 
	We can prove it in $(\Gr^\pm(n),\dg)$.
	By \Cref{pr:geodesics dg}(\ref{it:between segment}),
	$U= \gamma(t_0)$ for a $\gamma: V \segm W$ and $t_0$.
	
	If $\gamma$ is type I then $p\leq r \leq q$ and $\gamma = \mu\orthsum \eta$ with $\mu: V \segdim{p} W'$
		\OMIT{\ref{pr:geodesics dg}(\ref{it:I II segments})}
	and $\eta:0 \nullpath T$ for $W' \in \PP_W(V)$  and $T = W'^\perp \cap W \subset V^\perp$.
	Also, $U'=\mu(t_0) \subset T^\perp$, $R = U'^\perp \cap U = \eta(t_0) \subset T$,
	and $U' \in [V,W']_{\dg}$ by \Cref{pr:d VS WT}, so we have (\ref{it:pqr}).

	If $\gamma$ is type II, \Cref{pr:geodesics dg}(\ref{it:I II segments}) gives $V\perp W$ or $p > q$.
	As $\dg$ is $\ell^2$, $U$ lies after the contraction of $\gamma$, so $U \subset W$ and we have (\ref{it:pr}).
		\OMIT{\ref{pr:null geodesic}}
			
%
		
	($\Leftarrow$) 
	(\ref{it:pqr}) 
	Follows from \Cref{pr:d VS WT}.
%
%
	(\ref{it:pr}) 
	By \Cref{pr:d min max}(\ref{it:d=Delta}), 
		\OMIT{or \ref{pr:geodesics dg}(\ref{it:I II segments})}
	the type II path
	$V \contr 0 \expan U \expan W$ 
	(or $V = 0 \expan U \expan W$, or $V \contr 0 = U \expan W$)
	is a segment,
		\OMIT{or $\gamma: V \contr U \expan W$ if $U=0$}
	so $U \in (V,W)_{\dg}$.
\end{proof}

\section{Asymmetric $\wedge$ metrics}\label{sc:Asymmetric Fubini-Study}

Now we link the asymmetric $\wedge$ metrics to the asymmetric angle $\Theta_{V,W}$ and volume projection factor $\pi_{V,W}$ of \Cref{sc:Asymmetric angles},
which provide geometric interpretations.
The formulas in \Cref{sc:Computing asymmetric angles} can help compute the metrics,
whose use is facilitated by the many properties in
\cite{Mandolesi_Grassmann,Mandolesi_Products,Mandolesi_Trigonometry,Mandolesi_Pythagorean}.

\begin{proposition}\label{pr:Theta prod cos}
	For $V, W\in \Gr(n)$,
	\[ \dFS(V,W) = \Theta_{V,W}, \quad \dcw(V,W) = 2 \sin \frac{\Theta_{V,W}}{2}, \quad \dBC(V,W) = \sin \Theta_{V,W}.\] 
\end{proposition}
\begin{proof}
	Follows from \eqref{eq:adFS}, \eqref{eq:Theta thetai} 
	and the formulas in \Cref{tab:asymmetric metrics}.
\end{proof}

By \eqref{eq:Theta pi}, 
	\OMIT{\ref{pr:Theta prod cos}}
the (squared, if $\F=\C$) cosine of $\dFS(V,W)$ is the factor $\pi_{V,W}$ by which volumes in $V$ contract if orthogonally projected on $W$ (Fig.\,\ref{fig:projections-sin}).

\begin{corollary}
	For $V, W\in \Gr(n)$,
	\[ \dcw(V,W) = 
	\begin{cases}
		\sqrt{2 - 2\pi_{V,W}} &\text{if $\F=\R$}, \\
		\sqrt{2 - 2\pi_{U,W_\R}} &\text{if $\F=\C$,}
	\end{cases} \]
	for any maximal totally real subspace $U$ of $V_\R$ \wrt $W_\R$ (\Cref{df:tot real}).
\end{corollary}
\begin{proof}
	Follows from \Cref{pr:Theta prod cos,pr:Theta pi,pr:pi max tot real}.
\end{proof}

So, for $\F=\R$, $\dcw^2(V,W)$ is twice the volume fraction lost by a subset of $V$ when orthogonally projected on $W$ (Fig.\,\ref{fig:projections-sin}).
For $\F=\C$, this holds for subsets of maximal totally real subspaces of $V_\R$ \wrt $W_\R$.

\begin{figure}[t]
	\centering
	\includegraphics[width=0.7\linewidth]{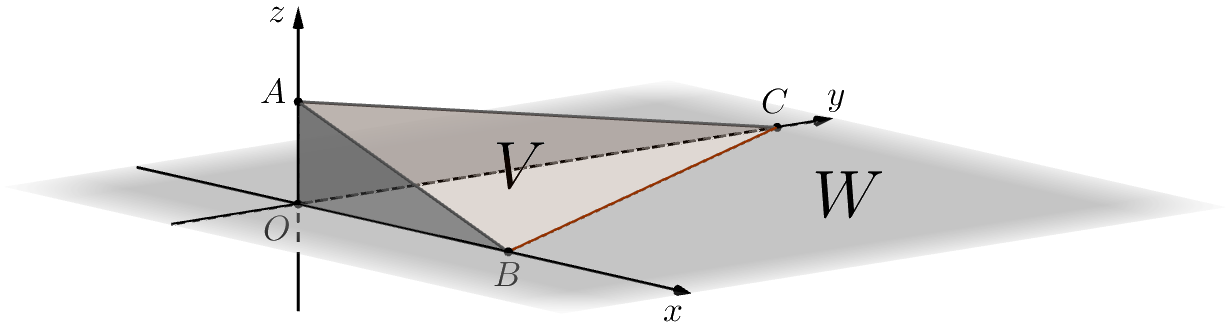}
	\caption{$\dFS(V,W) = \cos^{-1}\left(\frac{\area(OBC)}{\area(ABC)}\right)$ measures the separation between  $V$ (plane of $ABC$) and $W$ in terms of how areas of $V$ contract if orthogonally projected on $W$;
		$\dcw(V,W) = \left(2\frac{\area(ABC) - \area(OBC)}{\area(ABC)}\right)^{\frac12}$ uses the area lost by $ABC$ when projected; 
		$\dBC(V,W) = \left(\frac{\area^2(ABC) - \area^2(OBC)}{\area^2(ABC)}\right)^{\frac12} = \left(\frac{\area^2(OAB) + \area^2(OAC)}{\area^2(ABC)}\right)^{\frac12}$ is similar, but with squared areas, and also uses squared areas of the vertical triangles that `lift' $V$.}
	\label{fig:projections-sin}
\end{figure}

\begin{proposition}
	Let $V_{(p)}, W_{(q)} \in \Gr(n)$, and extend an orthogonal basis $\{w_1,\ldots,w_q\}$ of $W$ to another $\beta = \{w_1,\ldots,w_n\}$ of $\F^n$.
	Then
	\SELF{If $V=0$ or $W=\F^n$ there is no $\ii\in\II_p^n$ with $[w_\ii]\not\subset W$, and the sum is $0$. \\ If $V\neq 0$ and $W=0$ then $[w_\ii]\not\subset W \ \forall \ii\in\II_p^n$, and the sum is $1$}
	\begin{equation*}
		\dBC(V,W) = 
		\begin{cases}
			\sqrt{1 - \pi_{V,W}^2} = \sqrt{\sum_{U \not\subset W} \pi_{V,U}^2} &\text{if } \F=\R, \\
			\sqrt{1 - \pi_{V,W}} = \sqrt{\sum_{U \not\subset W} \pi_{V,U}} &\text{if } \F=\C,
		\end{cases}
	\end{equation*} 
	where the sums run over all coordinate $p$-subspaces%
	\footnote{A \emph{coordinate $p$-subspace} of a basis is a subspace spanned by $p$ of its elements.\label{ft:coord subspace}}
	$U$ of $\beta$ with $U\not\subset W$.
\end{proposition}
\begin{proof}
	See \Cref{sc:Grassmann algebra} for notation and results we use.
	
	If $p > q$ then $\dBC(V,W) = 1$ (see \Cref{tab:asymmetric metrics}), $\pi_{V,W} = 0$ by \eqref{eq:def Theta pi}, and as $U \not\subset W$ is always satisfied the result follows from \eqref{eq:sum pi2}.
	
	Let $p\leq q$. 
	The coordinate $p$-subspaces of $\beta$ are
	$U = [w_{i_1}\wedge\cdots\wedge w_{i_p}]$ with $1\leq i_1 < \cdots<i_p\leq n$, and $U \not\subset W \Leftrightarrow i_j>q$ for some $j$.
	We can assume $\beta$ is orthonormal, so $\{w_{i_1}\wedge\cdots\wedge w_{i_p}:1\leq i_1 < \cdots<i_p\leq k\}$ is an orthonormal basis of $\bigwedge^p W$ for $k=q$, 
	or of $\bigwedge^p \F^n$ for $k=n$.
	For a unit $A\in\bigwedge^p V$, \Cref{pr:Theta prod cos} and \eqref{eq:def Theta pi} give
	\begin{align*}
		\dBC^2(V,W) &=  1 - \cos^2 \Theta_{V,W} = \|A\|^2 - \|P_W A\|^2 \\
		&= \sum_{\substack{1\leq i_1 < \cdots<i_p\leq n, \\ i_j>q \text{ for some } j}} |\inner{w_{i_1}\wedge\cdots\wedge w_{i_p},A}|^2 \\
		&= \sum_{U \not\subset W}  \|P_U A\|^2  
		= \sum_{U \not\subset W} \cos^2 \Theta_{V,U}, 
	\end{align*}
	so the result follows from \eqref{eq:Theta pi}.
\end{proof}

So, $\dBC^2(V,W)$ is 1 minus the (squared, if $\F=\R$) volume of the projection on $W$ of a unit volume of $V$.
Equivalently, it is the sum of (squared, if $\F=\R$) volumes of projections of a unit volume of $V$ on all coordinate $p$-subspaces $U\not\subset W$ (Fig.\,\ref{fig:projections-sin}).

\begin{example}
	If $\{f_1,\ldots,f_5\}$ is the canonical basis of $\R^5$, 
	$e_1=\frac{\sqrt{3}f_1+f_4}{2}$ and $e_2=\frac{f_2+f_5}{\sqrt{2}}$ then $V=\Span\{e_1,e_2\}$ and $W=\Span\{f_1,f_2,f_3\}$ have principal angles $30^\circ$ and $45^\circ$. 
	
	Then $\dFS(V,W) = \cos^{-1}(\frac{\sqrt{3}}{2}\cdot\frac{\sqrt{2}}{2}) = \Theta_{V,W} \cong 52.2^\circ$, and 
	areas in $V$ shrink by $\frac{\sqrt{6}}{4}$ if orthogonally projected on $W$.
	Volumes in $W$ vanish if projected on $V$, so $\dFS(W,V) = \Theta_{W,V}=90^\circ$.
	
	And $\dcw(V,W) = \sqrt{2-2 \cdot \frac{\sqrt{3}}{2}\cdot\frac{\sqrt{2}}{2}} = 2 \sin \frac{\Theta_{V,W}}{2} \cong 0.88$, whose square is twice the fraction $1- \frac{\sqrt{6}}{4} \cong 0.39$ that areas lose when projected from $V$ to $W$. 
	And $\dcw(W,V) = \sqrt{2}$ as volumes are completely lost (fraction $1$) when projected from $W$ to $V$.
	
	Also, $\dBC(V,W) = \sqrt{1-(\frac{\sqrt{3}}{2})^2 \cdot (\frac{\sqrt{2}}{2})^2} = \sin \Theta_{V,W}  = \sqrt{\frac{5}{8}}$.
	It is easy to check that $\pi_{V,[f_i \wedge f_j]}$ is $0$ for $(i,j) = (1,4), (2,5), (3,4)$ or $(3,5)$, $\frac{\sqrt{6}}{4}$ for $(1,5)$, and  
	$\frac{\sqrt{2}}{4}$ for $(2,4)$ or $(4,5)$, so we obtain the same with
	\[ \dBC(V,W) = \Big(\sum_{\substack{1\leq i < j \leq 5, \\ \{i,j\}\not\subset \{1,2,3\}}} \pi_{V,[f_i \wedge f_j]}^2 \Big)^\frac12. \]
	Extending $\{e_1,e_2\}$ orthogonally to $\{e_1,\ldots,e_5\}$, \eqref{eq:sum pi2} gives
	 \begin{align*}
		\dBC^2(W,V) &= \sum_{\substack{1\leq i < j < k \leq 5, \\ \{i,j,k\}\not\subset \{1,2\}}} \pi_{W,[e_i \wedge e_j \wedge e_k]}^2 = \sum_{1\leq i < j < k \leq 5} \pi_{W,[e_i \wedge e_j \wedge e_k]}^2  = 1.
	\end{align*} 
\end{example}

\subsection{Convexity in $(\Gr(n),\dFS)$}\label{sc:Asymmetric Fubini-Study metric}

Results here are for $(\Gr(n),\dFS)$, with $\tau^\pm$ when relevant.
Compare with \Cref{Between-points for dg}.
By \Cref{pr:geodesics G^n} and \Cref{tab:asymmetric metrics}, the intrinsic asymmetric metric of $\dFS$ is $D_{FS}(V,W) = \min\{\dg(V,W),\frac\pi2\}$,
and type I (\resp II) paths from $V$ to $W$ are minimal geodesics $\Leftrightarrow \dg(V,W) \leq \frac\pi2$ (\resp $\geq$).

\begin{proposition}\label{pr:segment FS}
	$(\Gr^\pm(n),\dFS)$ has a segment $\gamma:V \segm W$ of:%
		\SELF{If $\dim(V\cap W) + 1 = p \leq q$, $\gamma = \mu \orthsum \eta$ for a min geod $\mu$ between lines $(V\cap W)^\perp \cap V$ and $(V\cap W)^\perp \cap W_P$, and null geod $\eta$ from $V \cap W$ to $W$. If $\dim(V\cap W) = p$, $\gamma$ is null geod}
	\begin{enumerate}[(i)]
		\item type I $\Leftrightarrow V = K\orthsum S$ and $W = L\orthsum T$ for $K,L \in \Gr_1(T^\perp)$ and $S\subset T$. 
			\SELF{otherwise any $L \subset S^\perp \cap W$}
		In such case, $\gamma = \mu \orthsum \eta$ for $\mu:K \segdim{1} L$ in $\Gr_1(T^\perp)$ and $\eta: S \nullpath T$.\label{it:segm I}%
		\SELF{possibly for a different choice of $L$ and $T$, if $K \perp W$}
			
		\item type II $\Leftrightarrow V\pperp W$.\label{it:segm II}
	\end{enumerate}
\end{proposition}
\begin{proof}
	(\ref{it:segm I}) By \Cref{pr:I II exist},
		\OMIT{geodesic space, $d_{FS} \leq D_{FS}$}
	there is $\gamma:V_{(p)} \path_I W_{(q)}$ $\Leftrightarrow$ $p \leq q$.
	It is a segment $\Leftrightarrow$ $\dFS(V,W) = L_1 = \dg(V,W)$ $\Leftrightarrow$ $\dim(V\cap W) \geq p-1$, by \Cref{pr:extended inequalities}(\ref{it:new ineqs2}).
	This corresponds to $V$ and $W$ as above.
	
	Then $\gamma=\tilde{\mu}\orthsum\tilde{\eta}$ for $\tilde{\mu}:V \geodim{p} W'$ and $\tilde{\eta}: 0 \nullpath T'$ with $W' \in \PP_W(V)$ and $T' = W'^\perp \cap W$.
	We can assume $W' = L \orthsum S$ and $T' = S^\perp \cap T$, since $L = P_W (K)$ if $K \not\perp W$ and otherwise we can choose any line $L \subset S^\perp \cap W$.
	By Propositions \ref{pr:geodesic phi+X} and \ref{pr:PR geometry}, 
		\OMIT{minimal geodesics are segments}
	$\tilde{\mu} = \mu \orthsum S$ with $\mu$ as above.
	So $\gamma = \mu \orthsum \eta$ with $\eta = \tilde{\eta} \orthsum S: S \nullpath T$.
		\OMIT{null by \Cref{pr:L gamma eta}}

	(\ref{it:segm II}) 
	By \Cref{pr:I II exist},
	\OMIT{(\ref{it:I II exist})} 
	there is $\gamma:V_{(p)} \path_{II} W_{(q)}$ $\Leftrightarrow$ $p \neq 0$.
	It is a segment $\Leftrightarrow$ $\dFS(V,W) = L_2 = \frac\pi2 \Leftrightarrow V \pperp W$, by \Cref{pr:d min max}.
\end{proof}

So, for $\dFS$, minimal geodesics are usually not segments.
	\SELF{unlike $\dg$}
If $V \pperp W$ and $\dim(V\cap W) + 1 < p=q$, there is no segment in $\Gr_p(n)$, by \Cref{pr:segment Gp}, but a type II segment $V \contr 0 \expan W$ gives a shortcut through $\Gr^{\pm}(n)$.

%
\begin{proposition}\label{pr:U in segm}
	Let $U,V,W\in (\Gr^{\pm}(n),\dFS)$ be distinct.
		\SELF{If $V = W$ the only segment is constant (type I), $U=V=W$ and the conditions in (\ref{it:U in I}) hold, and vice versa.}
	Then $U$ is in a segment from $V$ to $W$ of:
	\begin{enumerate}[(i)]
		\item type I $\Leftrightarrow$ $U = J\orthsum R$, $V = K\orthsum S$ and $W = L\orthsum T$ with $J,K,L \in \Gr_1(T^\perp)$, $S\subset R\subset T$ and%
			\footnote{\Cref{pr:PR geometry} shows when $J \in (K,L)_{\dFS}$.}
		$J \in [K,L]_{\dFS}$;
		\label{it:U in I} 
		
		\item type II $\Leftrightarrow$ $V\pperp W$, with $V\subset U$ or $U\subset W$.
		\label{it:U in II}
	\end{enumerate}
\end{proposition}
\begin{proof}	
	(\ref{it:U in I}, $\Rightarrow$)   
	If $U=\gamma(t_0)$ for $\gamma$ as in \Cref{pr:segment FS}(\ref{it:segm I}), take $J= \mu(t_0)$ and $R=\eta(t_0)$.
	(\ref{it:U in I}, $\Leftarrow$) \Cref{pr:PR geometry} 
		\OMIT{$J\in[K,L] \Rightarrow J$ in $K\segm L$}
	gives $\tilde{\mu}:K \segm L$ in $\Gr_1(T^\perp)$ with $\tilde{\mu}(t_0) = J$ for some $t_0$.
	For $\mu = \tilde{\mu} \orthsum S$ and $\eta: 0 \nullpath S^\perp \cap T$ with $\eta(t_0) = S^\perp \cap R$, we have $\mu \orthsum \eta: V \path_I W$ with $(\mu \orthsum \eta) (t_0) = U$.
	By Propositions \ref{pr:L gamma eta} and \ref{pr:d VS WT}, $L(\mu \orthsum \eta) = L(\mu) = L(\tilde{\mu}) = \dFS(K,L) = \dFS(V,W)$.
	
	(\ref{it:U in II}) By \Cref{pr:segment FS}(\ref{it:segm II}), $U$ is in a type II segment $\Leftrightarrow V \pperp W$ and either $V \subset U$ (if $U$ is before the contraction) or $U \subset W$ (after).
\end{proof}
%

\begin{proposition}\label{pr:Total Grassmannian}
	For $U,V,W\in \Gr(n)$, $U \in [V,W]_{\dFS}$
	if, and only if:%
		\SELF{cases intersect (e.g., if $V \subset U \subset W$)}
	\begin{enumerate}[(i)]
		\item $U$ lies in a type I segment from $V$ to $W$ in $(\Gr^{\pm}(n),\dFS)$; or \label{it:UVW ABC}
		\item $V\subset U$, with $V\pperp W$ or $V^\perp\cap U\subset W$; or \label{it:UV}
		\item $U\subset W$, with $V\pperp W$ or $P_W(V)\subset U$. \label{it:VW} 		
	\end{enumerate}
\end{proposition}
\begin{proof}
	By \Cref{pr:Theta prod cos}, 
	$U \in [V,W]_{\dFS}$ $\Leftrightarrow$ $\Theta_{V,W} = \Theta_{V,U} + \Theta_{U,W}$.
	Propositions \ref{pr:d min max} and \ref{pr:equalities subspaces}(\ref{it:V' sub V}) show
	(\ref{it:UV}) $\Leftrightarrow (\Theta_{V,U}=0$ and $\Theta_{V,W} = \Theta_{U,W})$,
	and
	(\ref{it:VW}) $\Leftrightarrow (\Theta_{U,W}=0$ and $\Theta_{V,W} = \Theta_{V,U}$).

	It is immediate that (\ref{it:UVW ABC}) $\Rightarrow$ $U \in [V,W]_{\dFS}$, and we prove that (\ref{it:UVW ABC}) follows from $\Theta_{V,W} = \Theta_{V,U} + \Theta_{U,W}$ with $\Theta_{V,U} \neq 0$ and $\Theta_{U,W}\neq 0$.
	In such case, $\Theta_{V,U} \neq\frac\pi 2$ and $\Theta_{U,W}\neq\frac\pi 2$.
	Let $U' = P_U(V) \neq V$ (as $\Theta_{V,U} \neq 0$)  and $W' = P_W (U')$.
	By \Cref{pr:d subspaces},
	$\Theta_{U',W} \neq \frac{\pi}{2}$.
	By \eqref{eq:Theta dim}, $V, U', W' \in \Gr_p(n)$ for some $p$.
	The triangle inequality, \eqref{eq:Theta VPV} and \Cref{pr:d subspaces} give 
	\[ \Theta_{V,W'} 
	\leq \Theta_{V,U'} + \Theta_{U',W'} 
	= \Theta_{V,U} + \Theta_{U',W} 
	\leq \Theta_{V,U} + \Theta_{U,W} = \Theta_{V,W}. \]

	By \Cref{pr:d subspaces}, these $\leq$'s are $=$'s.
	The last one gives $\Theta_{U',W} = \Theta_{U,W} \neq 0$,
	so $U' \neq W'$.	
	The first one gives $U' \in (V,W')_{\dFS}$.
	By \Cref{pr:triangle equality same dim},
	$U' = J \orthsum S$, $V = K \orthsum S$ and $W' = L \orthsum S$
	for $S \in \Gr_{p-1}(n)$ and distinct lines $J,K,L \subset S^\perp$ with $J \in (K,L)_{\dFS}$.

	By \Cref{pr:PR geometry}, $K \subset J \oplus L \subset U' + W'$.
	As $\Theta_{U',W} = \Theta_{U,W}$ and, by \eqref{eq:Theta dim}, $U' \not\pperp W$, \Cref{pr:equalities subspaces}(\ref{it:V' sub V})
	gives
	$R' = U'^\perp\cap U \subset W$.
	As $P_W(U') = W'$, we have $U' + W' \perp R'$ and $U' + W' \perp T' = (W' \orthsum R')^\perp \cap W$.
	Thus $J,K,L \perp T$ for $T = S \orthsum R' \orthsum T'$.
	Let $R = S \orthsum R'$.
	As $U = U' \orthsum R' = J \orthsum R$ and
	$W = W' \orthsum R' \orthsum T' = L \orthsum T$,
	\Cref{pr:U in segm}(\ref{it:U in I}) gives (\ref{it:UVW ABC}).
\end{proof}

\begin{example}
	Let $\{e_1,\ldots,e_6\}$ be the canonical basis  of $\R^6$.
	By \Cref{pr:segment FS}, there is no segment from $V=\Span\{e_1+e_5,e_2+e_6\}$ to $W=\Span\{e_1,\ldots,e_4\}$,
	but $U_1 = V \orthsum \Span\{e_3\}$ and $U_2 = \Span\{e_1,e_2,e_3\}$ are in $(V,W)_{\dFS}$, 
	respectively by
		\OMIT{\ref{pr:Total Grassmannian}}
	(\ref{it:UV}) and (\ref{it:VW}) above.
\end{example}

So, unlike \Cref{pr:geodesics dg},
	\OMIT{(\ref{it:between segment})}
$U \in [V,W]_{\dFS}$ does not mean $U$ lies in a segment $V \segm W$ of $(\Gr^{\pm}(n),\dFS)$.
But we have the following.

\begin{proposition}\label{pr:U betw in geod}
	$U \in [V,W]_{\dFS}$ $\Rightarrow$ $U$ lies in a minimal geodesic from $V$ to $W$ in $(\Gr^{\pm}(n),\dFS)$.
\end{proposition}
\begin{proof}\OMIT{\ref{pr:Total Grassmannian}}	
	Assume \WLOG $U,V,W$ distinct.
	In case (\ref{it:UVW ABC}) above, it is immediate.	
	In (\ref{it:UV}) and (\ref{it:VW}), with $V\pperp W$, it follows from \Cref{pr:U in segm}(\ref{it:U in II}).
		
	
	If $V_{{(p)}} \not\pperp W_{(q)}$ then $\dim P_W(V) = p \leq q$.
	In (\ref{it:UV}), 
		\OMIT{$V \subset U$ and $V^\perp \cap U \subset W$}
	$U_{(r)}$ lies in a type I path $V \expan U \geodim{r} (P_W V) \orthsum (V^\perp \cap U)  \expan  W$
	and, if $V\neq 0$, in a type II path $V \expan U \contr 0 \expan  W$.
	In (\ref{it:VW}), 
		\OMIT{$P_W(V) \subset U \subset W$} 
	it lies in a type I path $V  \geodim{p} P_W (V)  \expan U \expan W$,
	and, if $V\neq 0$, in a type II path $V \contr 0 \expan U \expan W$.
	In each case, the shorter of the two paths is a minimal geodesic.
\end{proof}

The converse does not hold:

\begin{example}
	$\gamma(t) = \Span\{ (\cos \frac{t\pi}{3}, \sin \frac{t\pi}{3} , 0, 0) , (0,0,\cos \frac{t\pi}{3},\sin \frac{t\pi}{3}) \}$ is a minimal geodesic from $V= \gamma(0)$ to $W=\gamma(1)$ in $(\Gr_2(4),\dFS)$, and also in $(\Gr^{\pm}(4),\dFS)$,
	as principal angles are $\theta_1=\theta_2=\frac\pi3$ and so $L_1 = \dg(V,W) = \frac{\pi \sqrt{2}}{3} < \frac\pi2 = L_2$.
		\OMIT{\Cref{ex:shortcut}}
	But $U=\gamma(\frac12) \notin [V,W]_{\dFS}$,
	as $\dFS(V,W) = \cos^{-1}(\cos^2 \frac\pi3) \cong 76^\circ$ and 
	$\dFS(V,U) = \dFS(U,W) = \cos^{-1}(\cos^2 \frac\pi6) \cong 41^\circ$.
\end{example}

Unlike \Cref{pr:dg not convex}, $(V,W)_{\dFS} = \emptyset$ is quite common:

\begin{proposition}\label{pr:distinct between FS}
	$(\Gr(n),\dFS)$ is not Menger convex for $n \neq 0$.
		\SELF{$n=0$ is vacuously convex, as there are no $V \neq W$}
	Also,  $(V_{(p)},W_{(q)})_{\dFS} \neq \emptyset$
	if and only if:
		\OMIT{the cases intersect}
	\begin{enumerate}[(i)]
		\item $V \pperp W$ with $V \neq \F^n$ or $W \neq 0$; or \label{it:bet1}
		\item $W\pperp V$ and $V\not\subset W$; or \label{it:bet2}
		\item $p \leq q-2$ or $p = q = \dim(V\cap W) +1$. \label{it:bet3}
	\end{enumerate}
\end{proposition}
\begin{proof}
	It is not convex as $(\F^n,0)_{\dFS} = \emptyset$, by Propositions \ref{pr:trivial geod} and \ref{pr:U betw in geod}.
		\OMIT{$\F^n \neq 0$ for $n \neq 0$}
		
	($\Rightarrow$) 
	There is $U \in (V,W)_{\dFS}$, so $U,V,W$ are distinct.
	If $V \pperp W$ we have (\ref{it:bet1}), as $(\F^n,0)_{\dFS} = \emptyset$.
	If $V \subset W$ then $V \subsetneq U \subsetneq W$, so $p \leq q-2$.
	Assuming $V \not\pperp W$ and $V \not\subset W$,
	\Cref{pr:Total Grassmannian} shows either:
	\begin{itemize}
		\item $U$ lies in a $V \segm_I W$, so $U,V,W$ are as in \Cref{pr:U in segm}(\ref{it:U in I}),
		and if $S = T$ then $p = q = \dim(V\cap W) +1$,
		and otherwise $W \pperp V$; or
		
		\item $V \subsetneq U$ and $V^\perp \cap U \subset W$, or $P_W(V) \subset U \subsetneq W$, so $W \pperp V$.
	\end{itemize}
%

	($\Leftarrow$)
	By \Cref{pr:Total Grassmannian}, we have $U \in (V,W)_{\dFS}$:
	in (\ref{it:bet1}), for any $U \supsetneq V$ or $U\subsetneq W$;
	in (\ref{it:bet2}), for $U=P_W(V)$ ($\neq V,W$ as $V\not\subset W$ and $W\pperp V$);
	in (\ref{it:bet3}), for any $P_W(V) \subsetneq U \subsetneq W$ or
	$U$ as in \Cref{pr:triangle equality same dim}.
\end{proof}

\section{Final remarks}\label{sc:conclusion}

We have presented a simple and natural method for extending regular Grassmannian metrics to asymmetric metrics in the Total Grassmannian,
and have given explicit formulas for the main ones.
The resulting geometry is more adequate for $\Gr(n)$ than the usual disjoint union of $\Gr_p(n)$'s, as it reflects containment and proximity relations between subspaces of any dimensions,
and provides more flexibility for applications, with continuous paths linking the $\Gr_p(n)$'s.
A complete description of minimal geodesics was given for the asymmetric $\ell^2$ and $\wedge$ metrics, as well as convexity results for the asymmetric geodesic and Fubini-Study metrics. 

It would be interesting to extend such results to other asymmetric metrics.
We expect the max ones to have more unwieldy geodesics, 
as they control directly only the largest principal angle.
We note that their simplicity can be convenient for some purposes, but such rough metrics are inadequate for applications in which many small differences between subspaces can be more relevant than a single large one.

There are many ways to change the asymmetric metrics of \Cref{tab:asymmetric metrics} into others more suitable for particular problems.
For example, with their conjugates $\bar{d}$ (\Cref{df:conjugate maxsym}) many results are simply reversed:
\begin{itemize}
	\item $\cd(V,W)$ measures the failure of $V\supset W$, so columns in \Cref{tab:asymmetric metrics} are swapped: $\cd(V_{(p)},W_{(q)}) = \Delta_q$ for $p < q$, the formula for $p \geq q > 0$;
	\item $\Delta L \neq 0$ at expansions, and null curves have contractions; etc.
\end{itemize}

Symmetrizing them via different norms yields new metrics for $\Gr(n)$, but some information is lost.
\SELFR{$\ell^2$ norm turns $\dpF$ into $\tilde{d}_{pF}(V,W) = \sqrt{2\sum_{i=1}^p \sin^2 \theta_i}$ if $p = q$, otherwise $\sqrt{\max\{p,q\} + \sum_{i=1}^{\min\{p,q\}} \sin^2 \theta_i}$}
Max-symmetrized ones are constant for distinct dimensions, but can still have some uses: 
e.g., $\MFS$ is an angle $\hat{\Theta}_{V,W} = \max\{\Theta_{V,W},\Theta_{W,V}\}$,
which extends the second formula in \eqref{eq:norm contr} for $p \neq q$, as $|\inner{A,B}| = \|A\|\|B\|\cos\hat{\Theta}_{[A],[B]}$ \cite{Mandolesi_Products}.

The definition of $d(V,W)$ in \eqref{eq:asym metric} takes the infimum in $[0,\Delta_p]$,
but it can be tweaked.
With $[0,\infty]$ we obtain $d(V,W) = \infty$ if $\dim V > \dim W$, which may be useful in applications where dimensions can increase but not decrease 
(or vice versa, with $\cd$).
With $[0,\Delta_n]$, 
\Cref{pr:antiisometry} extends to the $\ell^2$ case; 
in \Cref{pr:length partition}, $\Delta L = \Delta_{n,d}$ for contractions no longer depends on $p$; 
\Cref{df:types}(\ref{df:type II}) no longer needs $\Delta_{q,d} = \Delta_{p,d}$;
type II paths can expand before contracting even in the $\ell^2$ case; etc.

Links of (asymmetric or not) $\wedge$ metrics with Grassmann or Clifford geometric algebras have been little explored in applications, possibly due to lack of familiarity.
Asymmetric angles and projection factors provide a bridge:
one can take advantage of their properties \cite{Mandolesi_Grassmann,Mandolesi_Pythagorean,Mandolesi_Products,Mandolesi_Trigonometry},
given by the algebras, without using these directly
(although anyone working with subspaces would benefit from learning the basics of such formalisms).

The $\wedge$ metrics might not be suitable for all uses, as they quickly approach their maximum values if several principal angles are large, or many are not too small.
Still, this can be helpful in some problems:
e.g., it is crucial for quantum decoherence \cite{Schlosshauer2007}, allowing many-particle quantum states to quickly become nearly orthogonal, since their Fubini-Study distance approaches $\frac\pi2$ very fast as perturbations propagate.

High dimensional asymmetric metric spaces with measure can, under some conditions, be nearly symmetric \cite{Stojmirovic2004}, \ie have negligible \emph{asymmetry} $|d(x,y)-d(y,x)|$ outside a small measure set of pairs $(x,y)$,
due to concentration of measure \cite{Ledoux2001}.
Though the details have yet to be worked out, this seems to be the case with $(\Gr(n),\dFS)$ for large $n$.
The $\Gr_p(n)$'s have natural measures given by the action of the unitary group, 
	\SELF{$O(n)$ has Haar measure, so given $V_0$ and $S\subset \Gr_p(n)$, $\mu(S) = \mu\{T\in O(n): T(V_0) \in S\}$}
and taking, for example, their dimensions $p(n-p)$ as relative weights, we obtain a Borel measure in $\Gr(n)$.
	\SELF{for its finer topology $\tau$, hence also for $\tau^\pm$} 
If $n$ is large, most pairs $(V,W)$ should have many non-negligible principal angles, 
so that $\dFS(V,W)  \cong \frac\pi2 \cong \dFS(W,V)$.

Decoherence is deemed responsible for the quantum-classical transition as the dimension of the quantum state space increases, 
so this begs the question of whether this transition might be linked to a loss of asymmetry.
As far as we know, there have been no attempts to relate quantum theory to asymmetric metric spaces.
But such relation would not be surprising, as fermionic Fock space is the projective space of a Grassmann algebra, the Fubini-Study metric $\dFS$ appears naturally in the theory, and the asymmetric angle $\Theta_{V,W}$ has led to complex volumetric Pythagorean theorems \cite{Mandolesi_Pythagorean} with unexpected links to quantum probabilities \cite{Mandolesi_Born}.

\appendix

\section{Grassmann algebra}\label{sc:Grassmann algebra}

The \emph{Grassmann (exterior) algebra} \cite{Kozlov2000I,MacLane1988,Dorst2007}
of $V\in \Gr(\F^n)$ is a graded algebra $\bigwedge V = \bigoplus_{p\in\Z} \bigwedge^p V$, with \emph{exterior powers} $\bigwedge^0 V = \F$, $\bigwedge^1 V = V$, and $\bigwedge^p V = 0$ for $p \not\in [0,\dim V]$.\label{df:Grass alg}
For $A\in \bigwedge^p V$ and $B\in \bigwedge^q V$,
its bilinear associative \emph{exterior product} $\wedge$
satisfies%
	\footnote{The sign is related to orientations, which are important for some uses.}
$A\wedge B = (-1)^{pq} B\wedge A \in \bigwedge^{p+q} V$.
For $u,v\in V$, 
\begin{equation}\label{eq:wedge}
	u \wedge v = -v\wedge u \quad\text{and}\quad v\wedge v=0.
\end{equation}

For $p>0$, elements of $\bigwedge^p V$ are sums of \emph{($p$-)blades}, or \emph{decomposable $p$-vectors},
$B = v_1\wedge\cdots\wedge v_p$ for $v_1,\ldots,v_p\in V$. 
And
$B \neq 0$ $\Leftrightarrow$ $v_1,\ldots,v_p$ are linearly independent, in which case $B$ represents the $p$-subspace 
\begin{equation}\label{eq:blade space}
	[B] = \{v\in V: v\wedge B=0\} = \Span\{v_1,\ldots,v_p\}.
\end{equation}
A $0$-blade is any $\lambda\in\F$, and $[\lambda] = 0$.
	\SELF{Only represents $0$, and has $V=\Ann(\lambda)$, $\bigwedge^0 V=\Span\{\lambda\}$, if $\lambda \neq 0$} 

For nonzero blades $A,B\in \bigwedge V$,
if $[A] \cap [B] = 0$ then $[A\wedge B] = [A]\oplus [B]$, and otherwise $A\wedge B = 0$.
For nonzero $p$-blades $A,B\in \bigwedge^p V$, $A+B$ is still a blade 
$\Leftrightarrow \dim ([A]\cap[B]) \geq p-1$ \cite[p.\,555]{MacLane1988}.

The inner product $\inner{A,B} = \det\!\big(\inner{v_i,w_j}\big)$ of $A=v_1\wedge\cdots\wedge v_p$ and $B=w_1\wedge\cdots\wedge w_p$ \label{df:inner AB}
	\OMIT{and $\inner{\kappa,\lambda}=\bar{\kappa}\lambda$ for  $\kappa,\lambda \in \F$.}
is extended linearly (sesquilinearly, if $\F=\C$), with distinct $\bigwedge^p V$'s set as  orthogonal.
If $\F=\R$, $\|A\|=\sqrt{\inner{A,A}}$ is the $p$-dimensional volume of the parallelotope spanned by $v_1,\ldots,v_p$, often used to represent $A$ (Fig.\,\ref{fig:blades}).
If $\F=\C$, $\|A\|^2$ is the $2p$-dimensional volume of that spanned by $v_1,\im v_1,\ldots,v_p, \im v_p$ in the underlying real space \cite{Mandolesi_ComplexDet}.
	\OMIT{If $[A] \perp [B]$, $\|A\wedge B\| = \|A\|\|B\|$.
	For $u,v \in V$, $\|u\wedge v\| = \|u\|\|v\|\sin \gamma_{u,v}$.}

\begin{figure}
	\centering
	\includegraphics[width=0.75\linewidth]{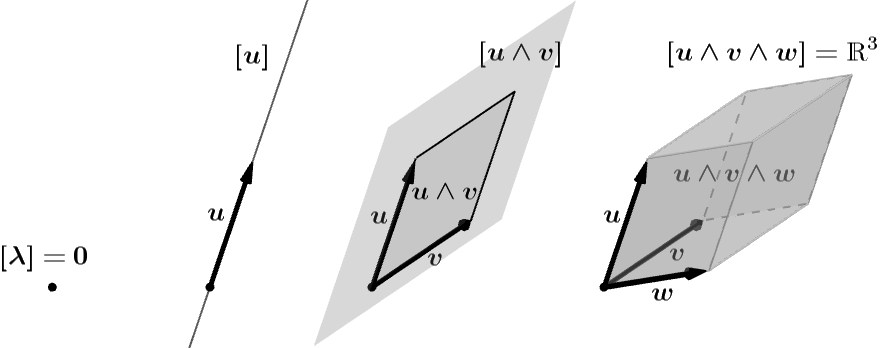}
	\caption{In $\bigwedge \R^3$, a 0-blade is a scalar $\lambda$, a 1-blade is a vector $u$, a 2-blade $u\wedge v$ is shown as a parallelogram, and a 3-blade $u\wedge v\wedge w$ as a parallelepiped. Their subspaces are a point $[\lambda]=0$, a line $[u]$, a plane $[u\wedge v]$ and the space $[u\wedge v\wedge w] = \R^3$. Norms are $|\lambda|$, the length of $u$, area of $u\wedge v$ and volume of $u\wedge v\wedge w$. Shapes do not matter: e.g., any (oriented) area of same value in the same plane also represents $u \wedge v$.}
	\label{fig:blades}
\end{figure}

Given an orthonormal basis $\{v_1,\ldots,v_q\}$ of $V \in \Gr_q(\F^n)$, 
another one of $\bigwedge^p V$ is $\{v_{i_1}\wedge\cdots\wedge v_{i_p}:1\leq i_1 < \cdots<i_p\leq q\}$,
so $\bigwedge^p V \cong \F^{\binom{q}{p}}$.
In particular, $\bigwedge^q V = \Span\{v_1 \wedge\cdots\wedge v_q\}$ is a line in $\bigwedge^q \F^n$.
For $0 \neq B \in \bigwedge \F^n$,
we have $\bigwedge^q V = \Span\{B\} \Leftrightarrow B$ is a $q$-blade and $V=[B]$.

The orthogonal projection $P_V:\F^n\rightarrow V$ extends to an orthogonal projection $P_V:\bigwedge \F^n\rightarrow \bigwedge V$ with $P_V(A\wedge B) = P_V A \wedge P_V B$ for $A,B\in \bigwedge \F^n$.
If $A \in \bigwedge^p \F^n$ then $P_V A \in \bigwedge^p V$, 
so $P_V A = 0$ if $p>\dim V$.
And $P_V \lambda = \lambda$ for $\lambda \in \F$.

The \emph{interior product}%
	\footnote{Also called \emph{contraction}. There are different conventions, differing by certain signs \cite{Mandolesi_Contractions}.}
$A\lcontr B \in\bigwedge^{q-p} \F^n$ of $A\in\bigwedge^p \F^n$ on $B\in\bigwedge^q \F^n$ is defined by $\inner{C,A \lcontr B} = \inner{A\wedge C,B}$ for all $C\in\bigwedge^{q-p} \F^n$.\label{df:contr}
It is bilinear (sesquilinear if $\F=\C$), $A\lcontr B = \inner{A,B}$ if $p=q$, and $A\lcontr B = 0$ if $p>q$ (so it is asymmetric).
Many properties facilitate its use \cite{Mandolesi_Contractions,Mandolesi_Contractions2}:
e.g., if $\{u_1,\ldots,u_n\}$ is orthonormal, $i\neq j$, and $k, l, m$ are distinct, then
\begin{equation}\label{eq:contraction orthonormal}
		(u_i \wedge u_j) \lcontr (u_k \wedge u_l \wedge u_m) = 
		\begin{cases}
			0 &\text{if } \{i,j\} \not\subset \{k,l,m\}, \\
			u_m &\text{if } i=k \text{ and } j=l,
		\end{cases}
\end{equation}
and other cases follow via anticommutativity of $\wedge$.
For $A = v_1\wedge\cdots\wedge v_p$ and $B = w_1\wedge\cdots\wedge w_q \neq 0$,
using matrices (though it is not the best way) $\mathbf{B} = \big(\inner{w_i,w_j}\big)_{q \times q}$ and
$\mathbf{C} =\big(\inner{w_i,v_j}\big)_{q \times p}$ we have \cite{Mandolesi_Contractions}
\begin{equation}\label{eq:norm contr matrices}
	\|A\lcontr B\| = \sqrt{\det \mathbf{B} \cdot \det(\mathbf{C}^\dagger \mathbf{B}^{-1}\mathbf{C})}.
\end{equation}

\section{Metrics in Grassmannians $\Gr_p(n)$}\label{sc:Metrics and distances on Grassmannians}

Here we review and classify the main metrics used in $\Gr_p(n)$ \cite{Edelman1999,Deza2016,Stewart1990}%
	\footnote{In the notation of \cite{Stewart1990}, $\rho_\mathrm{b}=\dcF$, $\rho_\mathrm{p,F} = \sqrt{2}\dpF$, $\rho_\theta = \dFS$, $\rho_\mathrm{s} = \dBC$ and $\rho_\mathrm{g,2}=\dpd$.},
and prove some inequalities.

\begin{figure}[t]
	\centering
	\includegraphics[width=0.38\linewidth]{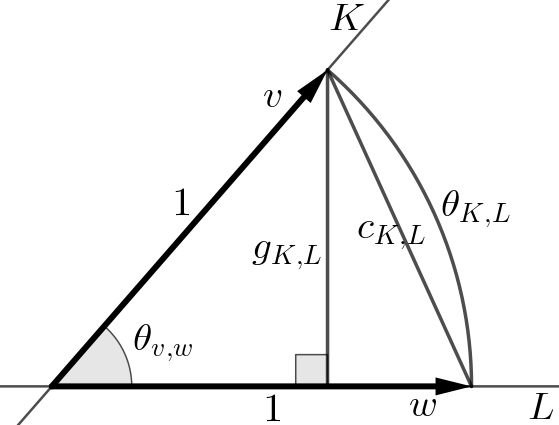}
	\caption{For unit $v$ and $w$ with $\inner{v,w} \geq 0$, the \emph{angular}, \emph{chordal} and \emph{gap distances} of $K=\Span\{v\}$ and $L=\Span\{w\}$ are
		$\theta_{K,L} = \theta_{v,w}$, 
		$c_{K,L} = \|v-w\| = 2\sin\frac{\theta_{v,w}}{2}$ and 
		$g_{K,L} = \|v-P_{L}v\| = \sin \theta_{v,w}$, respectively.
		If $K\neq L$, $\frac\pi2 g_{K,L} \geq \theta_{K,L} > c_{K,L} > g_{K,L}$.}
	\label{fig:distances_lines}
\end{figure}

Though often presented as a loose collection of formulas, these metrics derive from distances between lines (Fig.\,\ref{fig:distances_lines}),
	\SELF{$\Span\{u\}$ ao invés de $[u]$ para não confundir quando  usar com lines spanned by blades}
and fit into a natural scheme.
\Cref{tab:metrics same dim} classifies them according to the line distance (angular, chordal or gap) and which lines are used or how they are combined ($\ell^2$, $\wedge$ or max).
Their usual formulas are given in \Cref{tab:metrics Gpn},%
	\SELF{$P_V = \mathbf{A}\mathbf{A}^\dagger$, $P_W = \mathbf{B}\mathbf{B}^\dagger$, $\tr(P_V P_W) = \|\mathbf{A}^\dagger\mathbf{B}\|_\mathrm{F}^2 = \sum_{i=1}^m \cos^2\theta_i$}
and they are described below with the notation from the figure and tables.
	\SELF{$\theta$ more fundamental \cite{Qiu2005} as $c$, $g$ are concave functions of it; triangle ineq attains equal for $c$, $g$ only trivially (if 2 lines coincide), $\PR(X)$ is convex for $\theta$. Geodesics in $\PR(\R^n)$ are quotients of great circles of $S^{n-1}$ by antipodals; in $\PR(\C^n)$ great circles in the $\PR(\C^2)\cong S^2$ given by 2 distinct complex lines.}

\begin{table}[t]
	\centering
	\renewcommand{\arraystretch}{1}
	\begin{tabular}{llll}
		\toprule
		Type & angular & chordal & gap \\
		\cmidrule(lr){1-1} \cmidrule(lr){2-2} \cmidrule(lr){3-3} \cmidrule(lr){4-4}
		$\ell^2$ & $\dg = \sqrt{\sum_{i=1}^p \theta_{K_i,L_i}^2}$ & $\dcF = \sqrt{\sum_{i=1}^p c_{K_i,L_i}^2}$ & $\dpF = \sqrt{\sum_{i=1}^p g_{K_i,L_i}^2}$
		\\[6pt]	
		$\wedge$ & $\dFS = \theta_{K_\wedge,L_\wedge}$ & $\dcw = c_{K_\wedge,L_\wedge}$ & $\dBC = g_{K_\wedge,L_\wedge}$
		\\[6pt]
		max & $\dA = \theta_{K_p,L_p}$ & $\dcd = c_{K_p,L_p}$ & $\dpd = g_{K_p,L_p}$
		\\
		\bottomrule
	\end{tabular}
	\caption{Metrics in $\Gr_p(n)$ in terms of angular, chordal or gap distances between lines $K_i=\Span\{e_i\}$ and $L_i=\Span\{f_i\}$, or $K_\wedge = \Span\{A\}$ and $L_\wedge = \Span\{B\}$, for principal bases $(e_1,\ldots,e_p)$ and $(f_1,\ldots,f_p)$ of $V$ and $W$, $A=e_1\wedge\cdots\wedge e_p$ and $B=f_1\wedge\cdots\wedge f_p$.}
	\label{tab:metrics same dim}
\end{table} 

\begin{table}[t]
	\centering
	\renewcommand{\arraystretch}{1}
	\begin{tabular}{ll@{}}
		\toprule
		Metric & Formulas
		\\
		\cmidrule(lr){1-1} \cmidrule(lr){2-2}
		$\dg$ & $\sqrt{\sum_{i=1}^p \theta_i^2}$ 
		\\[3pt]
		$\dcF$ &  $2\sqrt{\sum_{i=1}^p \sin^2 \frac{\theta_i}{2}} = \sqrt{\sum_{i=1}^p \|e_i - f_i\|^2} =  \|\mathbf{E}-\mathbf{F}\|_{\mathrm{F}}$. 
		\\[3pt] 
		$\dpF$ &  $\sqrt{\sum_{i=1}^p \sin^2 \theta_i} = \sqrt{\sum_{i=1}^p \|e_i - P_W e_i\|^2} = \frac{\|P_V-P_W\|_{\mathrm{F}}}{\sqrt{2}}  = \sqrt{p - \|\mathbf{A}^\dagger\mathbf{B}\|_\mathrm{F}^2}$ 
		\\[6pt]
		$\dFS$ &  $\cos^{-1}(\prod_{i=1}^p \cos\theta_i) = \theta_{A,B} = \cos^{-1} |\det \mathbf{A}^\dagger\mathbf{B}|$
		\\[3pt]
		$\dcw$ &  $\sqrt{2-2\prod_{i=1}^p \cos\theta_i} = 2\sin \frac{\theta_{A,B}}{2} = \|A-B\|$
		\\[3pt] 
		$\dBC$ & $\sqrt{1-\prod_{i=1}^p \cos^2\theta_i} = \sin\theta_{A,B} = \|A-P_W A\| = \sqrt{1-|\det \mathbf{A}^\dagger\mathbf{B}|^2}$
		\\[6pt] 
		$\dA$ & $\theta_p = \sin^{-1} \|P_V-P_W\|_2 $ 
		\\[3pt] 
		$\dcd$ &  $2\sin \frac{\theta_p}{2} = \|e_p-f_p\| =  \|\mathbf{E}-\mathbf{F}\|_2 = \!\!\max\limits_{v\in V, \|v\|=1} \min\limits_{w\in W, \|w\|=1} \|v-w\|$ 
		\\[3pt]
		$\dpd$ &  $\sin \theta_p = \|e_p - P_W e_p\| = \|P_V-P_W\|_2 = \max\limits_{v\in V, \|v\|=1} \|v-P_W v\|$ 
		\\
		\bottomrule
	\end{tabular}
		\caption{Distances between $V,W\in \Gr_p(n)$, with
		principal angles $\theta_1\leq\cdots\leq\theta_p$,
		unit blades $A \in \bigwedge^p V$ and $B \in \bigwedge^p W$ with $\inner{A,B} \geq 0$, 
		matrices $\mathbf{E}$ and $\mathbf{F}$ with principal bases $(e_1,\ldots,e_p)$ and $(f_1,\ldots,f_p)$ as columns,
		$\mathbf{A}$ and $\mathbf{B}$ with any orthonormal bases.}
	\label{tab:metrics Gpn}
\end{table}

The \emph{$\ell^2$ metrics} use the $\ell^2$ norm of distances between $K_i$ and $L_i$ for $1\leq i\leq p$.
	The \emph{geodesic metric} $\dg$ \cite{Kozlov2000III,Wong1967}
		\CITEL{Zhang2018, Zuccon2009, Lerman2011}
	is the length of minimal geodesics for the unique%
	\footnote{Up to scaling, and with an exception for $\Gr_2(\R^4)$ \cite[p.\,591]{Wong1967}.} 
		\CITE{In $\Gr_2(\R^4)$, $\dg$ still optimal: \cite[p.\,2249]{Kozlov2000I}, oriented case}
	Riemannian metric invariant by unitary maps, 
	given by the Hilbert-Schmidt product 
		\SELF{$\inner{S,T} = \tr(S^*T)$ \\ Bendokat2024 com fator $\frac12$}
	in the tangent space $\Hom(V,V^\perp)$ at $V$.
	The \emph{chordal Frobenius} $\dcF$%
		\SELF{$= \sqrt{2p-2\sum_{i=1}^p \cos\theta_i}$}
	\cite{Edelman1999,Paige1984}
		\CITE{Chikuse2012, Turaga2008}
	and \emph{projection Frobenius} $\dpF$ \cite{Hamm2008,Draper2014}
		\SELF{Tem $\frac{1}{\sqrt{2}}$ porque repete o $\|e_i - P_W e_i\| = \|P_V f_i - f_i\|$; Deza2016 omite esse fator}
	are obtained embedding $\Gr_p(n)$, respectively, in sets of $n\times p$ matrices and projections, with $\|\cdot\|_{\mathrm{F}}$.
	Other names are \emph{Grassmann} \cite{Deza2016,Ye2016} for $\dg$, \emph{Procrustes} \cite{Hamm2008,Ye2016} for $\dcF$, and, due to another embedding in a sphere, \emph{chordal} \cite{Barg2002,Conway1996,Dhillon2008} for $\dpF$.

The \emph{$\wedge$ metrics} use distances between $K_\wedge$ and $L_\wedge$ in the Grassmann algebra.
	The \emph{Fubini-Study} $\dFS$ \cite{Love2005,Pereira2021,Lu1963}
		\CITE{Dhillon2008}
	is a geodesic distance through $\PR(\bigwedge^p \F^n)$  (see \Cref{sc:Fubini-Study metric}).
	The \emph{chordal $\wedge$} metric $\dcw$ seems to be new in the literature,
	and the \emph{Binet-Cauchy} $\dBC$ is used in \cite{Hamm2008,Vishwanathan2006}.
		\SELF{$= \|\Psi_A - P_{\Psi_B} \Psi_A\|$ with $\Psi_A$ formed by \Plucker\ coord in orthonormal basis Wolf2003}

\emph{Max metrics} use only the distance of $K_p$ and $L_p$ (the largest one for the $K_i$'s and $L_i$'s).
	The \emph{Asimov} $\dA$ \cite{Asimov1985,Weinstein2000} 
		\CITER{Weinstein2000 Appendix. \\ $\dA$ está errada em Ye2016}
	is the 
		\SELF{largest angular dist $S(V)$ to $S(W)$; cos is largest semi-axis of $P_W(S(V))$}
	geodesic distance for a Finsler metric 
		\SELF{Finsler é dada por norma no espaço tang, simétrica ou não}
	given by $\|\cdot\|_2$ in $\Hom(V,V^\perp)$.
	The \emph{chordal 2-norm} $\dcd$ \cite{Barg2002} 
		\SELF{largest dist $S(V)$ to $S(W)$ (Hausdorff dist.); 
			`2-norm' is norm in $\F^n$}
	and \emph{projection 2-norm} $\dpd$ \cite{Edelman1999,Barg2002,Ye2016}
		\SELF{largest dist. from $S(V)$ to $W$; in normed spaces gap is not metric {Kato1995}}
	originate like $\dcF$ and $\dpF$, but with $\|\cdot\|_2$.
	Other names are \emph{spectral distance} \cite{Deza2016,Ye2016}  for $\dcd$, and \emph{gap} \cite{Kato1995,Stewart1990} or \emph{min-correlation} \cite{Hamm2008} for $\dpd$.%
		\CITE{\emph{containment gap} or \emph{projection distance} in {Deza2016}}

Other families of metrics are obtained in \cite{Stewart1990,Qiu2005} via norms of principal angles or line distances (e.g., $\sqrt[r]{\sum_{i=1}^p \theta_i^r}$ for the $\ell^r$ norm).
They include our $\ell^2$ and max ones, and the $\wedge$ ones fit into the same structure, except that lines are in $\bigwedge^p \F^n$.
The following distances do not fit, and are not metrics in $\Gr_p(n)$, though sometimes they have been treated as such:

\begin{itemize}
	\item The \emph{max-correlation} or \emph{spectral distance} \cite{Hamm2008,Dhillon2008} $d = \sin \theta_1$ 
		\SELFR{smallest dist from $S(V)$ to $W$}
	does not satisfy a triangle inequality, and $d(V,W) = 0 \Leftrightarrow V\cap W \neq 0$.
	
	\item The \emph{Martin distance} \cite{Deza2016,Ye2016,Martin_2000,De_Cock_2002} 
	for certain subspaces 
		\CITE{spanned by vector of the form $(1,\alpha,\alpha^2,\ldots)$ with $|\alpha|<1$}
	associated to ARMA
		\CITE{Auto-Regressive Moving Average} 
	models is $d = \sqrt{-\log \prod_{i=1}^p\cos^2\theta_i}$. 
	For general subspaces it does not satisfy a triangle inequality (e.g., take 3 lines in $\R^2$, separated by angles $\theta$, $\theta$ and $2\theta$), and $d(V,W) =\infty$ if $V\pperp W$.
\end{itemize}

The results below imply the well known topological equivalence of the metrics of \Cref{tab:metrics Gpn} (they all give $\Gr_p(n)$ the same topology). We provide proofs as we could not find them in the literature, and because there have been imprecise statements: 
e.g., \cite[eq.\,(4.2)]{Edelman1999}
	\OMIT{also (4.2)--(4.5)}
gives $\dg > \dFS$ for $V \neq W$, but if $\dim(V\cap W) = p-1$ then $\theta_1 = \cdots = \theta_{p-1} = 0$ and $\dg = \dFS = \theta_p$.

Note how \Cref{tab:metrics same dim} reflects the inequalities: distances decrease as one moves right (strictly, if $V \neq W$) or down (strictly, if $\dim(V\cap W) \leq p-2$) in it.
\CITE{$\dg>\dFS$, $\dcF>\dcd$, $\dpF>\dpd$}
The distances can be compared, for $p=2$, in the graphs of Fig.\,\ref{fig:distances Gp}.

\begin{figure}
	\centering
	\begin{subfigure}[b]{0.325\textwidth}
		\includegraphics[width=\textwidth]{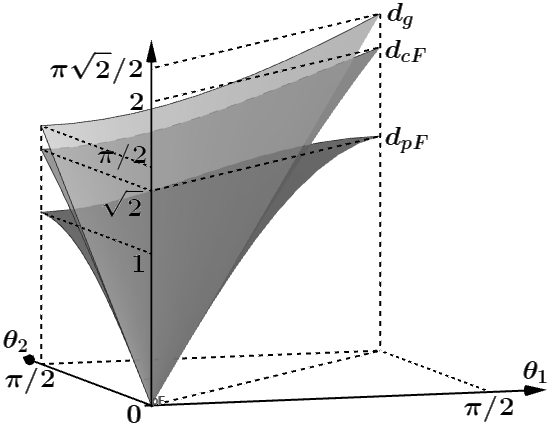}
		\caption{$\ell^2$ metrics}
	\end{subfigure}
	\begin{subfigure}[b]{0.325\textwidth}
		\includegraphics[width=\textwidth]{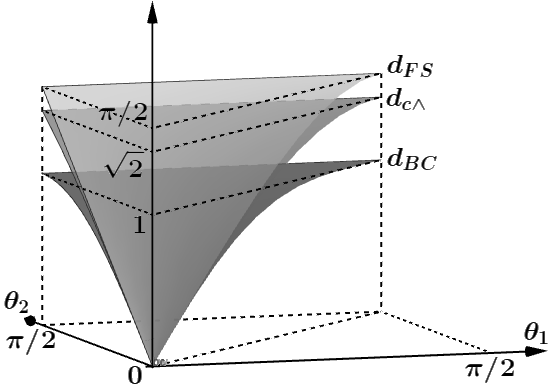}
		\caption{$\wedge$ metrics}
	\end{subfigure}
	\begin{subfigure}[b]{0.325\textwidth}
		\includegraphics[width=\textwidth]{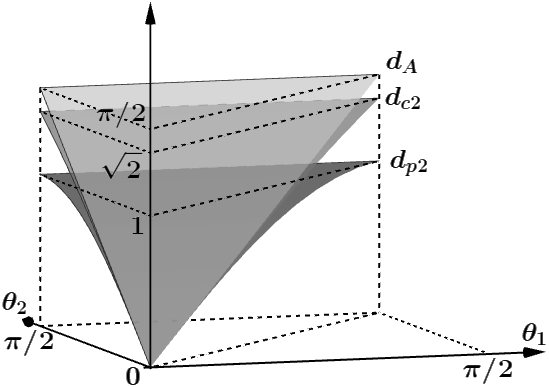}
		\caption{max metrics}
	\end{subfigure}
	\begin{subfigure}[b]{0.325\textwidth}
		\includegraphics[width=\textwidth]{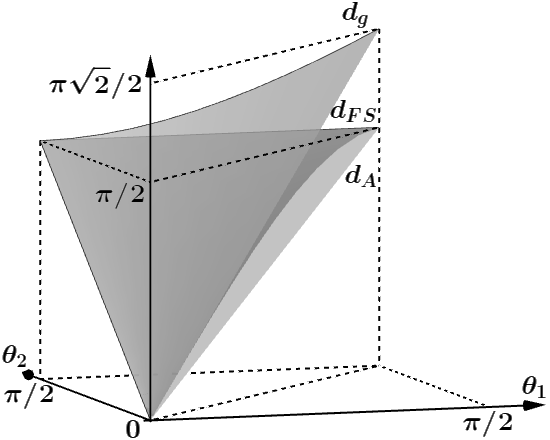}
		\caption{angular metrics}
	\end{subfigure}
	\begin{subfigure}[b]{0.325\textwidth}
		\includegraphics[width=\textwidth]{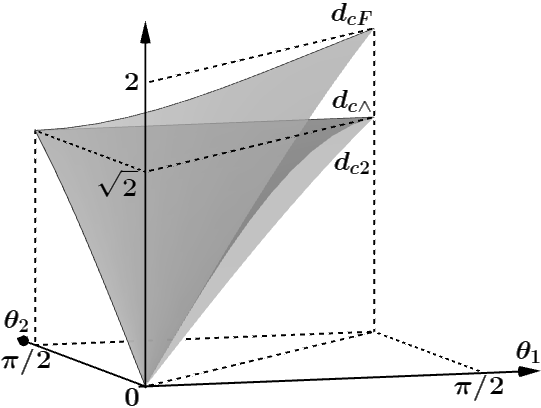}
		\caption{chordal metrics}
	\end{subfigure}
	\begin{subfigure}[b]{0.325\textwidth}
		\includegraphics[width=\textwidth]{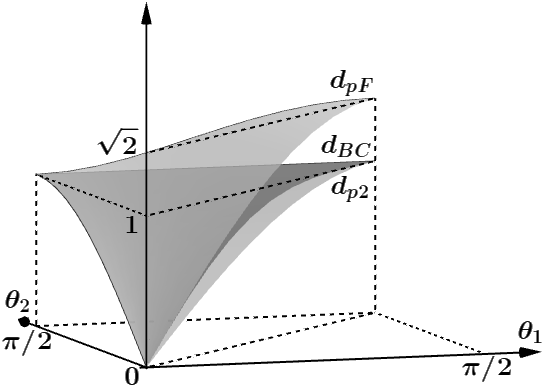}
		\caption{gap metrics}
	\end{subfigure}
	\caption{Distances in $\Gr_2(n)$ as functions of principal angles $\theta_1\leq\theta_2$}
	\label{fig:distances Gp}
\end{figure}

\begin{proposition}\label{pr:ineqs1}
	For distances from $V$ to $W$ in $\Gr_p(n)$, with $V \neq W$:
	\begin{enumerate}[(i)]
		\item $\frac\pi2 \dpF \geq \dg > \dcF > \dpF$.
		\item $\frac\pi2 \dBC \geq \dFS > \dcw > \dBC$.
		\item $\frac\pi2 \dpd \geq \dA > \dcd > \dpd$.
	\end{enumerate}
\end{proposition}
\begin{proof}
	Follow from the formulas in \Cref{tab:metrics same dim} and inequalities in Fig.\,\ref{fig:distances_lines}.
\end{proof}

\begin{proposition}\label{pr:ineqs2}
	For distances from $V$ to $W$ in $\Gr_p(n)$, the inequalities below hold if	$\dim(V\cap W) \leq p-2$, and otherwise
	all $>$'s become $=$'s.
		\SELFR{no lugar de $\sqrt{p}$ pode ser $\sqrt{p-r}$, com $r=\dim V\cap W$}
	\begin{enumerate}[(i)]
		\item $\sqrt{p}\, \dA \geq \dg > \dFS > \dA$. \label{it:angular ineq}
		\item $\sqrt{p}\, \dcd \geq \dcF > \dcw > \dcd$. \label{it:chordal ineq}
		\item $\sqrt{p}\, \dpd \geq \dpF > \dBC > \dpd$. \label{it:gap ineq}
	\end{enumerate}
\end{proposition}
\begin{proof}
	If $\dim(V\cap W)\geq p-1$ then $\theta_1=\cdots=\theta_{p-1}=0$, and the formulas in \Cref{tab:metrics Gpn} give the equalities.
	If $\dim(V\cap W) \leq p-2$ then $\theta_{p-1},\theta_p\neq 0$.
	We prove the second inequality in each item, as the others are simple.
	
	(\ref{it:angular ineq})
	We must show $\sqrt{\theta_1^2 + \cdots + \theta_p^2} > \cos^{-1}(\cos\theta_1 \cdots \cos\theta_p)$ 
	for $\theta_i \in [0,\frac\pi2]$ 
	with $\theta_{p-1},\theta_p\neq 0$.
	Assume $\theta_p \neq \frac\pi2$, otherwise it is trivial.
	
	For $p=2$,
	$f(x,y) = \sqrt{x^2+y^2} > \cos^{-1}(\cos x \cos y) = g(x,y)$ 
	for $x,y\in\ (0,\frac\pi2)$
	since
	$f(0,y) = g(0,y)$ and, 
	as $\frac{\partial f}{\partial x} = \frac{x}{\sqrt{x^2+y^2}}$ is increasing in $x$,
		\OMITR{$\frac{\partial^2 f}{\partial x^2}>0$; and decreasing in $y$}
	$\frac{\partial f}{\partial x}  > \frac{\sin x}{\sqrt{\sin^2 x + \tan^2 y}} = \frac{\partial g}{\partial x}$.
	Assuming the result for some $p\geq 2$, 
	let 
	$x = \sqrt{\theta_1^2 + \cdots + \theta_p^2} > \cos^{-1}(\cos\theta_1 \cdots \cos\theta_p) > 0$.
	If $x > \frac\pi2$ the result for $p+1$ is trivial, and otherwise
	$\sqrt{\theta_1^2 + \cdots + \theta_{p+1}^2} = \sqrt{x^2 + \theta_{p+1}^2} > \cos^{-1}(\cos x \cos\theta_{p+1}) \geq \cos^{-1}(\cos\theta_1 \cdots \cos\theta_{p+1})$.
		\OMITR{Last $\geq$ not strict if $\theta_{p+1}=\frac\pi2$}
	
	(\ref{it:chordal ineq})
	$\dcF = \sqrt{2p-2\sum_{i=1}^p x_i}$ and $\dcw = \sqrt{2-2\prod_{i=1}^p x_i}$ for $x_i = \cos\theta_i$, 
	so we show $p-\sum_{i=1}^p x_i > 1-\prod_{i=1}^p x_i$ for $x_i\in [0,1]$ with $x_{p-1},x_p \neq 1$.%
		\OMITR{$x_i=\cos\theta_i$}
	
	If $p=2$ it is easy to check that 
	$2-x_1-x_2 > 1-x_1x_2$.
		\OMIT{$2-x_1-x_2 = 1 - x_1x_2 + (1-x_1)(1-x_2) > 1-x_1x_2$}
	Assuming the result for some $p\geq 2$, let 
	$x = 1-(p - \sum_{i=1}^p x_i) < \prod_{i=1}^p x_i < 1$.
	If $x \geq 0$ then
		\OMIT{Last $\geq$ not strict if $x_{p+1}=0$}
	$p+1-\sum_{i=1}^{p+1} x_i = 2 - x - x_{p+1} > 1 - x x_{p+1} \geq 1-\prod_{i=1}^{p+1} x_i$,
	and otherwise
	$p+1-\sum_{i=1}^{p+1} x_i = 2 - x - x_{p+1} > 1 \geq 1-\prod_{i=1}^{p+1} x_i$.
	
	
	(\ref{it:gap ineq})
	$\dpF = \sqrt{p-\sum_{i=1}^p x_i}$ and $\dBC = \sqrt{1-\prod_{i=1}^p x_i}$ for $x_i=\cos^2\theta_i$, so it follows as in (\ref{it:chordal ineq}).
\end{proof}

\section{Fubini-Study distances}\label{sc:Fubini-Study metric}

The Fubini-Study distance%
	\footnote{In $\PR(\F^n)$, `Fubini-Study metric' usually refers to a Riemannian metric that gives $\dFS$.}
\cite{Deza2016,Goldman1999,Bengtsson2017}
	\CITER{Bengtsson, Zyczkowski p.141}
originates in the projective space $\PR(\F^n) = \Gr_1(\F^n)$ as $\dFS(K,L) = \theta_{K,L}$ (see \Cref{df:angle lines}).%
	\CITE{real: Reid2005 p.38 \\
		complex: Goldman1999 p.16}

It transfers to $\Gr_p(n)$ via its \emph{\Plucker\ embedding} 
\cite{Harris1992}
	\CITE{Harris p.64 \\
	Griffiths, Harris p.209}
in the projective space $\PR(\bigwedge^p \F^n)$ of $\bigwedge^p \F^n \cong \F^{\binom{n}{p}}$, that maps $V \in \Gr_p(n)$ to the line $\bigwedge^p V$.
If $V,W \in\Gr_p(n)$
	\OMIT{$p\neq0$}
have principal vectors and angles $e_i$, $f_i$ and $\theta_i$
then $\bigwedge^p V = \Span\{A\}$ and $\bigwedge^p W = \Span\{B\}$ for unit blades $A = e_1\wedge\cdots\wedge e_p$ and $B = f_1\wedge\cdots\wedge f_p$,
and since
$\inner{A,B} = \prod_{i=1}^p \cos \theta_i$ we have
\begin{equation}\label{eq:usual FS}
	\begin{aligned}
		\dFS(V,W) &= \dFS(\textstyle{\bigwedge^p V, \bigwedge^p W}) = \theta_{\bigwedge^p V, \bigwedge^p W} = \theta_{A,B} \\
		&= \cos^{-1} \big(\textstyle{\prod_{i=1}^p} \cos \theta_i\big).
	\end{aligned}
\end{equation}
This is the same as a \emph{volumetric angle} $\Theta_{V,W}$ \cite{Gluck1967,Jiang1996} between subspaces of same dimension, whose cosine (in the real case) measures volume contraction in orthogonal projections between them.

For $V_{(p)}, W_{(q)} \in \Gr(n)$ with $p \neq q$, one usually takes the angle between the smaller subspace and its projection on the other (or $\frac\pi2$ if projecting reduces its dimension) \cite{Venticos1956,Jiang1996,Gunawan2005}.
This gives a \emph{minimal Fubini-Study distance} $\mFS$ \cite{Pereira2021}: for $m=\min\{p,q\}$ and principal angles $\theta_1,\ldots,\theta_m$,
\begin{equation}\label{eq:mFS}
	\mFS(V,W) = 
	\begin{cases}
		0 &\text{ if } m=0, \\
		\cos^{-1}(\prod_{i=1}^m \cos\theta_i) &\text{ if } m \neq 0.
	\end{cases}
\end{equation}
It is not a metric, as the triangle inequality fails (see p.\,\pageref{pg:triangle fail}), and its balls do not form a basis for a topology in $\Gr(n)$.

The total \Plucker\ embedding of $\Gr(n)$ in $\PR(\bigwedge \F^n)$ gives it a \emph{maximal Fubini-Study metric} $\MFS$. As the $\bigwedge^p \F^n$'s are orthogonal, 
\begin{equation}\label{eq:dFS}
	\MFS(V,W) = 
	\begin{cases}
		0 &\text{ if } p=q=0, \\
		\cos^{-1}(\prod_{i=1}^p \cos\theta_i) &\text{ if } p=q \neq 0, \\
		\frac\pi2 &\text{ if } p\neq q.
	\end{cases}
\end{equation}
It is a metric, and gives $\Gr(n)$ the disjoint union topology, as the $\Gr_p(n)$'s are separated by a fixed distance of $\frac\pi2$.

\subsection{Convexity in $(\Gr_p(n),\dFS)$}

Here we extend convexity results of $(\Gr_p(n),\dFS)$ to the complex case.
	\SELF{Qiu2005: $\dg$ has triangle equality in direct rotations, $\dcF, \dpF, \dcd, \dpd$ only trivially}
	
First, in $\Gr_1(n) = \PR(\F^n)$, we give equality conditions for the triangle inequality of $\dFS(K,L) = \theta_{K,L}$.
They are obvious in the real case, and still hold in the complex one, as we show.
Please recall \Cref{df:angle lines}.

\begin{proposition}\label{pr:PR geometry}
	In $\PR(\F^n)$, $\dFS$ is intrinsic.
	For distinct 
	\SELF{Precisa distinct para a volta}
	$J,K,L \in \PR(\F^n)$, 
	$J \in (K,L)_{\dFS} \Leftrightarrow J$ lies in a segment from $K$ to $L$ $\Leftrightarrow$ $J$, $K$, $L$ are respectively spanned by aligned $u,v,w\in \F^n$ 
	\OMITR{nonzero}
	with $u= v + w$.%
		\OMIT{\ref{scrpr:spherical triangle inequality lines} has direct proof}
\end{proposition}
\begin{proof}
	$\dFS = \dg$ in $\Gr_1(n)$, so it is intrinsic.
	For the last equivalence:
		\OMIT{first one follows from intrinsic}
	
	($\Rightarrow$)
	Let $K$ and $L$ have principal vectors $e$ and $f$ and principal angle $\theta \neq 0$.
	As $J = \mu(t)$ for $\mu$ as in \eqref{eq:geodesic Grp} and some $t \in (0,1)$, it is spanned by
	$u = \cos(t\theta) e + \sin (t\theta) \frac{f - P_{K} f}{\|f - P_{K} f\|} 
	= \cos(t\theta) e + \sin (t\theta) \frac{f - \cos(\theta) e}{\sin (\theta)}
	= v + w$
	for $v = \frac{\sin((1-t)\theta)}{\sin(\theta)} e$
	and $w = \frac{\sin(t\theta)}{\sin(\theta)} f$, which are aligned (with $u$ as well).

	($\Leftarrow$)
	As $u= v +  w$, 
	\OMIT{O cálculo direto com $\inner{\cdot,\cdot}$ é trabalhoso, mas padrão}
	they are coplanar vectors (in $\R^n$ or $(\C^n)_\R$) with $u$ between $v$ and $w$, so
	$\theta_{K,L} = \theta_{v,w}  = \theta_{v,u} + \theta_{u,w} = \theta_{K,J} + \theta_{J,L}$.
		\OMIT{as $u, v, w$ are aligned}
\end{proof}

Alignment is needed so $\theta_{K,L} = \theta_{v,w}$, for example.
If $\F=\R$, it just ensures $\theta_{v,w} \in [0,\frac\pi2]$, but if $\F=\C$ its role is more subtle:

\begin{example}
	In $\C^2$, $v=(\im,0)$, $w = (\frac{\sqrt{2}}{2},\frac{\sqrt{2}}{2})$ and $u=v+w$ are not aligned. 
	Though $\theta_{v,w}  = \theta_{v,u} + \theta_{u,w}$, for $J,K,L$ as above we find $\theta_{K,L} = \phi_{v,w} = 45^\circ$, by \eqref{eq:angles}, but $\theta_{K,J} = \theta_{J,L} = 30^\circ$, so $J \notin (K,L)_{\dFS}$.
\end{example}

The conditions in $\Gr_p(n)$ are obtained via its \Plucker\ embedding and:

\begin{lemma}\label{pr:linear combination blades}
	Let $B,C \in \bigwedge^p \F^n$ be nonzero blades with $[B] \neq [C]$.
	If $B + C$ is still a blade then $B=v\wedge D$ and $C=w\wedge D$
	for a unit blade $D\in\bigwedge^{p-1} \F^n$ 
		\SELF{$[D] = [B]\cap [C]$}
	and $v,w\in [D]^\perp$,
	with $\inner{v,w} = \inner{B,C}$.
\end{lemma}
\begin{proof}
	If $B+C$ is a blade and $[B] \neq [C]$ then $\dim ([B]\cap[C]) = p-1$ (see \Cref{sc:Grassmann algebra}).
	Given a unit blade $D \in \bigwedge^{p-1} \F^n$ with $[D] = [B] \cap [C]$, there are $v,w \in [D]^\perp$ such that $B=v\wedge D$ and $C=w\wedge D$.
	And $\inner{B,C} = \inner{v\wedge D,w\wedge D} = \inner{v,w}\cdot\|D\|^2 = \inner{v,w}$.
\end{proof}

\begin{proposition}\label{pr:triangle equality same dim}
	For $U,V,W\in \Gr_p(n)$, 
		\SELF{Não precisa especificar distinct, é automático de $U \in (V,W)_{\dFS}$ ou de $J \in (K,L)_{\dFS}$}
	we have
	$U \in (V,W)_{\dFS}$ $\Leftrightarrow U = J \orthsum S$, $V = K \orthsum S$ and $W = L \orthsum S$ for $S \in \Gr_{p-1}(n)$ and lines $J,K,L \subset S^\perp$ with $J \in (K,L)_{\dFS}$.
%
\end{proposition}
\begin{proof}
	\emph{($\Rightarrow$)}
	By \eqref{eq:usual FS},
	$\tilde{J} \in (\tilde{K},\tilde{L})_{\dFS}$
	for lines $\tilde{J}= \bigwedge^p U$, $\tilde{K}= \bigwedge^p V$ and $\tilde{L}= \bigwedge^p W$.
	\Cref{pr:PR geometry}, in $\PR(\bigwedge^p \F^n)$, shows $\tilde{J}, \tilde{K}, \tilde{L}$ are spanned by aligned  $A,B,C \in \bigwedge^p \F^n$ with $A = B +  C$.
	But then (see \Cref{sc:Grassmann algebra}) these are blades, $U=[A]$, $V=[B]$ and $W=[C]$ .	
	By \Cref{pr:linear combination blades}, 
	$B=v\wedge D$, $C=w\wedge D$ and  $A = (v+w) \wedge D$
	for a blade $D\in\bigwedge^{p-1} \F^n$ and aligned
	$v,w\in [D]^\perp$.
	We have $U$, $V$, $W$ as above for $J=[v+w]$, $K=[v]$, $L=[w]$ and $S=[D]$.
	By \Cref{pr:PR geometry}, $J \in (K,L)_{\dFS}$.
		\OMIT{now in $\PR(\F^n)$}
	
	\emph{($\Leftarrow$)} 
	The only nonzero principal angle of $K \orthsum S$ and $L \orthsum S$ is $\theta_{K,L}$, so $\dFS(V,W) = \dFS( K \orthsum S, L \orthsum S) = \dFS(K,L)$, and likewise for the others.
	Thus $J \in (K,L)_{\dFS} \Rightarrow U \in (V,W)_{\dFS}$.
\end{proof}

This extends \cite[Thm.\,3]{Jiang1996} to the complex case.
	\CITE{Jiang1996: $1<p<n$, real}
Note that the result does not hold with $[V,W]_{\dFS}$, as $V \in [V,W]_{\dFS}$ $\forall\, V,W$.

\begin{proposition}\label{pr:segment Gp}
	For distinct $V, W \in (\Gr_p(n),\dFS)$, the following are equivalent:
	\begin{enumerate}[(i)]
		\item $(V,W)_{\dFS} \neq \emptyset$. \label{it:between point Gp}
		\item $\dim(V\cap W)=p-1$. \label{it:dim p-1}
		\item Minimal geodesics from $V$ to $W$ are segments. \label{it:geod segm Gp}
		\item There is a segment from $V$ to $W$. \label{it:tight geodesic Gp}
	\end{enumerate}
\end{proposition}
\begin{proof}
	(\ref{it:between point Gp}$\,\Rightarrow\,$\ref{it:dim p-1}) Follows from \Cref{pr:triangle equality same dim}.
	(\ref{it:dim p-1}$\,\Rightarrow\,$\ref{it:geod segm Gp}) By \Cref{pr:ineqs2}, $\dg(V,W) = \dFS(V,W)$.
	(\ref{it:geod segm Gp}$\,\Rightarrow\,$\ref{it:tight geodesic Gp})
	$(\Gr_p(n),\dFS)$ is geodesic space.
	(\ref{it:tight geodesic Gp}$\,\Rightarrow\,$\ref{it:between point Gp})
	In metric spaces, segments are nontrivial for $V\neq W$.
\end{proof}

\begin{corollary}\label{pr:Gp convex}
	$(\Gr_p(n),\dFS)$ is Menger convex $\Leftrightarrow$ $p\in\{0,1,n - 1, n\}$.
\end{corollary}
\begin{proof}
	Let $0<p<n$, otherwise it is trivial.
	Then
		\OMIT{\ref{pr:segment Gp}} 
	 $(V,W)_{\dFS} \neq \emptyset$ $\forall\, V \neq W$ $\Leftrightarrow$ $\dim(V\cap W)=p-1$ $\forall\, V \neq W$ $\Leftrightarrow p=1$ or $p=n-1$. 
\end{proof}

This corrects a claim in \cite[Thm.\,6]{Jiang1996}, that $\Gr_p(n)$ is not Menger convex for $n>3$ and $p>1$ (e.g., $\Gr_3(4) \cong \Gr_1(4)$ is clearly convex for $\dFS$).

\section{Asymmetric angles}\label{sc:Asymmetric angles}

As seen, the volumetric angle $\Theta_{V,W}$ \eqref{eq:usual FS} is usually extended to $\Gr(n)$ as \eqref{eq:mFS}, breaking the triangle inequality.
A better extension is \cite{Mandolesi_Grassmann,Mandolesi_Pythagorean,Mandolesi_Products,Mandolesi_Trigonometry}:

\begin{definition}\label{df:Theta pi}
	The \emph{asymmetric angle}
	and \emph{volume projection factor} from $V\in \Gr_p(\F^n)$ to $W\in \Gr_q(\F^n)$ are, respectively, 
	\begin{equation}\label{eq:def Theta pi}
		\Theta_{V,W} = \cos^{-1} \frac{\|P_W A\|}{\|A\|} \quad\text{and}\quad \pi_{V,W}=\frac{\vol_{r}(P_W(S))}{\vol_{r}(S)},
	\end{equation}
	for $0 \neq A\in \bigwedge^p V$,
	\SELFR{$\Theta_{\{0\},W}=0$ for any $W$, \\ $\Theta_{V,\{0\}}=\frac{\pi}{2}$ for $V\neq\{0\}$.}
	$P_W = P_{W_\R}$,
	$r = \dim V_\R$ ($p$ if $\F=\R$, or $2p$ if $\F=\C$), 
	and a measurable subset $S\subset V_\R$ with $r$-di\-men\-sion\-al volume $\vol_r(S) \neq 0$.
\end{definition}

So, $\pi_{V,W}$ is the factor by which (top-di\-men\-sion\-al) volumes in $V_\R$ contract if orthogonally projected to $W$.
The next result links it to $\Theta_{V,W}$, which then gains a similar geometric interpretation.

\begin{proposition}\label{pr:Theta pi}
	For $V,W \in \Gr(n)$,
	\begin{equation}\label{eq:Theta pi}
		\pi_{V,W} = 
		\begin{cases}
			\cos \Theta_{V,W} &\text{if $\F=\R$}, \\
			\cos^2 \Theta_{V,W} &\text{if $\F=\C$.}
		\end{cases}
	\end{equation}
\end{proposition}
\begin{proof}
	Follows from \eqref{eq:def Theta pi}, as blade norms (squared, if $\F=\C$) are volumes of parallelotopes \cite{Mandolesi_Pythagorean,Mandolesi_ComplexDet}.
\end{proof}

\begin{figure}[t]
	\centering
	\includegraphics[width=0.6\linewidth]{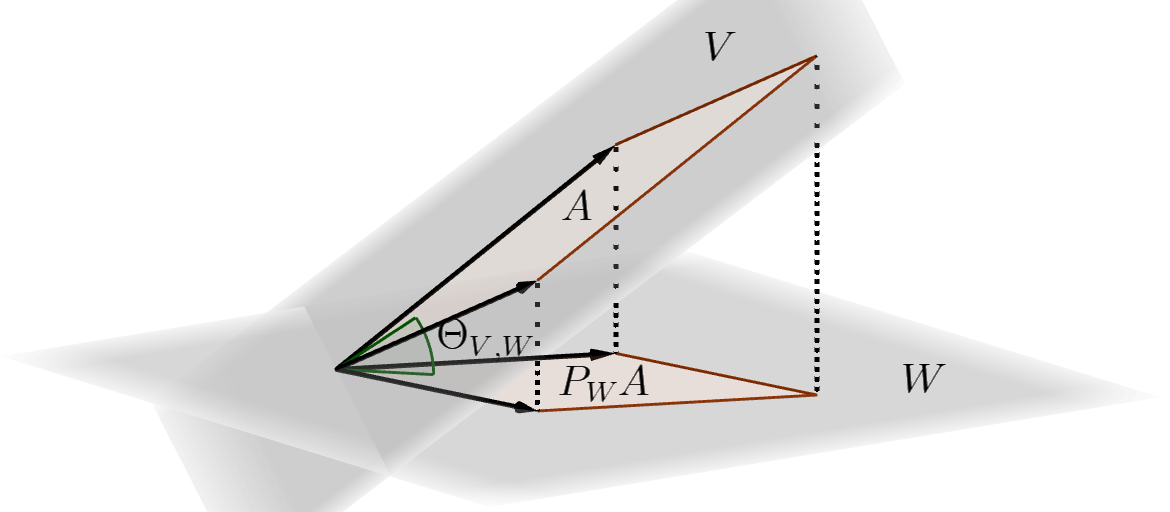}
	\caption{Orthogonal projection of a real 2-blade, $\cos \Theta_{V,W} = \frac{\|P_W A\|}{\|A\|} = \frac{\area(P_W A)}{\area(A)}$}
	\label{fig:projecao}
\end{figure}

As $\cos \Theta_{V,W}$ (squared, if $\F=\C$) measures volume contraction in orthogonal projections (Fig.\,\ref{fig:projecao}),
$\Theta_{V,W}$ does extend the volumetric angle (hence the same symbol).

In \Cref{sc:Asymmetric angle in complex spaces} we discuss the square in the complex case.
One might think it better to unify both cases by defining $\Theta_{V,W} = \cos^{-1} \pi_{V,W}$, but then most properties of this angle would become dependent on $\F$.

The next result shows $\Theta_{V,W}$ equals the asymmetric metric $\dFS$ \eqref{eq:adFS}.

\begin{proposition}\label{pr:Theta thetai}
	For $V_{(p)}, W_{(q)} \in \Gr(n)$,
	\begin{equation}\label{eq:Theta thetai}
		\Theta_{V,W} = 
		\begin{cases}
			0 &\text{ if } p=0, \\
			\cos^{-1}(\prod_{i=1}^p \cos\theta_i) &\text{ if } 0 < p \leq q, \\
			\frac\pi2 &\text{ if } p > q,
		\end{cases}
	\end{equation}
	where the $\theta_i$'s are the principal angles.
\end{proposition}
\begin{proof}
	If $p = 0$ then $1 \in \bigwedge^0 V$ and $P_W 1 = 1$, so $\Theta_{V,W} = 0$.

	If $p>q$ then $P_W A = 0$ for $A \in \bigwedge^p V$, so $\Theta_{V,W} = \frac\pi2$.

	If $0<p \leq q$ then $V$ and $W$ have principal bases and angles $(e_1,\ldots,e_p)$, $(f_1,\ldots,f_q)$ and $\theta_1 \leq\ldots\leq\theta_p$.
	Using the unit blade $A=e_1\wedge\cdots\wedge e_p$, 
	since $P_W e_i = f_i\,\cos\theta_i$ we obtain
	\[ \cos \Theta_{V,W} = \frac{\|P_W A\|}{\|A\|} = \| (f_1 \cos\theta_1) \wedge\cdots\wedge (f_p \cos\theta_p)\| = \prod_{i=1}^p \cos\theta_i. \qedhere \]
\end{proof}

The asymmetry $\Theta_{V,W}\neq \Theta_{W,V}$ for $p \neq q$ may seem strange, but reflects the dimensional asymmetry of subspaces.
It makes the angle simpler to define and compute (see \Cref{sc:Computing asymmetric angles}), and gives it better properties \cite{Mandolesi_Grassmann,Mandolesi_Products,Mandolesi_Trigonometry}.
For example, this asymmetry is crucial for the oriented triangle inequality (Fig.\,\ref{fig:oriented triangle inequality}),
and even for the apparently trivial relation 
\begin{equation}\label{eq:Theta VPV}
	\Theta_{V,W}=\Theta_{V,P_W(V)}.
\end{equation}
Symmetric angles \cite{Venticos1956,Jiang1996,Gunawan2005} corresponding to $\mFS$ \eqref{eq:mFS} have value $\frac\pi2$ for perpendicular planes $V,W\subset \R^3$, but are $0$ for $V$ and $P_W(V) = V \cap W$.
The asymmetry also lets the angle carry some important dimensional data: 
	\OMIT{\ref{pr:d min max}}
\begin{equation}\label{eq:Theta dim}
	\Theta_{V,W}\neq\frac{\pi}{2} \,\Rightarrow\, V \not\pperp W \,\Rightarrow\, \dim P_W(V) = \dim V \leq \dim W.
\end{equation}

\begin{figure}
	\centering
	\includegraphics[width=.6\textwidth]{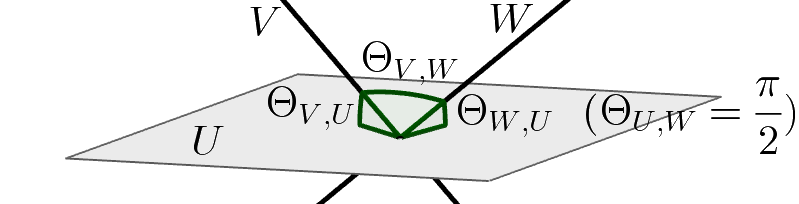}
	\caption{$\Theta_{V,W} \leq \Theta_{V,U} + \Theta_{U,W}$ due to $\Theta_{U,W} = \frac\pi2$, even if $\Theta_{V,W} > \Theta_{V,U} + \Theta_{W,U}$}
	\label{fig:oriented triangle inequality}
\end{figure}

Another useful property 
\cite{Mandolesi_Pythagorean} is that $\sum_U \cos^2 \Theta_{V,U} = 1$,
where $V \in \Gr_p(n)$ and the sum runs over all coordinate $p$-subspaces $U$ (see fn.\ \ref{ft:coord subspace} on p.\ \pageref{ft:coord subspace}) of an orthogonal basis of $\F^n$.
By \eqref{eq:Theta pi}, it is the same as 
\begin{equation}\label{eq:sum pi2}
	\sum_U \pi^2_{V,U} = 1 \text{ if } \F=\R, \quad \text{ or } \quad \sum_U \pi_{V,U} = 1 \text{ if } \F=\C,
\end{equation}
which corresponds to a volumetric Py\-thag\-o\-re\-an theorem: the (squared, if $\F=\R$) volume of a measurable subset $S \subset V_\R$ is the sum of (squared, if $\F=\R$) volumes of its orthogonal projections on the $U$'s.

\subsection{Asymmetric angle in complex spaces}\label{sc:Asymmetric angle in complex spaces}

Understanding $\Theta_{V,W}$ in the complex case requires some basic facts about Hermitian spaces \cite{Goldman1999}.
When needed, we write
$\F$-basis, $\Span_\F$, $\dim_\F$, etc.\ to specify the field $\F$.
Also, $\R$-orthonormal and $\perp$ are \wrt $\inner{\cdot,\cdot}_\R$, while $\C$-orthonormal and $\perp_\C$ are \wrt $\inner{\cdot,\cdot}$.
The \emph{complex structure} $J$ is the $\frac\pi2$ rotation in $(\C^n)_\R$ corresponding to multiplication by $\im$ in $\C^n$,
so \eqref{eq:Hermitian} becomes
\begin{equation}\label{eq:Hermitian J}
	\inner{v,w} = \inner{v,w}_\R + \im \inner{J v,w}_\R.
\end{equation}
For $W\in \Gr(\C^n)$, $P_W$ is $\C$-linear, so it commutes with $J$.

Let $(e_1,\ldots,e_p)$, $(f_1,\ldots,f_q)$ and $\theta_1 \leq\ldots\leq\theta_{\min\{p,q\}}$ be principal bases and angles of $V_{(p)},W_{(q)} \in \Gr(\C^n)$.
Then $V_\R$ and $W_\R$ have principal bases $(e_1, J e_1,\ldots,e_p, J e_p)$ and $(f_1, J f_1,\ldots,f_q, J f_q)$, and the same principal angles twice repeated, 
as, by \eqref{eq:Hermitian J},
$\inner{J e_i, f_j}_\R = -\inner{e_i,J f_j}_\R = \operatorname{Im} \inner{e_i, f_j} = 0$
and $\inner{J e_i,J f_j}_\R = \inner{e_i,f_j}_\R = \operatorname{Re} \inner{e_i,f_j} = \delta_{ij} \cos \theta_i$, 
so that $\theta_{Je_i,Jf_i} = \theta_{e_i,f_i} = \theta_i$.

This duplication of principal angles explains the square in \eqref{eq:Theta pi} for $\F=\C$:
\eqref{eq:def Theta pi} defines $\pi_{V,W}$ via volumes in $V_\R$, but while both $e_i$ and $J e_i$ contract by $\cos \theta_i$ if projected on $W_\R$,
\eqref{eq:Theta thetai} has only one $\cos \theta_i$.

On the other hand, \eqref{eq:Theta pi} and $\pi_{V,W} = \pi_{V_\R,W_\R}$ imply $\Theta_{V,W} \neq \Theta_{V_\R,W_\R}$, which may seem weird, as metrically $\C^n \cong \R^{2n}$. 
But $\Gr(\C^n) \neq \Gr(\R^{2n})$, so it is actually natural that $\dFS(V,W) \neq \dFS(V_\R,W_\R)$.
One can also understand these angles as different ways to convey the same projection factor (see \Cref{ex:Theta underlying} below).


The complex case admits yet another interpretation, as follows.

\begin{definition}\label{df:tot real}
	Let $U\in \Gr((\C^n)_\R)$ and $V,W \in \Gr(\C^n)$. We say $U$ is:
	\begin{itemize}
		\item \emph{totally real}%
		\footnote{We use the term as in \cite{Goldman1999}. Some authors use it to mean just $J(U) \cap U= 0$.} 
		if $J(U) \perp U$; 
		\item \emph{totally real \wrt $W_\R$} if $U$ and $P_W (U)$ are totally real; 
		\item a \emph{maximal} such subspace of $V_\R$ if, moreover, $U \orthsum J(U) = V_\R$.%
			\SELF{Contraexemplo: $(e_1,e_2)$ of $V$, $(f_1,f_2)$ of $W$, $\theta_1=\frac\pi4$, $\theta_2=\frac\pi3$, $U=\Span_\R\{e_1+\im e_2,\im e_1+e_2\}$ totally real, $P_W(U)=\Span_\R\{f_1+\im \sqrt{2} f_2, \im f_1+\sqrt{2} f_2\}$ is not}	
	\end{itemize}  
\end{definition}

By \eqref{eq:Hermitian J},
$U$ is totally real $\Leftrightarrow$ $\inner{\cdot,\cdot} = \inner{\cdot,\cdot}_\R$ on $U$.
And, for $V \in \Gr(\C^n)$ and $U \in \Gr(V_\R)$, the following are equivalent \cite{Goldman1999}: 
	\CITE{p.39, shows $(i) \Leftrightarrow (ii) \Rightarrow (iv)$, others are immediate}
\begin{enumerate}[(i)]
	\item $U$ is a maximal totally real subspace of $V_\R$;
	\item $U$ is totally real and $\dim_\R U = \dim_\C V$;
	\item $U$ is $\R$-spanned by a $\C$-orthonormal
		\SELF{Precisa. $u=(1,0)$ e $v=(\im,1)$ é base de $\C^2$, mas $Jv \not \perp u$}
	$\C$-basis of $V$;
	\item Any $\R$-orthonormal $\R$-basis of $U$ is a $\C$-orthonormal $\C$-basis of $V$.
\end{enumerate}

With total reality \wrt $W_\R$, we have:

\begin{proposition}
	$U$ is maximal totally real subspace of $V_\R$ \wrt $W_\R \Leftrightarrow U$ is $\R$-spanned by a principal $\C$-basis of $V$ \wrt $W$.
\end{proposition}
\begin{proof}
	($\Rightarrow$)  
	$P_W (U)$ is a maximal totally real subspace of $(P_W (V))_\R = P_W(U \orthsum J(U)) = P_W(U) \orthsum J(P_W(U))$.
		\SELF{Sem total reality \wrt $W_\R$, teria $(P_W V)_\R = (P_W U) + J(P_W(U))$, mas pode haver interseção, e até $(P_W V)_\R = (P_W U) = J(P_W(U))$}
	If $P_W(U) = 0$ then $V \perp_\C W$,
	so any $\R$-orthonormal $\R$-basis of $U$ is a principal $\C$-basis of $V$ \wrt $W$.
	
	Let $\beta = (e_1,\ldots,e_p)$ and $(f_1,\ldots,f_r)$ be principal $\R$-bases of  $U$ and $P_W(U) \neq 0$.
	We have
	$\inner{J e_i,f_j}_\R = \inner{P_W J e_i,f_j}_\R = \inner{J P_W e_i,f_j}_\R = \|P_W e_i\| \inner{J f_i,f_j}_\R = 0$,
	and so, by \eqref{eq:Hermitian J},
	$\inner{e_i,f_j} = \inner{e_i,f_j}_\R = \delta_{ij} \cos \theta_i$.
	Thus $\beta$ is a principal $\C$-basis of $V$ \wrt $P_W(V)$, hence also \wrt $W$.

	($\Leftarrow$)
	Given principal $\C$-bases $\beta_V = (e_1,\ldots,e_p)$ of $V$ and $\beta_W$ of $W$, let $U = \Span_\R \beta_V$.
	As $\beta_V$ is orthonormal \wrt \eqref{eq:Hermitian J}, $\inner{Je_i,e_j}_\R = 0$, hence $J(U) \perp U$
	and $V_\R = \Span_\R \{e_1,\ldots,e_p,Je_1,\ldots,Je_p\} = U \orthsum J(U)$.
	As $P_W(V)$ is 0 or $\C$-spanned by some $\beta' \subset \beta_W$, 
	$P_W(U)$ is 0 or $\R$-spanned by $\beta'$,
	and again \eqref{eq:Hermitian J} implies $J(P_W(U)) \perp P_W(U)$.
\end{proof}

\begin{proposition}\label{pr:pi max tot real}
	If $U$ is maximal totally real subspace of $V_\R$ \wrt $W_\R$
	then $\pi_{V,W} = \pi^2_{U,W_\R}$,\, $\Theta_{V,W} = \Theta_{U,W_\R}$, and $\cos \Theta_{V,W} = \pi_{U,W_\R}$.
\end{proposition}
\begin{proof}
	Let $p=\dim_\C V = \dim_\R U$.
	For $S \subset U$ with $\vol_p(S) \neq 0$,
	\begin{align*}
		\pi_{V,W} &= \frac{\vol_{2p}(P_W(S \times J(S)))}{\vol_{2p}(S \times J(S))}
		= \frac{\vol_{2p}(P_W(S) \times J(P_W(S)))}{\vol_{2p}(S \times J(S))} \\
		&= \frac{\vol_p(P_W(S)) \cdot \vol_p(J(P_W(S)))}{\vol_p(S) \cdot \vol_p(J(S))}
		= \frac{\vol_{p}(P_W(S))^2}{\vol_p(S)^2}
		= \pi_{U,W_\R}^2. 
	\end{align*}
	The other formulas follow from \eqref{eq:Theta pi}.
\end{proof}

So, $\cos \Theta_{V,W}$ also describes how volumes from maximal totally real subspaces of $V_\R$ \wrt $W_\R$ contract if orthogonally projected to $W_\R$.

Another way to see why $\Theta_{V,W} = \Theta_{U,W_\R}$ is to note that $U$ and $W_\R$ have the same principal angles as $V$ and $W$,
and if $\dim_\C V > \dim_\C W$ then $\Theta_{V,W}=\frac\pi2 = \Theta_{U,W_\R}$ by \eqref{eq:Theta dim},
	\SELF{even if $\dim U \leq \dim W_\R$} 
as $P_W(U) \orthsum J(P_W(U)) \subset W_\R$ 
and so $\dim_\R P_W(U) \leq \frac{\dim_\R W_\R}{2} = \dim_\C W < \dim_\C V = \dim_\R U$.

\subsection{Computing asymmetric angles}\label{sc:Computing asymmetric angles}

The asymmetric angle is linked to the various products of Grassmann and Clifford geometric algebras \cite{Mandolesi_Products,Mandolesi_Contractions}, which give formulas to compute it and useful properties \cite{Mandolesi_Grassmann,Mandolesi_Trigonometry}.
In particular, if $V_{(p)}=[A]$ and $W_{(q)} = [B]$ for nonzero blades $A, B \in \bigwedge \F^n$ then \cite{Mandolesi_Products}
\begin{subequations}\label{eq:norm prods}
\begin{align}
	\cos \Theta_{V,W} &= \frac{\|A\lcontr B\|}{\|A\|\|B\|} \quad \left(= \frac{|\inner{A,B}|}{\|A\|\|B\|} \text{ if } p=q\right), \label{eq:norm contr} \\ 
	\shortintertext{and}
	\cos \Theta_{V,W^\perp} &= \frac{\|A\wedge B\|}{\|A\|\|B\|} = \prod_{i=1}^{\min\{p,q\}} \sin \theta_i, \label{eq:norm wedge}
\end{align}
\end{subequations}
where the $\theta_i$'s are the principal angles of $V$ and $W$.

We note \cite{Mandolesi_Grassmann} that $\Theta_{V,W^\perp} = \Theta_{W,V^\perp}$, but in general $\Theta_{V,W^\perp}$ does not equal $\Theta_{V^\perp,W}$ or $\frac\pi2 - \Theta_{V,W}$.
In \cite{Mandolesi_Contractions}, $\Theta_{V^\perp,W}$ is related to Grassmann's regressive product, and in \cite{Mandolesi_Trigonometry} we give a matrix formula for it.

The matrix versions of \eqref{eq:norm prods} are a prime example of how matrices can be worse than Grassmann algebra for working with subspaces:

\begin{proposition}\label{pr:formula any base dimension}
	Given bases $(v_1,\ldots,v_p)$ of $V$ and $(w_1,\ldots,w_q)$ of $W$, let 
	$\mathbf{A}=\big(\inner{v_i,v_j}\big)_{p \times p}$, 
	$\mathbf{B} =\big(\inner{w_i,w_j}\big)_{q \times q}$ and
	$\mathbf{C} =\big(\inner{w_i,v_j}\big)_{q \times p}$.
	Then
	\begin{subequations}\label{eq:formula Theta matrices}
		\begin{align}
			\cos \Theta_{V,W} &= \frac{\sqrt{\det(\mathbf{C}^\dagger \mathbf{B}^{-1}\mathbf{C})}}{\sqrt{\det \mathbf{A}}} \quad \left(= \frac{|\det \mathbf{C}\,|}{\sqrt{\det \mathbf{A} \cdot\det \mathbf{B}}} \text{ if $p=q$}\right), \label{eq:formula Theta any base any dimension} \\[3pt]
			\shortintertext{and}
			\cos \Theta_{V,W^\perp} &= \frac{\sqrt{\det(\AA - \CC^\dagger \BB^{-1} \CC)}}{\sqrt{\det \AA}} = \frac{\sqrt{\det(\BB - \CC \AA^{-1} \CC^\dagger)}}{\sqrt{\det \BB}}. \label{eq:formula Theta perp} 
		\end{align}
	\end{subequations} 
\end{proposition}
\begin{proof}
	If $A=v_1\wedge\cdots \wedge v_p$ and $B=w_1\wedge\cdots\wedge w_q$
	then $\|A\| = \sqrt{\det \mathbf{A}}$ and $\|B\| = \sqrt{\det \mathbf{B}}$, 
	so \eqref{eq:formula Theta any base any dimension} follows from \eqref{eq:norm contr} and \eqref{eq:norm contr matrices},
	and \eqref{eq:formula Theta perp} from \eqref{eq:norm wedge} and Schur's determinant identity
		\OMIT{Schur (Brualdi1983): $\begin{vsmallmatrix}
			A & B \\ 
			C & D
		\end{vsmallmatrix} =$ \\ $\det(D)\det(A-BD^{-1}C)=$ \\ $\det(A)\det(D-CA^{-1}B)$}
	applied to $\|A\wedge B\|^2 = \begin{vsmallmatrix}
		\AA & \CC^\dagger \\ 
		\CC & \BB
	\end{vsmallmatrix}$.
\end{proof}

With orthonormal bases, the formulas become simpler.

\begin{corollary}\label{pr:determinant projection}
	Let $\mathbf{P}$ be a matrix for the orthogonal projection $V\rightarrow W$ in  orthonormal bases of $V$ and $W$. Then
	\begin{subequations}\label{eq:formulas orthonormal bases}
		\begin{align}
			\cos \Theta_{V,W} &= \sqrt{\det(\mathbf{P}^\dagger \mathbf{P})} \quad \left(=|\det \mathbf{P}| \text{ if } p=q \right), \label{eq:formula Theta orthonormal bases} \\
			\shortintertext{and}
			\cos \Theta_{V,W^\perp} &= \sqrt{\det(\mathds{1}_{p\times p} - \mathbf{P}^\dagger \mathbf{P})} = \sqrt{\det(\mathds{1}_{q\times q} - \mathbf{P}\mathbf{P}^\dagger)}. \label{eq:formula Theta perp orthonormal bases}
		\end{align}
	\end{subequations}
\end{corollary}
\begin{proof}
	In such bases, $\mathbf{A} = \mathds{1}_{p\times p}$, $\mathbf{B} = \mathds{1}_{q\times q}$ and $\mathbf{C} = \mathbf{P}$.
\end{proof}

In principal bases $\mathbf{P}$ is a diagonal matrix with the $\cos\theta_i$'s and extra lines or columns with $0$'s,%
	\OMIT{If $p>q$ the diagonal of $\bar{\mathbf{P}}^T \mathbf{P}$ has $0$'s}
so \eqref{eq:formulas orthonormal bases} also follows from \eqref{eq:Theta thetai} and \eqref{eq:norm wedge}.

\begin{example}\label{ex:contraction}
	Let $\{u_1,\ldots,u_5\}$ be the canonical basis of $\R^5$, 
	$V_{(2)}$ be spanned by $v_1=2u_1-u_2$ and $v_2 = 2u_1+u_3$,
	and $W_{(3)}$ by $w_1=u_2+u_5$, $w_2=u_3-u_4$ and $w_3=u_4$. 
	Then 
	\[ \mathbf{A}=\begin{psmallmatrix}
		5 & 4 \\ 4 & 5
	\end{psmallmatrix}, \quad
	\mathbf{B}=\begin{psmallmatrix}
		2 & 0 & 0 \\ 0 & 2 & -1 \\ 0 & -1 & 1
	\end{psmallmatrix}, \quad
	\mathbf{C}=\begin{psmallmatrix}
		-1 & 0 \\ 0 & 1 \\ 0 & 0
	\end{psmallmatrix}, \]
	and \eqref{eq:formula Theta matrices} gives
	$\Theta_{V,W}= \cos^{-1}\frac{\sqrt{2}}{6} \cong 76.4^\circ$ and 
	$\Theta_{V,W^\perp}= \cos^{-1}\frac{\sqrt{2}}{3} \cong 61.9^\circ$,
	so areas in $V$ contract by $\frac{\sqrt{2}}{6}$ (\resp $\frac{\sqrt{2}}{3}$) if orthogonally projected on $W$ (\resp $W^\perp$).	
	With 
	\[ \mathbf{A}=\begin{psmallmatrix}
		2 & 0 & 0 \\ 0 & 2 & -1 \\ 0 & -1 & 1
	\end{psmallmatrix}, \quad
	\mathbf{B}=\begin{psmallmatrix}
		5 & 4 \\ 4 & 5
	\end{psmallmatrix}, \quad
	\mathbf{C}=\begin{psmallmatrix}
		-1 & 0 & 0 \\ 0 & 1 & 0
	\end{psmallmatrix}, \]
	we find, as expected, $\Theta_{W,V}= 90^\circ$ and $\Theta_{W,V^\perp} = \Theta_{V,W^\perp}$, so volumes in $W$ vanish (\resp contract by $\frac{\sqrt{2}}{3}$) if orthogonally projected on $V$ (\resp $V^\perp$).	
	
	Grassmann algebra allows simpler calculations.
	Writing $v_{12} = v_1 \wedge v_2$, for example, we have
	$V=[A]$ for $A = v_{12} = 2u_{12}+2u_{13}-u_{23}$, by \eqref{eq:wedge},
	and
	$W=[B]$ for $B= w_{123} = u_{234}+u_{345}$.
	One easily finds
	$\|A\| = 3$ and $\|B\| = \sqrt{2}$ by orthonormality of $u_{ij}$'s and $u_{ijk}$'s,
	$A\lcontr B = -u_4$ by \eqref{eq:contraction orthonormal}, $B \lcontr A = 0$, and $A \wedge B = 2u_{12345}$.
	Then \eqref{eq:norm prods} gives the same results.
\end{example}

\begin{example}\label{ex:Theta underlying}
	In $\C^3$, let $V=\Span\{v\}$ and $W=\Span\{w_1,w_2\}$ for $v=(1,0,\im)$, $w_1=(1,0,0)$ and $w_2=(\im,1,0)$. 
	We  find  $\Theta_{V,W}= 45^\circ$ using 
	\[\mathbf{A}=(2), \quad
	\mathbf{B}=\begin{psmallmatrix}
		1 & \im \\ -\im & 2
	\end{psmallmatrix}, \quad
	\mathbf{C}=\begin{psmallmatrix}
		1 \\ -\im
	\end{psmallmatrix}. \] 
	Since $V_\R=\Span_\R\{v,Jv\}$ and $W_\R=\Span_\R\{w_1,Jw_1,w_2,Jw_2\}$ with
	\begin{align*}
		v&=(1,0,0,0,0,1), & w_1&=(1,0,0,0,0,0), & w_2&=(0,1,1,0,0,0), \\
		Jv&=(0,1,0,0,-1,0), & Jw_1&=(0,1,0,0,0,0), & Jw_2&=(-1,0,0,1,0,0),
	\end{align*}
	we find $\Theta_{V_\R,W_\R}= 60^\circ$ using
	\[ \mathbf{A}=\begin{psmallmatrix}
		2 & 0 \\
		0 & 2
	\end{psmallmatrix}, \quad
	\mathbf{B}=\begin{psmallmatrix}
		1 & 0 & 0 & -1 \\
		0 & 1 & 1 & 0 \\
		0 & 1 & 2 & 0 \\
		-1 & 0 & 0 & 2
	\end{psmallmatrix}, \quad
	\mathbf{C}=\begin{psmallmatrix}
		1 & 0 \\
		0 & 1 \\
		0 & 1 \\
		-1 & 0
	\end{psmallmatrix}. \]
	Though $\Theta_{V_\R,W_\R} \neq \Theta_{V,W}$, both tell us, via \eqref{eq:Theta pi}, that areas in $V_\R$ contract by $\pi_{V_\R,W_\R} = \pi_{V,W} = \frac12$ if orthogonally projected to $W_\R$.
\end{example}

\end{document}